\documentclass[10pt,letterpaper]{amsart} 

\usepackage{color}
\usepackage[letterpaper, margin=0.8in]{geometry}
\usepackage{diagbox}
\usepackage{siunitx}
\usepackage{subcaption}
\usepackage{listings}
\usepackage{amsmath, amsthm, amssymb, wasysym, verbatim, bbm, color, graphics}
\usepackage{url}
\usepackage{mathtools}
\usepackage[colorlinks,linkcolor=blue]{hyperref}
\usepackage{romannum}
\usepackage{xfrac}
\usepackage{float}
\usepackage[title]{appendix}
\usepackage{todonotes}
\usepackage{algorithm}
\usepackage{algpseudocode}

\newcommand{\phinew}{\phi^{n+1}}
\newcommand{\phiold}{\phi^n}
\newcommand{\Dtphi}{\frac{\phinew-\phiold}{\Delta t}}
\newcommand{\munew}{\mu^{n+1}}
\newcommand{\rnew}{r^{n+1}}
\newcommand{\rold}{r^n}
\newcommand{\Pold}{P^n}

\newcommand{\EnewIEF}{\hat{E}^{n+1}}
\newcommand{\EoldIEF}{\hat{E}^n}
\newcommand{\gnew}{g^{n+1}}
\newcommand{\gold}{g^n}

\newtheorem{lemma}{Lemma}[section]

\newtheorem{remark}[lemma]{Remark}

\newtheorem*{maintheorem*}{Main Theorem}
\theoremstyle{definition}{}

\theoremstyle{note}{
\newtheorem*{claim*}{Claim}}

\newtheorem{thm}[lemma]{Theorem}

\newtheorem{rem}[lemma]{Remark}

\newtheorem{cor}[lemma]{Corollary}

\allowdisplaybreaks

\numberwithin{equation}{section}

\title[Invariant Energy Convexication and Invariant Energy Functionalization]{Two novel numerical methods for gradient flows: generalizations of the Invariant Energy Quadratization method}
\date{\today}

\author[Y. Yue]{Yukun Yue}
\address[Yukun Yue]{\newline Department of Mathematical Sciences \newline Carnegie Mellon University \newline 5000 Forbes Avenue, Pittsburgh, PA 15213, USA.}
\email[]{yukuny@andrew.cmu.edu}
\thanks{Y.Y.'s work was supported in part by NSF award DMS-1912854}
\begin{document}
 \pagenumbering{arabic}
\maketitle
\begin{abstract}
  In this paper, we conduct an in-depth investigation of the structural intricacies inherent to the Invariant Energy Quadratization (IEQ) method as applied to gradient flows, and we dissect the mechanisms that enable this method to uphold linearity and the conservation of energy simultaneously. Building upon this foundation, we propose two methods: Invariant Energy Convexification and Invariant Energy Functionalization. These approaches can be perceived as natural extensions of the IEQ method. Employing our novel approaches, we reformulate the system connected to gradient flow, construct a semi-discretized numerical scheme, and obtain a commensurate modified energy dissipation law for both proposed methods. Finally, to underscore their practical utility, we provide numerical evidence demonstrating these methods' accuracy, stability, and effectiveness when applied to both Allen-Cahn and Cahn-Hilliard equations. 
\end{abstract}
\section{Introduction}

Gradient flows are a class of partial differential equations (PDEs) that arise in various scientific and engineering fields, such as fluid dynamics, materials science, and optimization, as demonstrated by \cite{AndersonMcFaddenWheeler, Cahn-Hilliard, christlieb2014high, ding2022overparameterization, ericksen1976equilibrium, LiLiu2004, liu2000approximation, vardoulakis1991gradient}. These equations model the evolution of a given quantity under the influence of a driving force, which is derived from an energy functional. The study of gradient flows has attracted considerable attention in recent years due to their wide applicability and inherent mathematical structures. In this paper, we present two innovative numerical methods for gradient flows that extend the Invariant Energy Quadratization (IEQ) method, which is recently introduced. These methods offer a generalization of the pre-existing method, leveraging the mathematical structures possessed by the method while preserving the favorable properties, such as the energy-stable property and the ability to construct efficient linear schemes to solve the problems. 

To illustrate, let's consider a free energy functional $E(\phi)=\int_\Omega \rho\left(\phi(x)\right)\,dx$, where $\rho\left(\phi(x)\right)=\left[\frac{1}{2}\lvert\nabla\phi\rvert^2+F(\phi)\right]$ serves as the energy density function of $E(\phi)$. The corresponding gradient flow can then be formulated as
\begin{subequations}
\label{eq:gradient_flow_system}
\begin{equation}\label{eq:time_derivative_gradient_flow}
\frac{\partial \phi}{\partial t}=\mathcal{G} \mu,
\end{equation}
\begin{equation}\label{eq:H^-1_gradient}
\mu = \frac{\delta E}{\delta\phi}=-\Delta \phi+\frac{\delta F}{\delta \phi}:=-\Delta \phi+f(\phi),
\end{equation}
\end{subequations}
with $\phi(0)=\phi_0$ to be its initial condition. In this context, $\mathcal{G}=I$ signifies the identity operator when we are considering a gradient flow in $L^2$, and $\mathcal{G}=\Delta$, which denotes the Laplacian, when we are investigating a gradient flow in $H^{-1}$ \cite{ambrosio2005gradient, mielke2011weighted}.

 Gradient flows have a rich mathematical structure and are related to several important concepts in mathematics, such as the Wasserstein distance, optimal transport, and convex optimization. There has been a surge of interest in unearthing its mathematical properties, with notable contributions made by \cite{chizat2018global, cortes2006finite, maas2011gradient, santambrogio2017euclidean}. Concurrently, researchers have been striving to establish various numerical methods to resolve gradient flow problems. These efforts have given rise to techniques such as the convex-splitting method \cite{baskaran2013convergence,glasner2016improving, shen2012second, shin2017unconditionally}, the Invariant Energy Quadratization (IEQ) \cite{bilbao2023explicit, guillen2019unconditionally,han2020second,  yang2017efficient, yang2017numerical,zhao2017numerical}, and the Scalar Auxiliary Variable (SAV) methods \cite{shen2018convergence, shen2018scalar, shen2019new,shen2020ieq}. These methods allow for the development of efficient schemes for solving gradient flows, and hold the energy-stable property, assuring the numerical solutions maintain certain physical and mathematical properties of the continuous problem.

The IEQ method, particularly, has gained popularity as a numerical method in recent years. It introduces an auxiliary variable whose quadratization equals the original energy density function when the energy density function has a lower bound. It is based on the idea of conserving certain invariants of the continuous problem in the discrete setting. As an illustration, one could define
\begin{equation*}
   r(\phi) = \sqrt{F(\phi)+A_1},
\end{equation*}
with $A_1$ as a constant ensuring $r$ to be well-defined. Provided that $F$ is bounded from below, an appropriate constant $A_1$ can always be found. Consequently, system \eqref{eq:gradient_flow_system} can be reformulated as
\begin{subequations}
    \label{eq:gradient_flow_reformulated_system}
    \begin{equation}
        \label{eq:phi_t_reformulated}
        \frac{\partial \phi}{\partial t} = \mathcal{G}\mu,
    \end{equation}
    \begin{equation}
        \label{eq:mu_reformulate}
        \mu = -\Delta \phi + 2rP,
    \end{equation}
    \begin{equation}
        \label{eq: r_t_reformulate}
        r_t = P\phi_t,
    \end{equation}
    \begin{equation}
        P=\frac{\frac{\delta F}{\delta \phi}}{2\sqrt{F(\phi)+A_1}}
    \end{equation}
\end{subequations}
The chain rule can be conveniently applied to the auxiliary variable $r$ to validate this reformulation. This approach has successfully facilitated the construction of linear numerical schemes for various gradient-flow type problems, including the Cahn-Hilliard equation \cite{yang2020convergence, yang2017numerical2}, the Ericksen–Leslie model \cite{chen2017second} and Beris-Edwards model for liquid crystals \cite{gudibanda2022convergence,yue2023convergence,zhao2017novel}, and the sine-Gordon equation \cite{fu2021linearly,jiang2019linearly}. For instance, a linear numerical scheme can be constructed by treating $r$ implicitly and $P$ explicitly. Specifically, we have:
\begin{subequations}
    \label{eq:IEQ_scheme}
    \begin{equation}
        \label{eq:IEQ_phi_t}
        \frac{\phinew-\phiold}{\Delta t} = \mathcal{G}\munew,
    \end{equation}
    \begin{equation}
        \label{eq:IEQ_mu_new}
        \munew=-\Delta\phinew+2\rnew\Pold,
    \end{equation}
    \begin{equation}
        \label{eq:IEQ_r_t}
        \rnew-\rold=\Pold:(\phinew-\phiold),
    \end{equation}
    \begin{equation}
        \label{eq:IEQ_P}
        \Pold=\frac{\frac{\delta F}{\delta \phi}(\phiold)}{2\sqrt{F(\phiold)+A_1}}.
    \end{equation}
\end{subequations}
The resulting numerical scheme is characterized by a modified energy dissipation law (see, for instance, \cite{zhao2017novel}). One of the notable strengths of the IEQ method is its energy-stable property, which ensures a monotonic decrease of the discrete energy functional along the trajectory of the numerical solution, akin to the behavior observed in the continuous problem. This property plays a vital role in the long-term behavior of the numerical solution and the preservation of inherent physical and mathematical structures.

The power of the IEQ method fundamentally lies in its innovative decomposition of the nonlinear term into a product of two functions, specifically emphasizing which part should be treated implicitly. This quadratization formulation aligns seamlessly with a usual form of the energy term encountered in physical processes, as evidenced by numerous scientific studies. In addition, quadratization formulation induces a linear term to appear, leading to a linear scheme. However, the energy term in various natural scientific systems may not be confined to a quadratic form. This observation prompts us to question the possibility of other formulations capable of preserving the advantageous properties of the IEQ method. In other words, we are interested in discovering other indigenous decomposition of the nonlinear term to construct energy-stable, linear numerical schemes.

In this paper, we propose two novel strategies for extending the Invariant Energy Quadratization (IEQ) method, namely the Invariant Energy Convexification (IEC) and the Invariant Energy Functionalization (IEF) approaches. Both methods are based on introducing an auxiliary variable originating from using different functions to replace the original energy density function. Moreover, they both induce a modified energy dissipation law within the reformulation of the original gradient-flow type system. Importantly, these two methods retain the advantageous ability to induce efficient linear numerical schemes that rival the efficacy of the IEQ method. A noteworthy contribution of these proposed approaches is their potential to accommodate various functions as auxiliary variables during the transformation process. This flexibility suggests the possibility of identifying an optimal form for implementing the method tailored to the specifics of a given physical system.

The organization of the remainder of this paper is as follows: In Section \ref{sec:IEC}, we introduce the Invariant Energy Convexification (IEC) method alongside a new reformulation of the system as described in equation \eqref{eq:gradient_flow_system}. We shall proceed to derive a semi-discrete numerical scheme and verify the stability of these numerical schemes under suitably modified energy. Likewise, in Section \ref{sec:IEF}, we outline the Invariant Energy Functionalization (IEF) method, present a corresponding numerical scheme, and confirm its energy stability. In Section \ref{sec:Numerics}, we substantiate the accuracy and efficacy of our proposed methods through a series of numerical experiments, providing compelling evidence of their utility.


\section{The invariant energy Convexification}\label{sec:IEC}

In this section, we introduce our first numerical method that expands upon the core principles of the IEQ method. As it has been stated in the introduction, the IEQ method has established itself as a potent and effective technique for addressing a broad class of problems. At the heart of the IEQ method is the introduction of an auxiliary variable representing the square root of the energy density function. The incorporation of this auxiliary variable enables the method to more effectively manage the non-linear terms that emerge in various problems. Specifically, the IEQ method handles a portion of the resulting non-linear term implicitly while maintaining the remainder in an explicit form. This blend of implicit and explicit treatment is a distinguishing feature of the IEQ method and contributes to its efficacy in constructing linear numerical schemes. A critical question arising from this approach is whether the decomposition of the non-linear term in the IEQ method is unique or if it can be further generalized. This will be discussed in the following.

\subsection{Reformulation with L-smooth  convex function}

To find a more general class of functions that share the advantages of the IEQ method, we need to figure out the nature of its effectiveness first. We notice that the energy-conserving property mainly relies on the convexity of the auxiliary function used to replace the original energy density function. Meanwhile, the linearity relies on the fact that the derivative of a quadratic function is a linear function. Therefore, it enlightens us to find a class of convex functions to introduce the auxiliary variable and look for a possible linear approximation of its derivative to construct the numerical scheme. To investigate this question, we propose a new formulation designed to imitate the IEQ formulation's structure and broaden its applicability while preserving its essential advantages. As a result, we discover that quadratization is not the sole viable option for variable transformation; a particular class of convex functions can also pave the way to an energy-stable linear scheme.

Specifically, let us presume the energy density function $F$ is bounded from below. We will take $c:\mathbb{R}\to\mathbb{R}$ as a smooth convex function that is monotonically increasing on a connected set $K$, and $\mathbb{R}^+$ is contained $c(K)$. We further assume that $c$ is $L-$smooth \cite{lu2018relatively}, which means that for any $x,y$ within domain of function $c$, a constant $0<L<\infty$ exists such that

\begin{equation*}
|\nabla c(x)-\nabla c(y)|\leq L\lvert x-y\rvert.
\end{equation*}
This condition implies that
\begin{equation}
\label{eq:L-smooth_property}
c(y)\leq c(x)+c'(x)(y-x)+\frac{L}{2}\lvert y-x\rvert^2.
\end{equation}
Now, consider $r:[0,T]\times\Omega\to\mathbb{R}$ to be a function solving equation $c\left(r(t,x)\right)=F\left(\phi(t,x)\right)+A_1$ for every $(t,x)\in [0,T]\times\Omega$ where $A_1$ is the constant to ensure the non-negativity of $F$. Given the invertibility of $c$ on $\mathbb{R}^+$, this definition is indeed valid. Differentiating both sides of the equation results in
\begin{equation*}
c'(r)\,\frac{\delta r}{\delta \phi}=\frac{\delta F(\phi)}{\delta \phi}=f(\phi).
\end{equation*}
Therefore, systems \eqref{eq:gradient_flow_system} will be rewritten as
\begin{subequations}
    \label{eq:IEC_formulation}
    \begin{equation}
    \label{eq:phi_update}
        \phi_t=\mathcal{G}\mu,
    \end{equation}
    \begin{equation}\label{eq:mu_update}
        \mu=-\Delta\phi +c'(r)\frac{\delta r}{\delta \phi},
    \end{equation}
    \begin{equation}\label{eq:r_update}
        r_t =\frac{\delta r}{\delta \phi }\phi_t.
    \end{equation}
\end{subequations}
Taking inner product above with $\mu, \phi_t$ and $c'(r)$ respectively, one can easily obtain a modified energy dissipation law as:
\begin{equation}
    \label{eq:modified_energy_dissipation_law_continuous}
    \frac{d}{dt}\left[\frac{1}{2}\|\nabla\phi\|^2+\int_\Omega c(r)\right]=(G\mu, \mu)\leq 0,
\end{equation}
for $\mathcal{G}=-I$ and $\mathcal{G}=\Delta$ where $\|\cdot\|$ is the $L^2$ norm. If $c'(r)$ is a linear function with respect to $r$, then we can construct a linear numerical scheme simply by treating $c'(r)$ implicitly and $\frac{\delta r}{\delta \phi}$, which is how IEQ works. However, when $c'(r)$ is not linear, treating $c'(r)$ implicitly will result in a nonlinear equation to be solved at each step which generally loses a very strong advantage of this approach. In addition, we want to point out that the convexity is not crucial in the continuous case as reformulation \eqref{eq:IEC_formulation} and the corresponding modified energy dissipation law actually can work by simply choosing $c$ as a smooth function, with the convexity condition being dropped. However, convexity will play an important role in the way of constructing linear energy-stable schemes in the discrete case. It will be the main problem that we want to solve in the next subsection.

\subsection{Numerical scheme}

Our newfound approach provides the cornerstone for devising a linear numerical scheme to resolve \eqref{eq:IEC_formulation}. We merely need a linear approximation for $c'(r)$ in the discretized scenario. With this in mind, we propose the following first-order semi-discrete scheme:

\begin{subequations}
\label{eq:IEC_scheme}
    \begin{equation}
    \label{eq:gf_numer}
        \Dtphi=\mathcal{G} \munew,
    \end{equation}
    \begin{equation}
    \label{eq:munew}
        \munew = -\Delta \phinew+\left[c'(\rold)+\alpha L(\rnew-\rold)\right]\Pold,
    \end{equation}
    \begin{equation}
    \label{eq:drdt}
        \rnew-\rold=\Pold(\phinew-\phiold),
    \end{equation}
\end{subequations}
where $P^n=\frac{\delta r}{\delta Q}(\phiold)$ and $\alpha\geq\frac{1}{2}$ is a chosen parameter. The elegance of this scheme lies in its efficient solvability at each time step. Specifically, given that $(\phiold, \rold)$ are known, substituting \eqref{eq:r_update} into \eqref{eq:mu_update} enables the replacement of $\rnew$ by the formula of $\phinew$, hence allowing the computation of $(\phinew, \rnew)$ by first finding $\phinew$ from
\begin{equation*}
    (I+\Delta t\, \mathcal{G}\Delta)\phinew-\alpha L\Delta t\,\mathcal{G} \left(\Pold\phinew\Pold\right) = \phiold+\Delta t\mathcal{G}\,\left[  c'(\rold)-\alpha L\Pold\phiold \right]\Pold,
\end{equation*}
and subsequently updating $\rnew$ by \eqref{eq:drdt}. Alternatively, the scheme can be executed by amalgamating $\munew$ in the solving process. This approach determines $(\phinew, \munew, \rnew)$ by resolving the linear system
\begin{equation}
\label{eq:IEC_scheme_implementation}
    \begin{pmatrix}
        \frac{1}{\Delta t}I & -\mathcal{G} & 0\\
        \Delta & I & -\alpha L \Tilde{P}^n\\
        -\Tilde{P}^n & 0 & I
    \end{pmatrix}
    \begin{pmatrix}
        \phinew\\ \munew \\ \rnew
    \end{pmatrix} = \begin{pmatrix}
         \frac{1}{\Delta t}\phiold\\
        c'(\rold)\Pold-\alpha L \rold\Pold\\
        \rold-P^n\phiold
    \end{pmatrix},
\end{equation}
at each time step, given that $(\phi^n, \rold)$ have been ascertained. Both implementations offer a linear pathway to updating the numerical results, thereby providing us a streamlined alternative to tackling nonlinear gradient-flow type problems, as opposed to implicitly addressing the nonlinear term and having to resolve a nonlinear equation at each juncture.
\begin{rem}
    In the second implementation approach above, we use $\Tilde{P}^n$ to denote a functional operator, a multiplication of $\Pold$ to $\rnew$ point-wisely. In the fully-discrete case, assume the space has been discretized by a $N_x\times N_y$ grid, we will treat $\phinew, \munew, \rnew$ as a $1\times (N_x*N_y)$ vector, respectively. And $\Tilde{P}^n$ will be taken as a $(N_x * N_y)\times(N_x * N_y)$ diagonal matrix by expanding the values of $\Pold$ at the $N_x*N_y$ grid points onto its diagonal elements. The same notation $\Tilde{\cdot}$ will be repeatedly used in the following to denote a point-wise multiplication operator. 
\end{rem}

Besides sustaining the linearity property, we can also immediately obtain the following modified discretized energy dissipation law from scheme \eqref{eq:IEC_scheme} .

\begin{thm}
\label{thm:IEC_energy_stability}
    The numerical scheme \eqref{eq:IEC_scheme} is energy-stable. Specifically, define 
    \begin{equation}\label{eq:modified_energy_IEC}
        E^n=\frac{1}{2}\|\nabla \phi^n\|^2+\int_\Omega c(r^n),
    \end{equation}then
    \begin{equation*}
        E^{n+1}-E^n\leq0.
    \end{equation*}
\end{thm}
\begin{proof}
    Taking inner product of \eqref{eq:phi_update} with $\munew\Delta t$, \eqref{eq:mu_update} with $\phinew-\phiold$, \eqref{eq:r_update} with $\left[c'(\rold)+\frac{L}{2}(\rnew-\rold)\right]$, and using property of $L$-smooth function \eqref{eq:L-smooth_property}, we obtain
    \begin{align*}
        (\mathcal{G}\munew, \munew)\Delta t&=(\phinew-\phiold, \munew)\\
        &=\left(\phinew-\phiold,  -\Delta \phinew+\left[c'(\rold)+\alpha L(\rnew-\rold)\right]\Pold\right)\\
        &=\frac{1}{2}\|\nabla \phinew\|^2-\frac{1}{2}\|\nabla \phiold\|^2+\frac{1}{2}\|\nabla\phinew-\nabla\phiold\|^2\\
        &\quad\,\,+\int_\Omega \left[c'(r^n)(\rnew-\rold)+\alpha L(\rnew-\rold)^2\right]\,dx\\
        &\geq \frac{1}{2}\|\nabla \phinew\|^2-\frac{1}{2}\|\nabla \phiold\|^2+\frac{1}{2}\|\nabla\phinew-\nabla\phiold\|^2+\int_\Omega \left[c(\rnew)-c(\rold)\right]\,dx\\
        &\geq E^{n+1}-E^n.
    \end{align*}
    Using the fact that $(\mathcal{G}\munew, \munew)\leq 0$ independent of choice of $\munew$, we have finished the proof.
\end{proof}
\begin{remark}
    From the deduction, we can see that L-smoothness of $c(x)$ plays a key role in keeping the energy dissipated rule. Specifically, if we modify $c(x)$ to be a concave function, then simply taking $L=0$ is enough to ensure the inequalities appeared in the proof to hold and it will also result in a linear energy-stable numerical scheme. 
\end{remark}

Following from this theorem, we can immediately obtain the following estimate holds for $\munew$:
\begin{cor}
\label{cor:mu_estimate_IEC}
    Assume $E^0$ is bounded. For fixed $N>0$, we have
    \begin{equation*}
       0\leq -\sum_{n=0}^N (\mathcal{G}\mu^n, \mu^n)\Delta t\leq -E^N+E^0\leq E^0,
    \end{equation*}
    and so it is bounded. Specifically, if $\mathcal{G}=\Delta$, we have 
    \begin{equation*}
       0\leq \sum_{n=0}^N \|\nabla\mu^n\|^2\Delta t\leq E^0,
    \end{equation*}
    if $\mathcal{G}=-I$,  we have 
    \begin{equation*}
       0\leq \sum_{n=0}^N \|\mu^n\|^2\Delta t\leq E^0.
    \end{equation*}
\end{cor}

\subsection{An example of the IEC scheme}

Having laid the groundwork for the standard formulation of the IEC approach and its corresponding numerical scheme's general form, we now face the necessity of an exact form of the chosen L-smooth convex function for concrete computational applications. As we conclude this section, we propose a viable candidate for such a function, demonstrating that it is indeed possible to construct an IEC scheme using a function that satisfies these criteria. Our goal is to identify an option that possesses the properties necessary to induce a valid IEC approach and is also conveniently calculable, thus facilitating our numerical implementation.

Before delving into that, it is worth mentioning that the IEQ approach can be considered a specific instance of the IEC formulation. Indeed, the quadratic function is an $L$-smooth convex function with $L=2$. Thus, by selecting $\alpha=1$ in \eqref{eq:IEC_scheme}, the resulting numerical scheme aligns perfectly with the scheme induced by the IEQ formulation. This observation serves to justify our perspective that the IEC method is a natural generalization of the IEQ method.

As an alternative to the quadratic function, we turn our attention to the Softplus function \cite{szandala2021review}. Defining $c(r) = \ln(1+e^r)$, we quickly find that the first-order derivative, $c'(r)=\frac{e^r}{1+e^r}$, and the second-order derivative, $c''(r)=-\frac{e^r}{(1+e^r)^2}$, are uniformly bounded. Consequently, $c(r) = \ln(1+e^r)$ establishes itself as a non-negative, monotonically increasing, L-smooth convex function with $L=\frac{1}{4}$. Taking $c(r)=F(\phi)+A_1$, this leads to
\begin{equation}
    \label{eq:r_softplus}
    r(\phi) = \ln (e^{F(\phi)+A_1}-1),
\end{equation}
and subsequently to $\Pold=P(\phiold)$ where
\begin{equation}
    \label{eq:P_softplus}
    P(\phi)=\frac{\delta r}{\delta \phi}(\phi)=\frac{e^{F(\phi)+A_1}}{e^{F(\phi)+A_1}-1}f(\phi).
\end{equation}
These equations provide an illustrative example of \eqref{eq:IEC_scheme}, demonstrating the applicability of the Softplus function in this context. We can conclude this procedure in the following Algorithm \ref{alg:IEC_softplus}.

\begin{algorithm}
\caption{IEC scheme induced by Softplus function} \label{alg:IEC_softplus}
\begin{algorithmic}[1] 
\Procedure{}{}
    \State Set the step size $\Delta t$
    \State Set the total number of iterations $N$
    \State Set the value for $\alpha$
    \State Set the initial value $\phi^0 = \phi_0$
    \State Compute $r^0 = r(\phi^0)$ using equation \eqref{eq:r_softplus}
    \State Compute $P^0 = P(\phi^0)$ using equation \eqref{eq:P_softplus}
    
    \For {$n = 0, 1, 2, ..., N-1$}
        \State Construct the matrix $A=\begin{pmatrix}
        \frac{1}{\Delta t}I & -\mathcal{G} & 0\\
        \Delta & I & -\frac{\alpha}{4}\Tilde{P}^n\\
        -\Tilde{P}^n & 0 & I
    \end{pmatrix}$ using $P^n$
        \State Construct the right-hand side $b=\begin{pmatrix}
         \frac{1}{\Delta t}\phiold\\
        \frac{e^{\rold}}{1+e^{\rold}}\Pold-\frac{\alpha}{4}\rold\Pold\\
        \rold-P^n\phiold
    \end{pmatrix}$
        \State Update $\phi^{n+1}$, $r^{n+1}$ by solving the linear system $A\begin{pmatrix}
            \phi^{n+1}\\\mu^{n+1}\\ r^{n+1}
        \end{pmatrix} = b$
        \State Update $P^{n+1} = P(\phi^{n+1})$ using equation \eqref{eq:P_softplus}
    \EndFor
\EndProcedure
\end{algorithmic}
\end{algorithm}

The flexibility in selecting the function $c$ is relatively wide as long as it conforms to the previously defined prerequisites that ensure a valid numerical scheme. In practice, this adjustment necessitates an alteration of the formulas for $r$ and $P$ and an adjustment of the expression for $c'(r)$ used in computing the right-hand side vector $b$ in the preceding algorithm. An important observation is that the second derivative of the Softplus function is uniformly bounded, thereby ensuring it is L-smooth uniformly across the full range of real numbers. Additionally, we may consider alternative convex functions like the exponential function $e^r$ and polynomials $r^k$ where $k\geq 3$, as potential replacements for the quadratic function, given that some other specific assumptions are met. Specifically, the exponential and polynomial functions both possess the local $L-$smoothness. Hence, if the range of $r$ is bounded, a sufficiently large $L$ can be identified to make these functions locally $L-$smooth and facilitate the implementation of corresponding numerical schemes. In Section \ref{sec:Numerics}, we will display some numerical evidence supporting these formulations' applicability.

\section{Invariant Energy Functionalization}\label{sec:IEF}

As previously established, the two main benefits of the IEQ method are its linearity and energy-dissipation preservation. The conservation of energy, particularly in a discrete case, is primarily ensured through the convexity of the auxiliary function, which replaces the original energy density function. The IEC formulation can be seen as an evolution of this concept, placing the preservation of convexity at its core.
In contrast, the linear approximation employed in \eqref{eq:IEC_scheme} serves as a mechanism to design a linear scheme following the assurance of convexity. It reveals a fundamental tenet of the IEC scheme: the priority of maintaining convexity over linearity. 

Nevertheless, it is the linearity that brings efficiency to the IEQ method. If we shift our focus towards preserving linearity, it becomes evident that this attribute is primarily derived from the auxiliary variable's linearity with respect to the auxiliary function's derivative. This understanding guides us to consider decomposing the auxiliary function $c(r)$ into a product of the auxiliary variable $r$ and $\frac{c(r)}{r}$. Here we can drop the convexity assumption of $c(r)$ and simply take it as a smooth function. If we consider $\frac{c(r)}{r}$ as an independent auxiliary variable, denoted as $g(r)$, then the derivative of $rg(r)$ ends up being linear concerning the auxiliary variable $r$. This idea provides an alternative means of replacing the original energy density function to achieve a linear scheme but emphasizes preserving linearity more than convexity. We introduce this method in the following section, referring to it as the "Invariant Energy Functionalization" method. This name stems from integrating the auxiliary variable $r$ with a specific function $g$ that fulfills some conditions.

\subsection{Model reformulation}

Let us consider $g$ to be a smooth function on $\mathbb{R}$. Analogous to the hypothesis in the IEC formulation, we assume that $s(r):=rg(r)$ is invertible over a connected set $K$ with $\mathbb{R}^+\subset s(K)$. We define $r:[0,T]\times\Omega\to K$ to be a function that resolves the equation $r(t,x)g\left(r(t,x)\right)=F\left(\phi(t,x)\right)+A_1$ for all $(t,x)\in [0,T]\times\Omega$, where $A$ is a constant employed to ascertain positivity of $F$. Utilizing the chain rule, we can then derive the following:
\begin{equation*}
    \left[rg'(r)+g(r)\right]\frac{\delta r}{\delta \phi}=f(\phi).
\end{equation*}
Thus, we can rewrite gradient-flow system \eqref{eq:gradient_flow_system} as
\begin{subequations}
    \label{eq:IEF_formulation}
    \begin{equation}
    \label{eq:phi_update_IEF}
        \phi_t=\mathcal{G}\mu,
    \end{equation}
    \begin{equation}\label{eq:mu_update_IEF}
        \mu=-\Delta\phi +\left[rg'(r)+g(r)\right]\frac{\delta r}{\delta \phi},
    \end{equation}
    \begin{equation}\label{eq:r_update_IEF}
        r_t =\frac{\delta r}{\delta \phi }\phi_t.
    \end{equation}
    \begin{equation}
        \label{eq:g_update_IEF}
        g_t = g'(r)\,r_t
    \end{equation}
\end{subequations}
Given the equations above, we can deduce a modified energy dissipation law by computing the inner product of \eqref{eq:IEF_formulation} with $\mu$, $\phi_t$, $\left[rg'(r)+g(r)\right]$, and $r$ respectively, and then summing the results. This leads us to
\begin{equation}
    \label{eq:IEF_modified_energy}
     \frac{d}{dt}\left[\frac{1}{2}\|\nabla\phi\|^2+\int_\Omega rg(r)\right]=(G\mu, \mu)\leq 0,
\end{equation}
which holds for both $\mathcal{G}=-I$ and $\mathcal{G}=\Delta$. It's worth pointing out that one of the significant advantages of this formulation is that the auxiliary variable $r$ appears in a linear form in the reformulation of $f(\phi)$. This linear structure facilitates the development and implementation of a linear scheme to solve the equation, a topic that we will discuss more in the upcoming subsection.

\subsection{Numerical scheme}
Motivated by the formulation proposed above, a natural next step is introducing an additional variable to represent $g$ within the corresponding numerical scheme. In the continuous formulation, $g=g(r)$ is a dependent function of $r$. Thus, once $r$ is established, the value of $g$ is subsequently determined. Nonetheless, to preserve the linearity of the numerical scheme, we propose the introduction of a discrete variable for $g$, and update its value through an explicit discretization. This idea leads to the following Invariant Energy Functionalization (IEF) numerical scheme:

\begin{subequations}
\label{eq:IEF_scheme}
    \begin{equation}
    \label{eq:gf_numer_IEF}
        \Dtphi=\mathcal{G} \munew,
    \end{equation}
    \begin{equation}
    \label{eq:munew_IEF}
        \munew = -\Delta \phinew+\left[\rnew g'(r^n)+g^{n+1}\right]\Pold,
    \end{equation}
    \begin{equation}
    \label{eq:drdt_IEF}
        \rnew-\rold=\Pold(\phinew-\phiold),
    \end{equation}
    \begin{equation}
        \label{eq:g_new_IEF}
        g^{n+1}-g^n=g'(r^n)(\rnew-\rold),
    \end{equation}
\end{subequations}
with $P^n=\frac{\delta r}{\delta Q}(\phiold)$. In addition to the regularity and the invertibility conditions that should be satisfied by $g$, we need to assume the derivative of $g$ to be non-negative, namely, 
\begin{equation}
    \label{eq:g_derivative_non_negative}
    g'(r)\geq 0, 
\end{equation}
for any $r\in\mathbb{R}$. This is necessary to ensure the modified energy stability in discrete cases to hold, as stated below in Theorem \ref{thm:IEF_energy_stability}. Before that, we will discuss the procedure to implement this method first.

In a manner analogous to that of the IEC scheme, \eqref{eq:drdt_IEF}, \eqref{eq:g_new_IEF} provide the path representing $\rnew, \gnew$ in terms of $\phinew$. Consequently, one can replace $\rnew, \gnew$ in \eqref{eq:munew_IEF}, resulting in a linear scheme pertaining to $\phinew$ when considered in conjunction with \eqref{eq:phi_update_IEF}. Alternatively, an update can be performed on the tuple $(\phinew, \munew, \rnew, \gnew)$ collectively by developing a linear scheme to resolve \eqref{eq:IEF_scheme} as a unified system. More precisely, at each iteration, the following linear system could be solved:

\begin{equation}
\label{eq:IEF_scheme_implementation}
    \begin{pmatrix}
        \frac{1}{\Delta t}I & -\mathcal{G} & 0 & 0\\
        \Delta & I & -\widetilde{{g}'(\rold){P}^n} & -\Tilde{P}^n \\
        -\Tilde{P}^n & 0 & I & 0 \\
        0 & 0 & -\widetilde{{g}'(\rold)}& I
    \end{pmatrix}
    \begin{pmatrix}
        \phinew \\ \munew \\ \rnew \\ \gnew
    \end{pmatrix} = \begin{pmatrix}
         \frac{1}{\Delta t}\phiold\\0\\
        \rold-P^n\phiold\\
        \gold-g'(\rold)\rold
    \end{pmatrix},
\end{equation}
given that $(\phiold, \mu^n, \rold, 
\gold)$ is known. Here $\widetilde{{g}'(\rold)}$ should also be interpreted as a multiplication operator similar as $\Tilde{P}^n$.

Now we turn to prove a modified energy dissipation law for the IEF scheme which is a discretized version of \eqref{eq:IEF_modified_energy}.

\begin{thm}\label{thm:IEF_energy_stability}
    The numerical scheme \eqref{eq:IEF_scheme} is energy-stable. Specifically, define 
    \begin{equation}
    \label{eq:modified_energy_IEF}
        \EoldIEF=\frac{1}{2}\|\nabla \phi^n\|^2+\int_\Omega g^nr^n,
    \end{equation}
    then
    \begin{equation*}
        \EnewIEF-\EoldIEF\leq0.
    \end{equation*}
\end{thm}
\begin{proof}
    Taking inner product of \eqref{eq:phi_update_IEF} with $\munew\Delta t$, \eqref{eq:mu_update_IEF} with $\phinew-\phiold$, \eqref{eq:r_update_IEF} with $g^{n+1}$, \eqref{eq:g_update_IEF}, and summing them up with \eqref{eq:g_update_IEF} together, we obtain, 
 \begin{align*}
        (\mathcal{G}\munew, \munew)\Delta t&=(\phinew-\phiold, \munew)\\
        &=\left(\phinew-\phiold,  -\Delta \phinew+\left[\rnew g'(r^n)+g^{n+1}\right]\Pold\right)\\
        &=\frac{1}{2}\|\nabla \phinew\|^2-\frac{1}{2}\|\nabla \phiold\|^2+\frac{1}{2}\|\nabla\phinew-\nabla\phiold\|^2\\
        &\quad\,\,+\int_\Omega(g^{n+1}-g^n)\,\rnew\,dx+\int_\Omega \gnew(\rnew-\rold)\,dx\\
        &=\frac{1}{2}\|\nabla \phinew\|^2-\frac{1}{2}\|\nabla \phiold\|^2+\frac{1}{2}\|\nabla\phinew-\nabla\phiold\|^2\\
        &\quad\,\,+\int_\Omega \gnew\rnew\,dx-\int_\Omega \gold\rold\,dx+\int_\Omega (\gnew-\gold)\,(\rnew-\rold)\,dx\\
        &=\frac{1}{2}\|\nabla \phinew\|^2-\frac{1}{2}\|\nabla \phiold\|^2+\frac{1}{2}\|\nabla\phinew-\nabla\phiold\|^2\\
        &\quad\,\,+\int_\Omega \gnew\rnew\,dx-\int_\Omega \gold\rold\,dx+\int_\Omega g'(\rold)(\rnew-\rold)^2\,dx\\
        &\geq \EnewIEF-\EoldIEF,
    \end{align*}
    where we have used the assumption \eqref{eq:g_derivative_non_negative} to see that $g'(\rold)\geq0$. Using the fact that $(\mathcal{G}\munew, \munew)\leq 0$ independent of choice of $\munew$, we have finished the proof.
\end{proof}
Similar as Corollary \ref{cor:mu_estimate_IEC}, we can also obtain an estimate for $\munew$ for the IEF formulation. 
\begin{cor}
\label{cor:mu_estimate_IEF}
    Assume $\hat{E}^0$ is bounded. For fixed $N>0$, we have
    \begin{equation*}
       0\leq -\sum_{n=0}^N (\mathcal{G}\mu^n, \mu^n)\Delta t\leq -\hat{E}^N+\hat{E}^0\leq \hat{E}^0,
    \end{equation*}
    and so it is bounded. Specifically, if $\mathcal{G}=\Delta$, we have 
    \begin{equation*}
       0\leq \sum_{n=0}^N \|\nabla\mu^n\|^2\Delta t\leq \hat{E}^0,
    \end{equation*}
    if $\mathcal{G}=-I$,  we have 
    \begin{equation*}
       0\leq \sum_{n=0}^N \|\mu^n\|^2\Delta t\leq \hat{E}^0.
    \end{equation*}
\end{cor}

\subsection{An example of the IEF formulation}\label{subsec:IEF_example}
Similar as what we have done after establishing the IEC scheme, in this part, we will provide a detailed example to illustrate the practicality of the IEF method. A suitable choice is to take $g(r)=r^{2k+1}$, for $k=0, 1, 2, \cdots$. It is easy to justify that $rg(r)=r^{2k+2}$ is invertible over $[0,\infty)$ and assumption \eqref{eq:g_derivative_non_negative} is satisfied since $g'(r)=(2k+1)r^{2k}\geq0$. This would result in 
\begin{equation}
    \label{eq:IEF_mono_r}
    r(\phi)=\left(F(\phi)+A_1\right)^{\frac{1}{2k+2}}.
\end{equation}
Therefore, the numerical scheme \eqref{eq:IEF_scheme} can be implemented as
\begin{subequations}
\label{eq:IEF_scheme_poly}
    \begin{equation}
    \label{eq:gf_numer_IEF_mono}
        \Dtphi=\mathcal{G} \munew,
    \end{equation}
    \begin{equation}
    \label{eq:munew_IEF_mono}
        \munew = -\Delta \phinew+\left[(2k+1)(\rold)^{2k}\rnew+\gnew\right]\Pold,
    \end{equation}
    \begin{equation}
    \label{eq:drdt_IEF_mono}
        \rnew-\rold=\Pold(\phinew-\phiold),
    \end{equation}
    \begin{equation}
        \label{eq:g_new_IEF_mono}
        g^{n+1}-g^n=(2k+1)(\rold)^{2k}(\rnew-\rold),
    \end{equation}
\end{subequations}
where $\Pold = P(\phiold)$ and 
\begin{equation}
    \label{eq:IEF_poly_P}
    P(\phi)=\frac{\delta r}{\delta \phi}(\phi)=\frac{f(\phi)}{(2k+2)(F(\phi)+A_1)^{\frac{2k+1}{2k+2}}}.
\end{equation}
This procedure can be summarized as the following Algorithm \ref{alg:IEF_monomial}.

\begin{algorithm}
\caption{IEF scheme induced by monomial functions} \label{alg:IEF_monomial}
\begin{algorithmic}[1] 
\Procedure{}{}
    \State Set the step size $\Delta t$
    \State Set the total number of iterations $N$
    \State Set the initial value $\phi^0 = \phi_0$
    \State Compute $r^0 = r(\phi^0)$ using equation \eqref{eq:IEF_mono_r}
    \State Compute $g^0=(r^0)^{2k+1}$
    \State Compute $P^0 = P(\phi^0)$ using equation \eqref{eq:IEF_poly_P}
    
    \For {$n = 0, 1, 2, ..., N-1$}
        \State Construct the matrix $A=\begin{pmatrix}
        \frac{1}{\Delta t}I & -\mathcal{G} & 0 & 0\\
        \Delta & I & -(2k+1)\widetilde{(\rold)^{2k}\Pold }& -\Pold \\
        -\Tilde{P}^n & 0 & I & 0 \\
        0 & 0 & -(2k+1)\widetilde{(\rold)^{2k}} & I
    \end{pmatrix}$ using $P^{n}$
        \State Construct the right-hand side $b=\begin{pmatrix}
         \frac{1}{\Delta t}\phiold\\0\\
        \rold-P^n\phiold\\
        \gold-(2k+1)(\rold)^{2k}\rold
    \end{pmatrix}$
        \State Update $\phi^{n+1}$, $r^{n+1}, \gnew$ by solving the linear system $A\begin{pmatrix}
            \phi^{n+1}\\\mu^{n+1}\\ r^{n+1}\\ \gnew
        \end{pmatrix} = b$
        \State Update $P^{n+1} = P(\phi^{n+1})$ using equation \eqref{eq:IEF_poly_P}
    \EndFor
\EndProcedure
\end{algorithmic}
\end{algorithm}

It merits attention that by simply assigning a value of $k=0$, which consequently yields $g(r)=r$, we can reconstruct the IEQ method. It is because, in this scenario, the computation $g(r)r=r^2$ holds, and the update rule given by \eqref{eq:g_new_IEF_mono} is equivalent to \eqref{eq:drdt_IEF_mono}. This enforces $\gnew=\rnew$ for each iterative step $n=0, 1, 2, \ldots ,N-1.$ Such a formulation endorses our claim that the IEF is genuinely an extension of the IEQ method.

When choosing a suitable function $g$ for the IEF scheme, our selection should not be confined to monomial functions, provided they adhere to the abovementioned conditions. However, finding functions other than the monomials that result in an explicit formula for $r$ in terms of $F(\phi)$ is challenging. In such instances, one might need to employ root-finding techniques to determine the value of $\Pold$ once $\phiold$ is known. This inherent challenge positions our method as an Invariant Energy Monomial method in the examples provided here and in subsequent numerical instances. Even so, should an efficient technique be found for computing the root of the equation $g(r)r=F(\phi)+A_1 $ for a specific form of $g$, constructing an IEF scheme using our approach would not pose a significant challenge. This potential advancement forms an intriguing aspect for future exploration of this method.

\section{Numerical Experiments}
\label{sec:Numerics}

In this section, we conduct a series of numerical experiments, focusing on the Allen-Cahn \cite{allen1979microscopic} and Cahn-Hilliard equations \cite{cahn1959free, Cahn-Hilliard} under the application of periodic boundary conditions in a two-dimensional domain. The equations above are commonplace test cases for gradient flow-targeted numerical algorithms, as illustrated in \cite{brenner2018robust,diegel2016stability,gao2022unconditionally,shen2018scalar, shen2019new, zhao2021second}. Moreover, these equations are particular manifestations of general gradient flows \eqref{eq:gradient_flow_system} and can be formulated using the free energy:
\begin{equation*}
    E(\phi)=\int_\Omega\left(\frac{\varepsilon^2}{2}|\nabla\phi|^2+F(\phi)\right)\,dx,
\end{equation*}
where 
\begin{equation}
    \label{eq:AC_CH_free_energy}
    F(\phi) = \frac{1}{4}\left( \phi^2-1\right)^2.
\end{equation}
Here $\epsilon$ symbolizes a parameter that introduces stiffness issues into the PDE system when $\varepsilon \ll 1$ \cite{yang2020convergence}.

Employing the variational approach to $E(\phi)$ in $L^2$, the Allen-Cahn equation materializes as:
\begin{subequations}
    \label{eq:Allen-Cahn}
    \begin{equation}
        \phi_t = -M\mu,
    \end{equation}
    \begin{equation}
        \mu = -\varepsilon^2\Delta \phi+f(\phi).
    \end{equation}
\end{subequations}
Here $M$ is the mobility constant, and $f(\phi)=\frac{\delta F}{\delta \phi}=\phi(\phi^2-1)$ \cite{chen2019fast}.

In a similar manner, applying the variational approach to $E(\phi)$ within the $H^{-1}$ space yields the Cahn-Hilliard equation:
\begin{subequations}
    \label{eq:Cahn-Hilliard}
    \begin{equation}
        \phi_t = M\Delta\mu,
    \end{equation}
    \begin{equation}
        \mu = -\varepsilon^2\Delta \phi+f(\phi).
    \end{equation}
\end{subequations}

We choose a computational domain defined by a square $\Omega=[0, 2\pi]\times [0, 2\pi]$ for the subsequent experiments. We will implement a standard finite difference method for this problem with periodic boundary conditions to facilitate full discretization. Unless otherwise stated, the domain will be discretized utilizing a $40\times 40$ grid, and we will set the parameters $M=0.6$ and $\varepsilon=0.4$.

\subsection{Accuracy test}

In this section, we start our numerical experiments by assessing the convergence rates of the Invariant Energy Convexification (IEC) scheme in its application to the Allen–Cahn equation, as given by \eqref{eq:Allen-Cahn}. For these experiments, we adopt

\begin{equation}
\label{eq:accuracy_test_function}
\phi(x,y,t) = \sin(x)\cos(y)\cos(t)
\end{equation}
as the exact solution and include an appropriate right-hand side force field term to ensure that the designated solution complies with the system \eqref{eq:Allen-Cahn}. We draw attention to the fact that the function $\phi$ employed here aligns with the choice made for the accuracy test in \cite{yang2020convergence}.

In the context of the IEC scheme \eqref{eq:IEC_scheme}, we examine three distinct choices for the function $c(r)$, namely, ${ \ln(1+e^r), \ln(r)^2, r^2}$. Each of these functions are L-smooth and convex, and an elementary analysis would suggest that $L=2$ serves as an appropriate choice for all three. Maintaining consistency, we set $\alpha = 0.5$ for our tests.

We then measure the $L^2$ errors for the variable $\phi$, contrasting our numerical solution with the exact solution at $T_{end}=1$ under varying time step sizes. As visualized in Figure \ref{fig:accuracy_test_a}, each of the three choices of convex functions demonstrates first-order accuracy. The error is linearly decreasing when the time step size is decreasing as well.  More specific error data for varying time step sizes can be found in Table \ref{tab:accuracy_test}.

A noteworthy observation emerges when keeping the parameters $\alpha$ and $L$ constant in the numerical scheme: errors vary depending on the choice of the corresponding convex function. Specifically, $c(r)=\ln(r)^2$ appears to yield the best performance, while $c(r)=r^2$ lags behind. This finding underscores the likelihood that, for a particular problem, a certain function may appear to be more suitable to the construction of the numerical scheme in contrast to a quadratic function to get better precision. 

\begin{figure}[h]
    \centering
    \begin{subfigure}{0.45\textwidth}
        \centering
        \includegraphics[width=\textwidth]{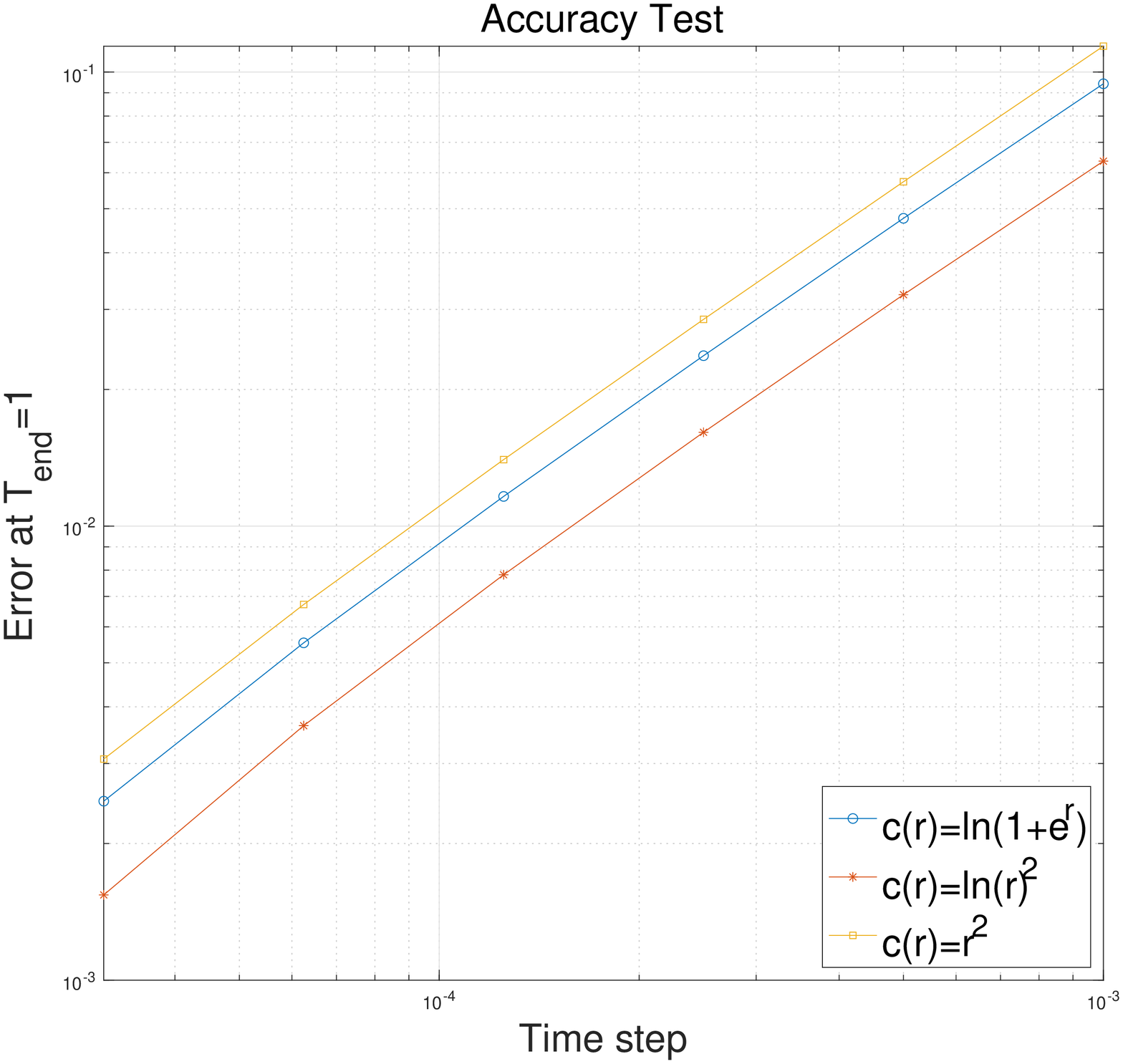}
        \caption{}
        \label{fig:accuracy_test_a}
    \end{subfigure}\hfill
    \begin{subfigure}{0.52\textwidth}
        \centering
        \includegraphics[width=\textwidth]{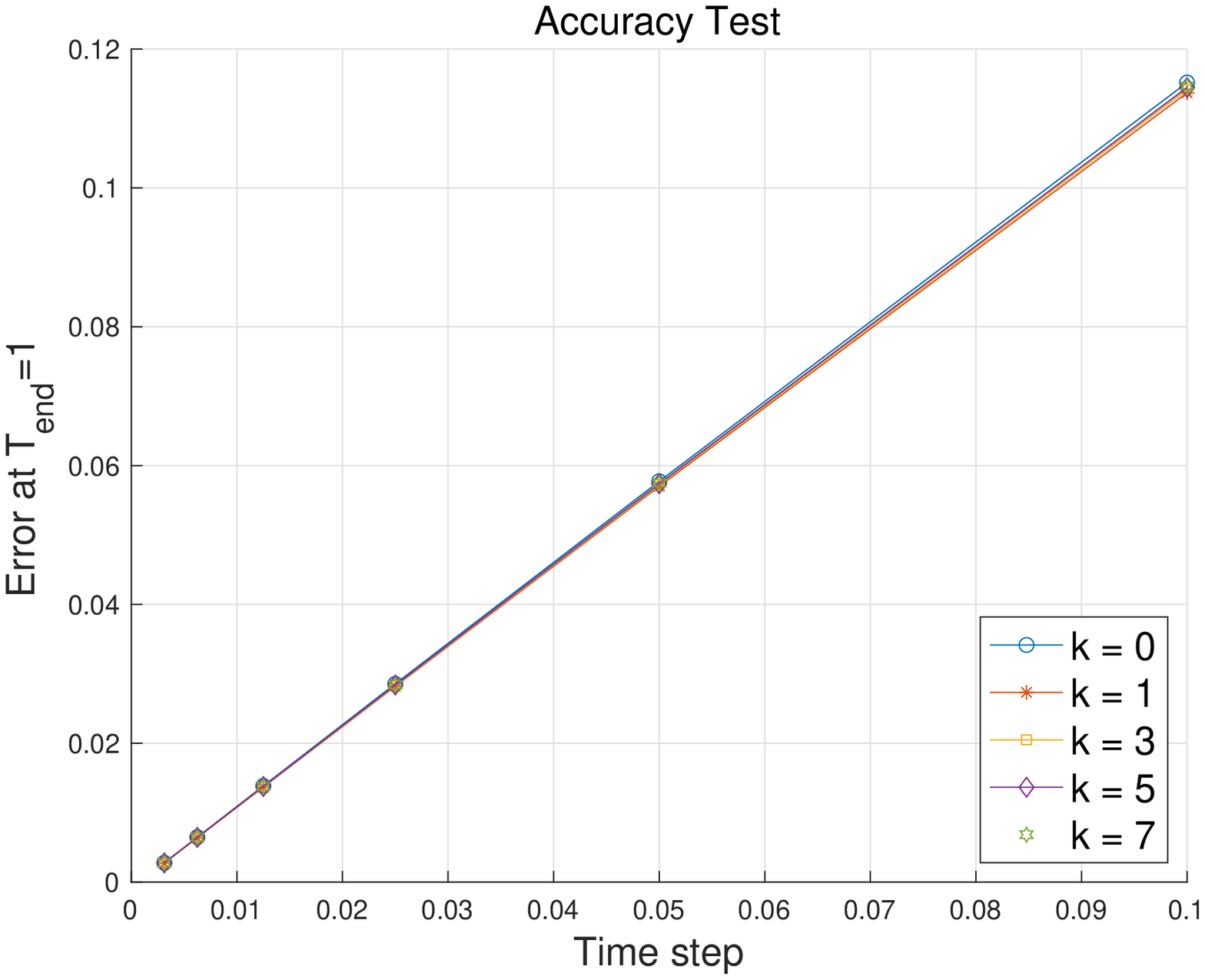}
        \caption{}
        \label{fig:accuracy_test_b}
    \end{subfigure}
    \caption{The $L^2$ numerical errors at $T_{end}=1$ for $\phi$ under different formulation. (a) Illustrates the numerical errors of the IEC scheme using different choices of convex functions. (b) Demonstrates the numerical errors of the IEF scheme, where varying monomials are utilized as the multiplying functions.}
    \label{fig:accuracy_test}
\end{figure}

\begin{table}[htbp]
    \centering
    \caption{The $L^2$ numerical errors at $T_{end}=1$ for $\phi$ with three different convex functions at different time step size}\footnotesize
    \begin{tabular}{
         |l
        |S[round-mode=figures, round-precision=3, scientific-notation=true, table-format=1.2e1]
        |S[round-mode=figures, round-precision=3, scientific-notation=true, table-format=1.2e1]
        |S[round-mode=figures, round-precision=3, scientific-notation=true, table-format=1.2e1]
        |S[round-mode=figures, round-precision=3, scientific-notation=true, table-format=1.2e1]
        |S[round-mode=figures, round-precision=3, scientific-notation=true, table-format=1.2e1]
        |S[round-mode=figures, round-precision=3, scientific-notation=true, table-format=1.2e1]
        |}
    \hline
    \diagbox{Function}{Time step} & {$\delta t = 1.00 \times 10^{-1}$} & {$\delta t = 5.00 \times 10^{-2}$} & {$\delta t = 2.50 \times 10^{-2}$} & {$\delta t = 1.25 \times 10^{-2}$} & {$\delta t = 6.25 \times 10^{-3}$} & {$\delta t = 3.12 \times 10^{-3}$}\\
    \hline
    $c(r)=\ln(1+e^r)$ & 0.094215 & 0.047625 & 0.023729 & 0.011624 & 0.005532 & 0.002479 \\
    \hline $c(r)=\ln(r)^2$ & 0.063632 & 0.032332 & 0.016092 & 0.007814 & 0.003635 & 0.001541 \\
    \hline $c(r)=r^2$ & 0.114006 & 0.057330 & 0.028529 & 0.014009 & 0.006720 & 0.003069\\
    \hline
    \end{tabular}
    \label{tab:accuracy_test}
\end{table}

We further seek to explore the impact of alterations in the values of $\alpha$ on the accuracy of the IEC scheme. With this objective in mind, we apply the IEC schemes with $c(r)=\ln(1+e^r)$ and $c(r)=\ln(r)^2$ using a range of values for $\alpha$. We continue to assess the $L^2$ error between the numerical solution and the exact solution at $T_{end}=1$. The outcomes of this exploration are depicted in Figure \ref{fig:accuracy_test_alpha}. For this experiment, we select values for $\alpha$ that range from its theoretical lower bound of $0.5$—a value necessary to ensure the energy stability of the numerical scheme—up to $16$. When the time step size is large, the performance of numerical schemes with different values of $\alpha$ has obvious differences. In comparison, such a difference becomes negligible when the time step size is small.  Another intriguing fact is that when the time step size is relatively large, an increase in the value of $\alpha$ initially enhances the accuracy of the corresponding numerical scheme but then reduces precision. When constructing the numerical scheme using $c(r)=\ln(1+e^r)$, we find that optimal accuracy is achieved for values of $\alpha$ within the interval $[4,8]$. Contrarily, for $c(r)=\ln(r)^2$, the scheme exhibits optimal performance when $\alpha$ lies within the range $[2,4]$. These findings suggest that different convex functions may produce distinct optimal values for $\alpha$. As such, when we use a relatively large time step, careful consideration of $\alpha$'s value is necessary when implementing the IEC scheme.

\begin{figure}[h]
    \centering
    \begin{subfigure}{0.45\textwidth}
        \centering
        \includegraphics[width=\textwidth]{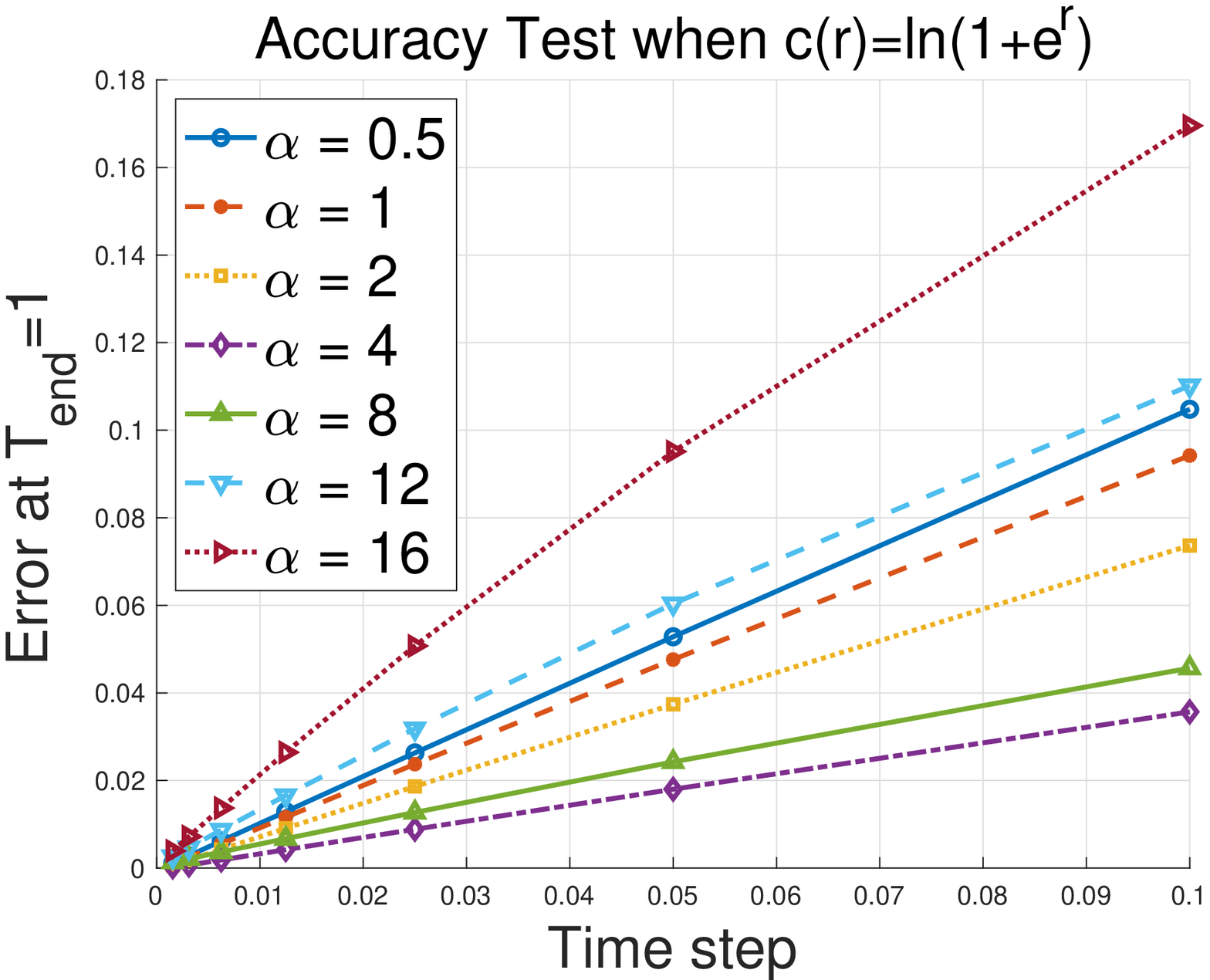}
        \caption{}
        \label{fig:accuracy_test_alpha_a}
    \end{subfigure}\hfill
    \begin{subfigure}{0.45\textwidth}
        \centering
        \includegraphics[width=\textwidth]{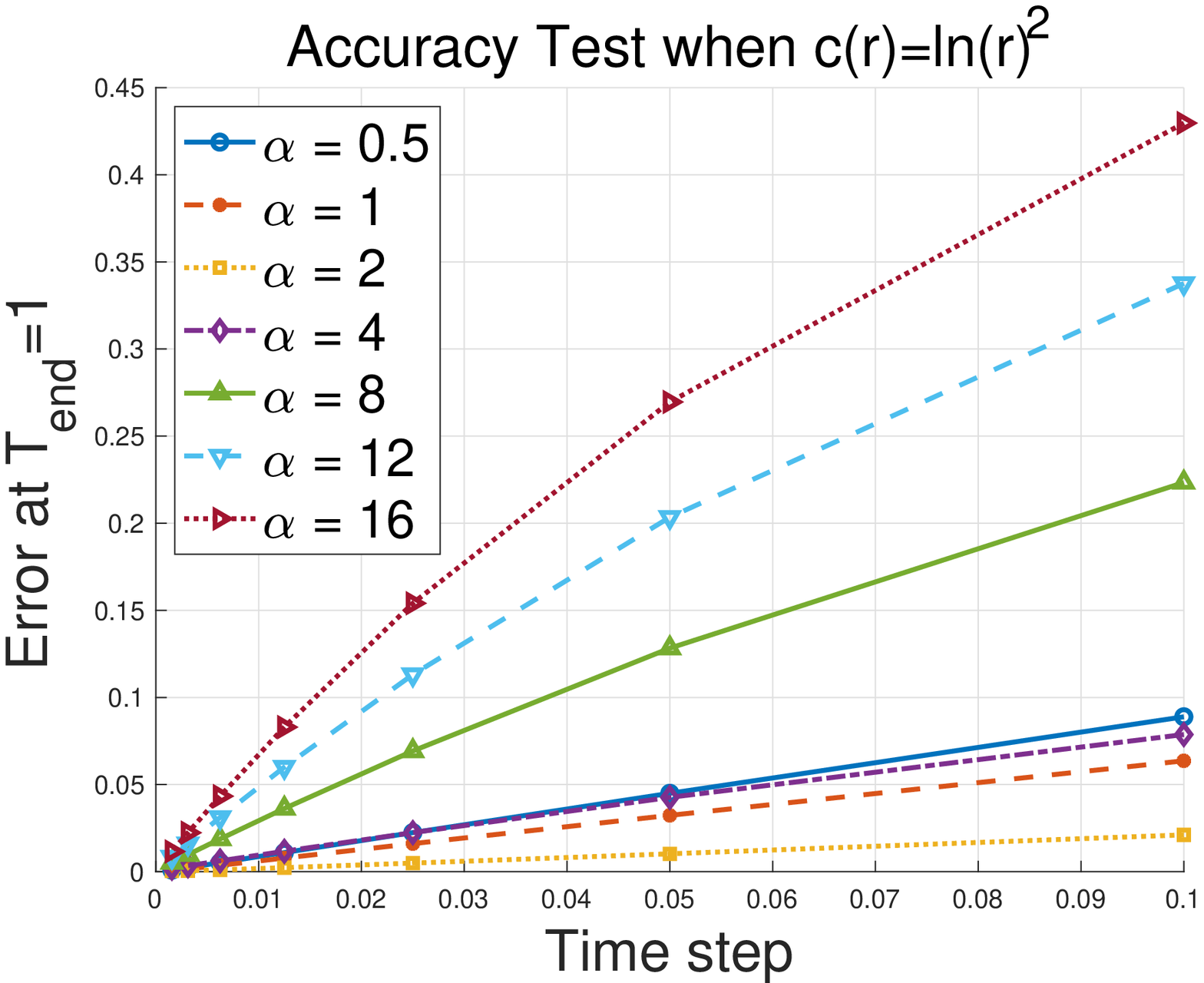}
        \caption{}
        \label{fig:accuracy_test_alpha_b}
    \end{subfigure}
    \caption{The $L^2$ numerical errors at $T_{end}=1$ for $\phi$ with different values of $\alpha$ under the IEC formulation. (a) Scheme constructed with $c(r)=\ln{(1+e^r)}$. (b) Scheme constructed with $c(r)=\ln(r)^2$.}
    \label{fig:accuracy_test_alpha}
\end{figure}

We now discuss the experimental findings of the Invariant Energy Functionalization (IEF) scheme. We will employ the scheme constructed using a monomial as the auxiliary function $g$, a concept introduced in Section \ref{subsec:IEF_example}. Our evaluation metric remains the $L^2$ error for $\phi$, calculated between the numerical and exact solutions at $T_{end}=1$. The errors are illustrated in Figure \ref{fig:accuracy_test_b}, and the precise data is provided in Table \ref{tab:IEF_accuracy_test}. Insight into these results reveals that the highest level of algorithmic precision is achieved when $g(r)=r$, corresponding to the Invariant Energy Quadratization (IEQ) method.
Nevertheless, it is worth noting that the discrepancies between different values of $k$ are tiny. Furthermore, regardless of the choice of $k$, all options yield first-order accuracy. It implies that the norm of errors diminishes linearly to the reduction in time step size. This outcome further validates the robustness and utility of the IEF scheme in this context. 

\begin{table}[htbp]
    \centering
    \caption{The $L^2$ numerical errors at $T_{end}=1$ for $g(r)=r^k$ with different values of $k$}\footnotesize
    \begin{tabular}{
         |l
        |S[round-mode=figures, round-precision=3, scientific-notation=true, table-format=1.2e1]
        |S[round-mode=figures, round-precision=3, scientific-notation=true, table-format=1.2e1]
        |S[round-mode=figures, round-precision=3, scientific-notation=true, table-format=1.2e1]
        |S[round-mode=figures, round-precision=3, scientific-notation=true, table-format=1.2e1]
        |S[round-mode=figures, round-precision=3, scientific-notation=true, table-format=1.2e1]
        |S[round-mode=figures, round-precision=3, scientific-notation=true, table-format=1.2e1]
        |}
    \hline
    \diagbox{$g(r)=r^k$}{Time step} & {$\delta t = 1.00 \times 10^{-1}$} & {$\delta t = 5.00 \times 10^{-2}$} & {$\delta t = 2.50 \times 10^{-2}$} & {$\delta t = 1.25 \times 10^{-2}$} & {$\delta t = 6.25 \times 10^{-3}$} & {$\delta t = 3.12 \times 10^{-3}$}\\
    \hline
    $k=0$ & 0.115178529752356&	0.0577313605929760&	0.0285583965072185&	0.0138571206292417&	0.00647896949207982&	0.00278705659803150 \\
    \hline $k=1$ & 0.113698137982773 &	0.0570092662997332&	0.0282023375688065&	0.0136805653966672&	0.00639130173889935&	0.00274374644253342 \\
    \hline $k=3$ &0.114145797493452	&0.0572262547547561	&0.0283090930170141	&0.0137335144095360	&0.00641767426632956	&0.00275689452740749\\
    \hline  $k=5$ & 0.114469994858300	&0.0573837230114717	&0.0283866211171265	&0.0137719627053119	&0.00643680380102496	&0.00276640171421471\\
    \hline  $k=7$ & 0.114664869002037	&0.0574783927357289	&0.0284332327431449	&0.0137950775853243	&0.00644830246187641	&0.00277211394424855\\
    \hline
    \end{tabular}
    \label{tab:IEF_accuracy_test}
\end{table}

\subsection{Energy stability}

In our discussion thus far, we have highlighted a significant advantage shared by the IEC and the IEF schemes, which they inherit from the Invariant Energy Quadratization (IEQ) method: the preservation of the energy dissipation law. In this section, we set out to empirically validate this property through a series of numerical experiments. We will examine an example with the initial condition given by:
\begin{equation}
\label{eq:energy_stability_example_initial_condition}
\phi(x,y,0)=\sin(x)\cos(y),
\end{equation}
and scrutinize the evolution of energy over the time interval $[0,5]$ when employing both the IEC and the IEF schemes to solve both the Allen-Cahn and Cahn-Hilliard equations.

We commence with the IEC scheme. We maintain a consistent choice of $c(r)$ as the Softplus function for these experiments, thereby employing Algorithm \ref{alg:IEC_softplus} for our implementations. As depicted in Figure \ref{fig:energy_stability_IEC}, the modified energy exhibits a monotonic decrease independent of the time step size for both the Allen-Cahn and Cahn-Hilliard equations. The definition of the modified energy, as referenced here, aligns with equation \eqref{eq:modified_energy_IEC}. This outcome confirms preserving the energy dissipation property within the IEC scheme.

\begin{figure}[h]
\centering
\begin{subfigure}{0.45\textwidth}
\centering
\includegraphics[width=\textwidth]{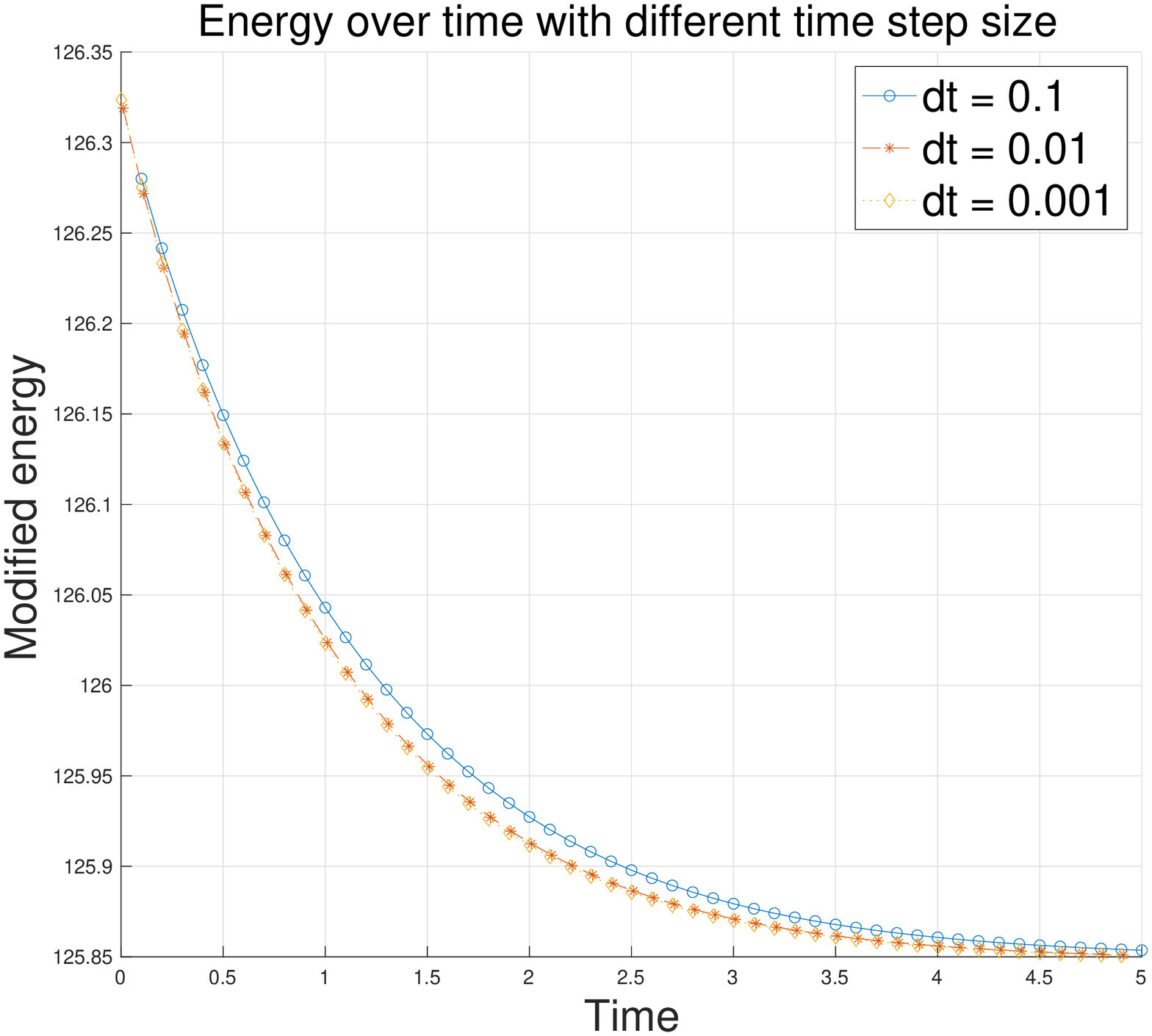}
\caption{}
\label{fig:energy_stability_IEC_AC}
\end{subfigure}\hfill
\begin{subfigure}{0.48\textwidth}
\centering
\includegraphics[width=\textwidth]{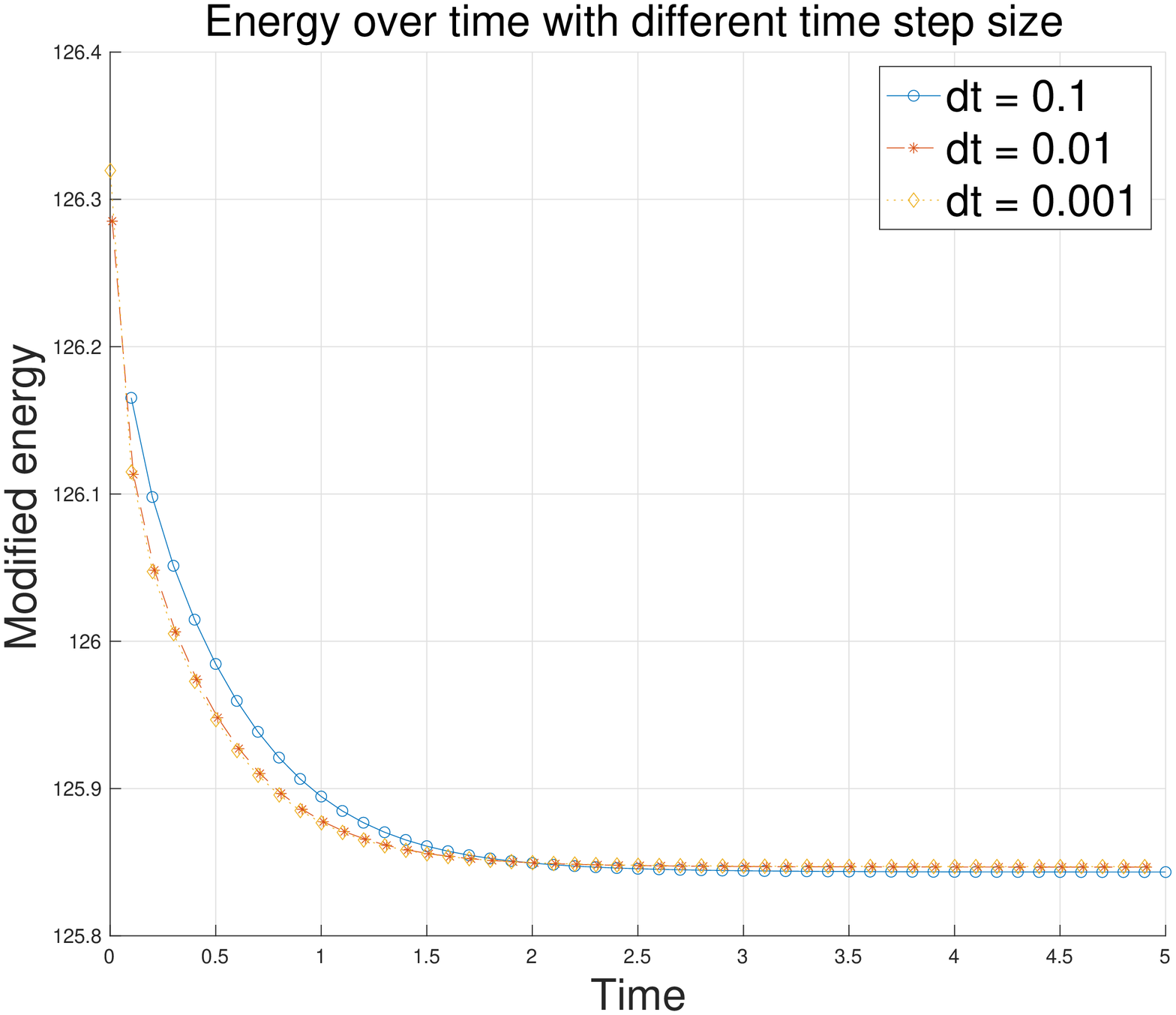}
\caption{}
\label{fig:energy_stability_IEC_CH}
\end{subfigure}
\caption{ Evolution of the modified energy over time under the IEC scheme with $c(r)=\ln(1+e^r)$: (a) Allen-Cahn equation; (b) Cahn-Hilliard equation.}
\label{fig:energy_stability_IEC}
\end{figure}

Another research question is the difference between the modified and original energies. Specifically, the original energy function $F(\phi)$, defined by equation \eqref{eq:AC_CH_free_energy}, is a function of the variable $\phi$, whereas its modified counterpart, $c(r)$, is a function of $r$ in the IEC formulation. An intriguing question is whether $c(r)$ would converge to $F(\phi)$ in the limit as the time step size approaches zero. In more technical terms, we aim to compute the quantity $\lvert \int_\Omega (c(r)-F(\phi))\vert$ at $T_{end}=5$ for various time step sizes and observe if this quantity approaches zero as the time step size diminishes.

To investigate this, we employ the same example with the initial condition specified by equation \eqref{eq:energy_stability_example_initial_condition}, persisting with the use of the Softplus function for constructing our IEC scheme. As demonstrated in Figure \ref{fig:energy_modified_energy_difference_IEC_AC}, the difference between the original energy and the modified energy indeed appears to approach zero as the time step size shrinks, when tested on the Allen-Cahn equation, with a first-order convergence rate. Furthermore, figure \ref{fig:energy_modified_energy_difference_IEC_CH} validates the same numerical results for the Cahn-Hilliard equation. This result offers additional evidence in favor of implementing the change of variable procedure, underscoring that the auxiliary variable computed by the numerical scheme will converge to the original function it aims to represent when the time step size is sufficiently small.

\begin{figure}[h]
\centering
\begin{subfigure}{0.45\textwidth}
\centering
\includegraphics[width=\textwidth]{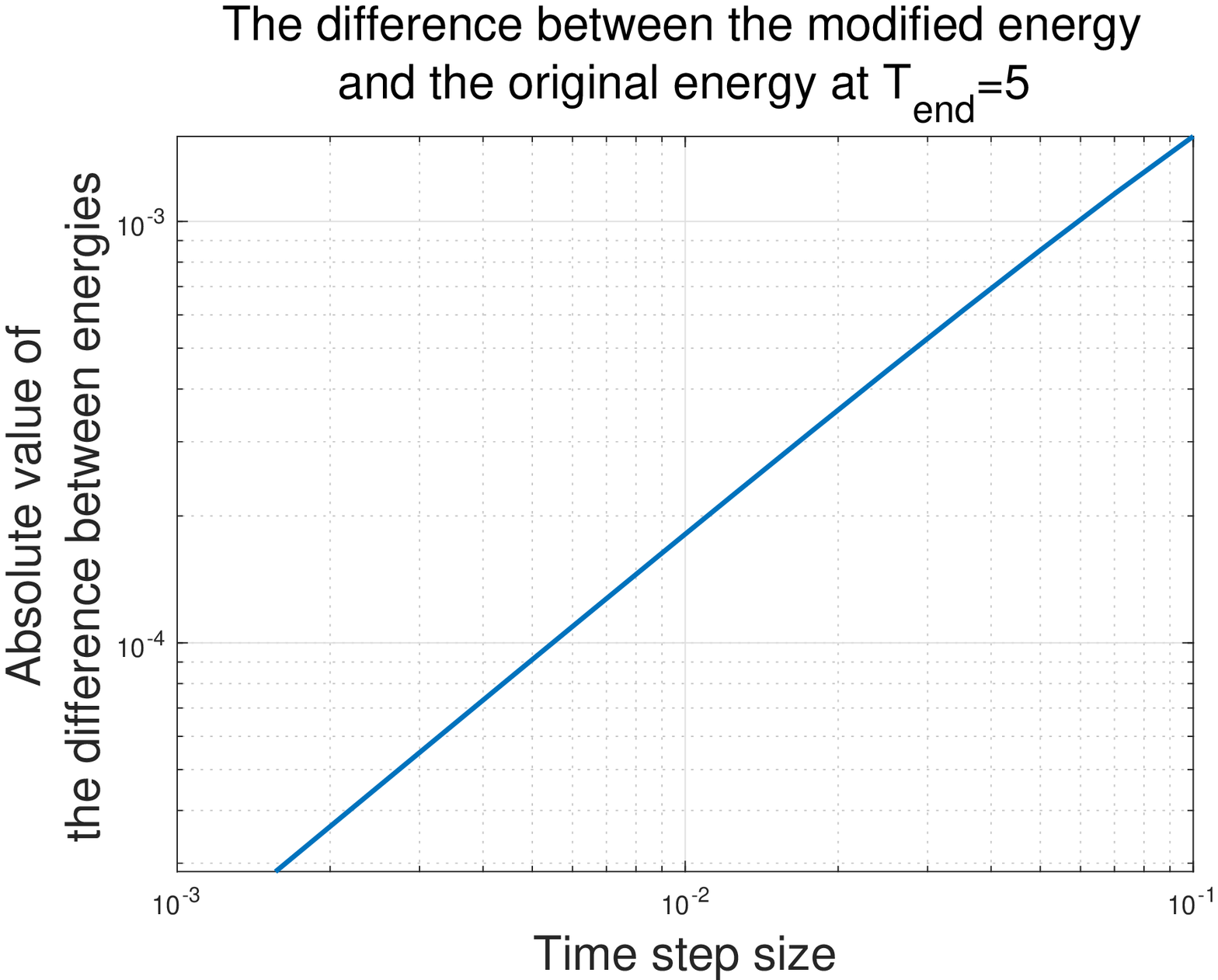}
\caption{}
\label{fig:energy_modified_energy_difference_IEC_AC}
\end{subfigure}\hfill
\begin{subfigure}{0.45\textwidth}
\centering
\includegraphics[width=\textwidth]{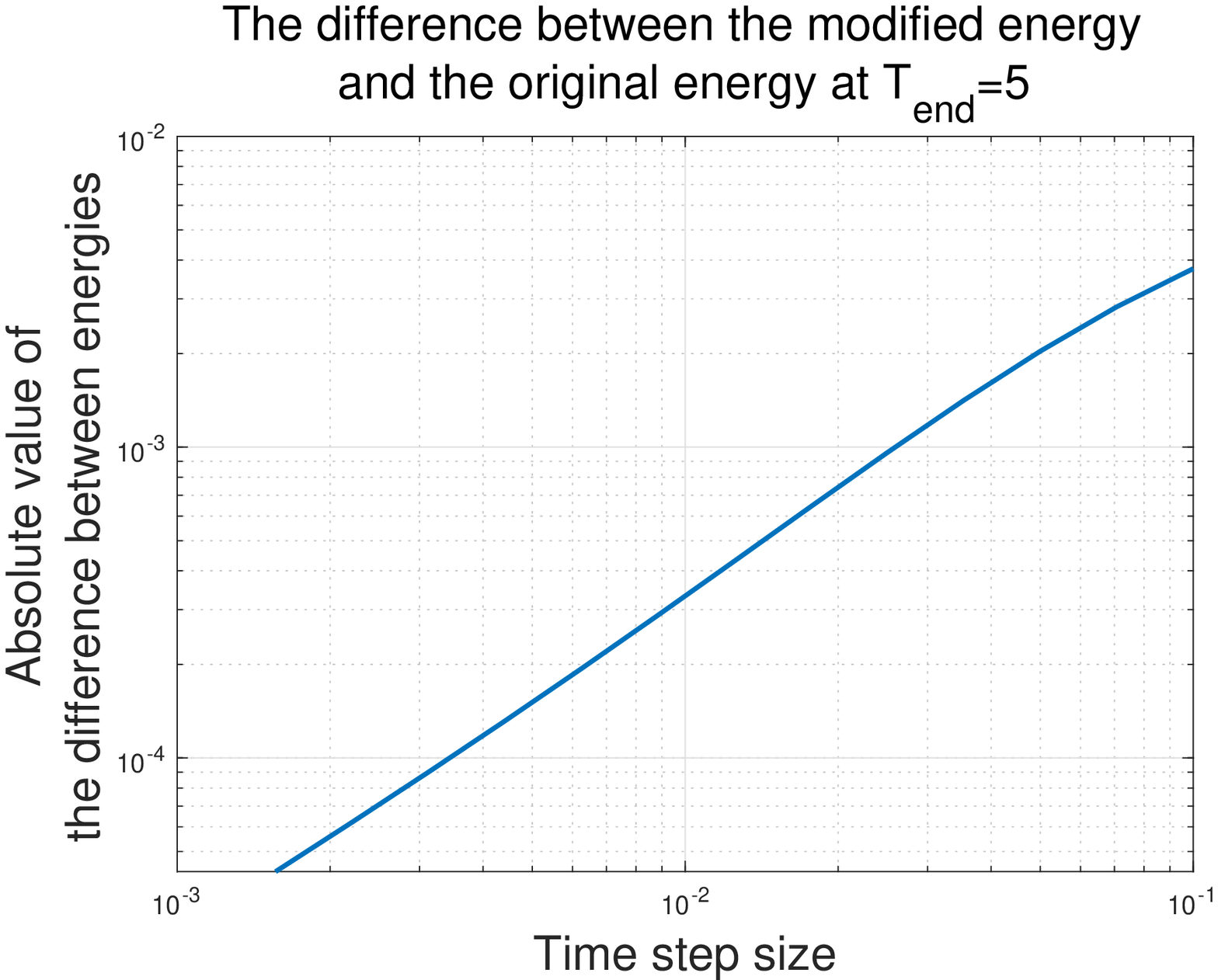}
\caption{}
\label{fig:energy_modified_energy_difference_IEC_CH}
\end{subfigure}
\caption{ Graphs depicting $\lvert \int_\Omega (c(r)-F(\phi))\vert$ at $T_{end}=5$ for varying time step sizes under the IEC scheme, constructed using the Softplus function: (a) Allen-Cahn equation; (b) Cahn-Hilliard equation.}
\label{fig:energy_modified_energy_difference_IEC}
\end{figure}

We next turn our attention to comparable experiments performed using the IEF scheme. We shall confine our discussion to applying the IEF scheme to the Cahn-Hilliard equation for brevity. Employing the same initial condition given by equation \eqref{eq:energy_stability_example_initial_condition}, we choose $g(r)=r^7$ to facilitate our investigations. The modified energy in this context is defined as per equation \eqref{eq:modified_energy_IEF}. Figure \ref{fig:IEF_CH_Energy_different_dt} delineates the evolution of this modified energy over time for varying time step sizes. The modified energy exhibits a monotonous decrease, in alignment with our theoretical predictions outlined in Theorem \ref{thm:IEF_energy_stability}.

For our second experiment using the IEF scheme, we diverge slightly from the approach used in the IEC scheme, where we computed the discrepancy between the modified and original energies. Instead, we focus on the disparity between the auxiliary variable and its corresponding value as a function of $\phi$ of differing time step sizes. As illustrated in Figure \ref{fig:IEF_CH_difference_auxiliary_variable}, the $L^2$ norm of the difference between $\lvert g-g\left(r(\phi)\right)\rvert$ and $\lvert r-r(\phi)\rvert$ at $T_{end}=5$ shows a linear dependency on the time step size, again demonstrating first-order accuracy. The results of these two experiments attest to the robustness and accuracy of the IEF scheme.

\begin{figure}[h]
\centering
\begin{subfigure}{0.45\textwidth}
\centering
\includegraphics[width=\textwidth]{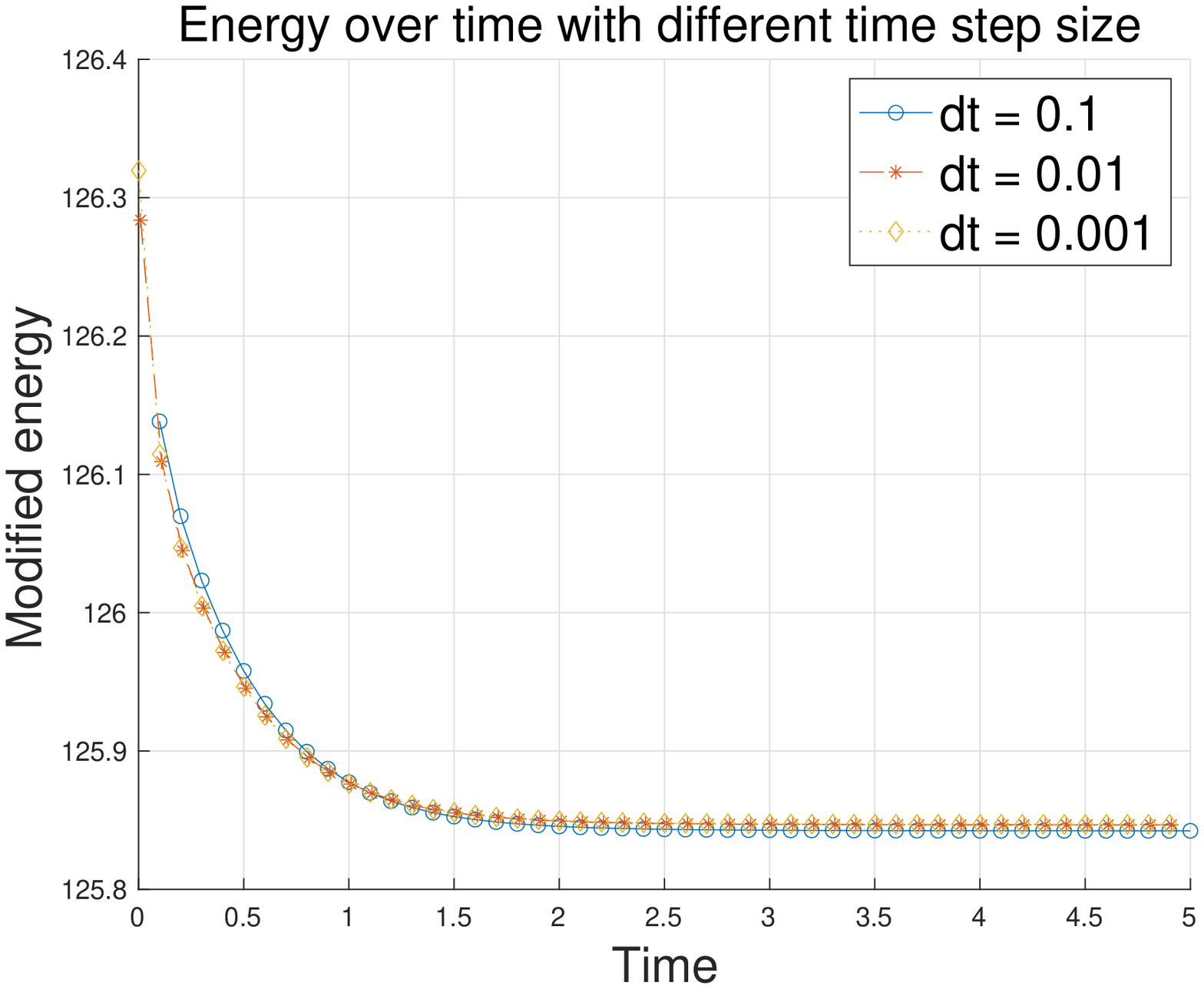}
\caption{Evolution of the modified energy over time under the IEF scheme for Cahn-Hilliard equation, constructed using $g(r)=r^7$}
\label{fig:IEF_CH_Energy_different_dt}
\end{subfigure}\hfill
\begin{subfigure}{0.43\textwidth}
\centering
\includegraphics[width=\textwidth]{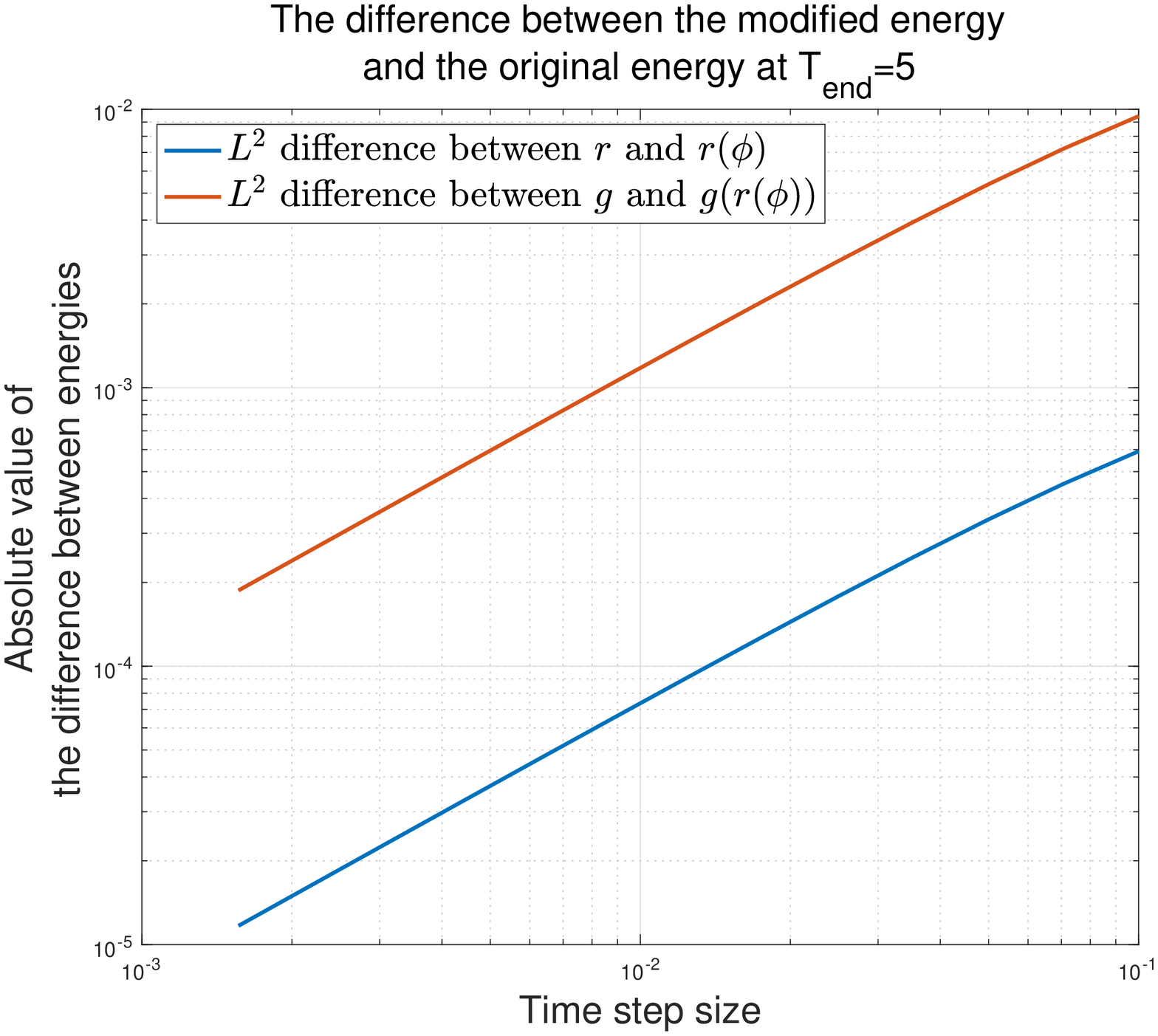}
\caption{Graph depicting $L^2$ norm of $\lvert g-g\left(r(\phi)\right)\rvert$ and $\lvert r-r(\phi)\rvert$ at $T_{end}=5$ for varying time step sizes under the IEF scheme for Cahn-Hilliard equation, constructed using $g(r)=r^7$}
\label{fig:IEF_CH_difference_auxiliary_variable}
\end{subfigure}
\caption{ Numerical results of the IEF scheme applied to the Cahn-Hilliard equation}
\label{fig:energy_modified_energy_difference_IEF}
\end{figure}

\subsection{Coarsening effect}

The coarsening effect \cite{dai2016computational,ju2015fast} refers to the phenomenon where small-scale structures or features in the system tend to merge and form larger-scale structures over time. This effect is related to the system's dynamics of phase separation or pattern formation. Here we will present some experimental results by applying the IEC and the IEF scheme to the Cahn-Hilliard equation to simulate such an effect. We firstly consider a benchmark problem for the corasening effect, for example, see \cite{chen2019fast, liu2022step, yang2020convergence}. The initial condition is set to be 
\begin{equation}
    \label{eq:coarsening_effect_initial_condition1}
    \phi(x,y,0)=-\sum_{i=1}^2\tanh\left( \frac{\sqrt{(x-x_i)^2+(y-y_i)^2}-r_i}{1.2\varepsilon}\right)+1,
\end{equation}
where we choose $(x_1,y_1,r_1)=(\pi-0.7, \pi-0.6, )$ and $(x_2, y_2, r_2)=(\pi+1.65, \pi+1.6, 0.8)$. We apply the IEC scheme formulated by the Softplus function, which is Algorithm \ref{alg:IEC_softplus}, to solve this problem with times step size set to be $0.001$. As we can see from the snapshots presented in Figure \ref{fig:snapshots_CH_IEC}, we observe the coarsening effect as the small circle is generally absorbed by the large circle over the time span $[0,3]$. The fully absorption happens at around $t=2.00s$.

\begin{figure}[ht]
    \centering
    \begin{subfigure}{0.22\textwidth}
        \centering
        \includegraphics[width=\linewidth]{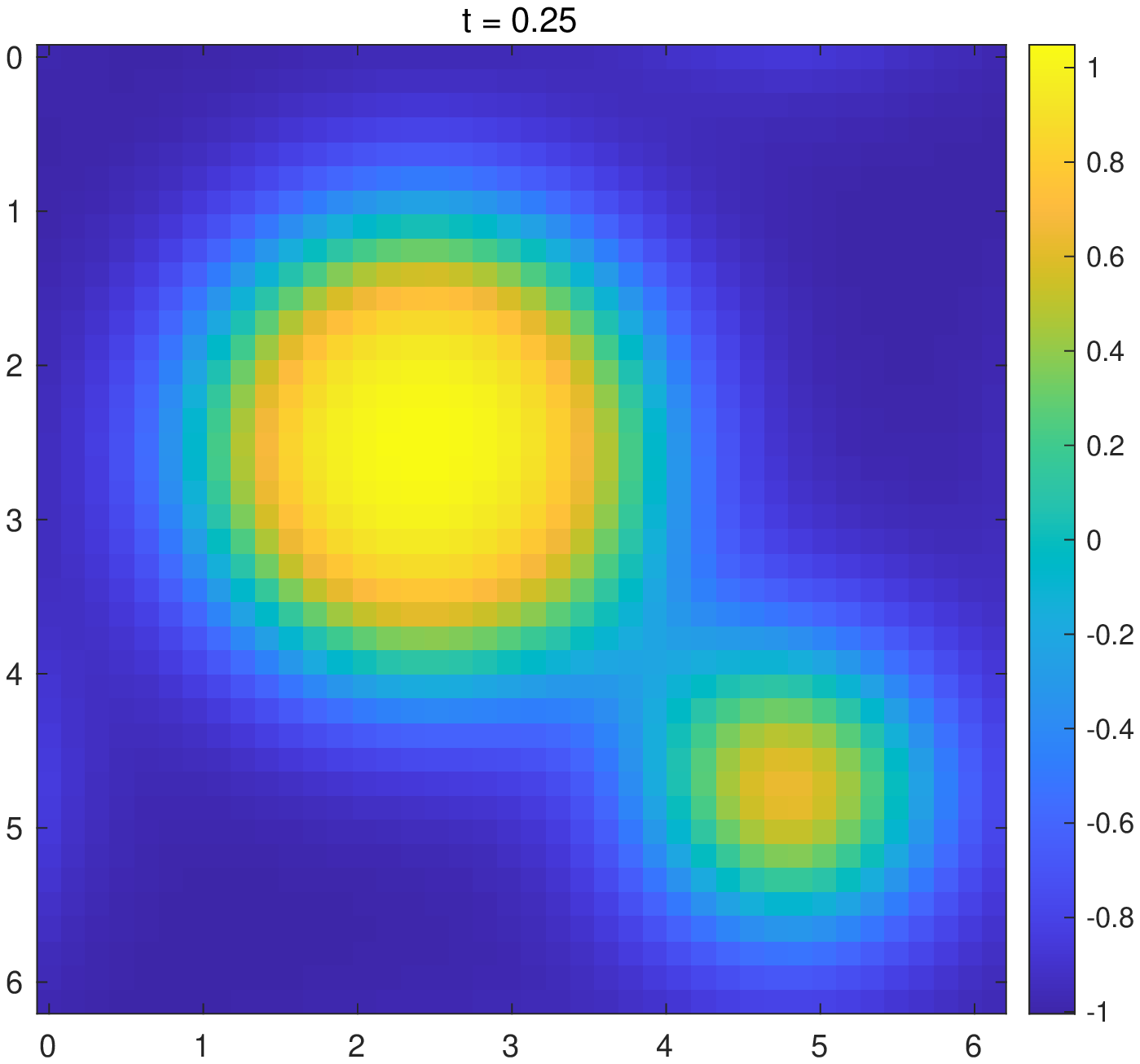}
        \caption*{$t=0.25s$}
    \end{subfigure}%
    \begin{subfigure}{0.22\textwidth}
        \centering
        \includegraphics[width=\linewidth]{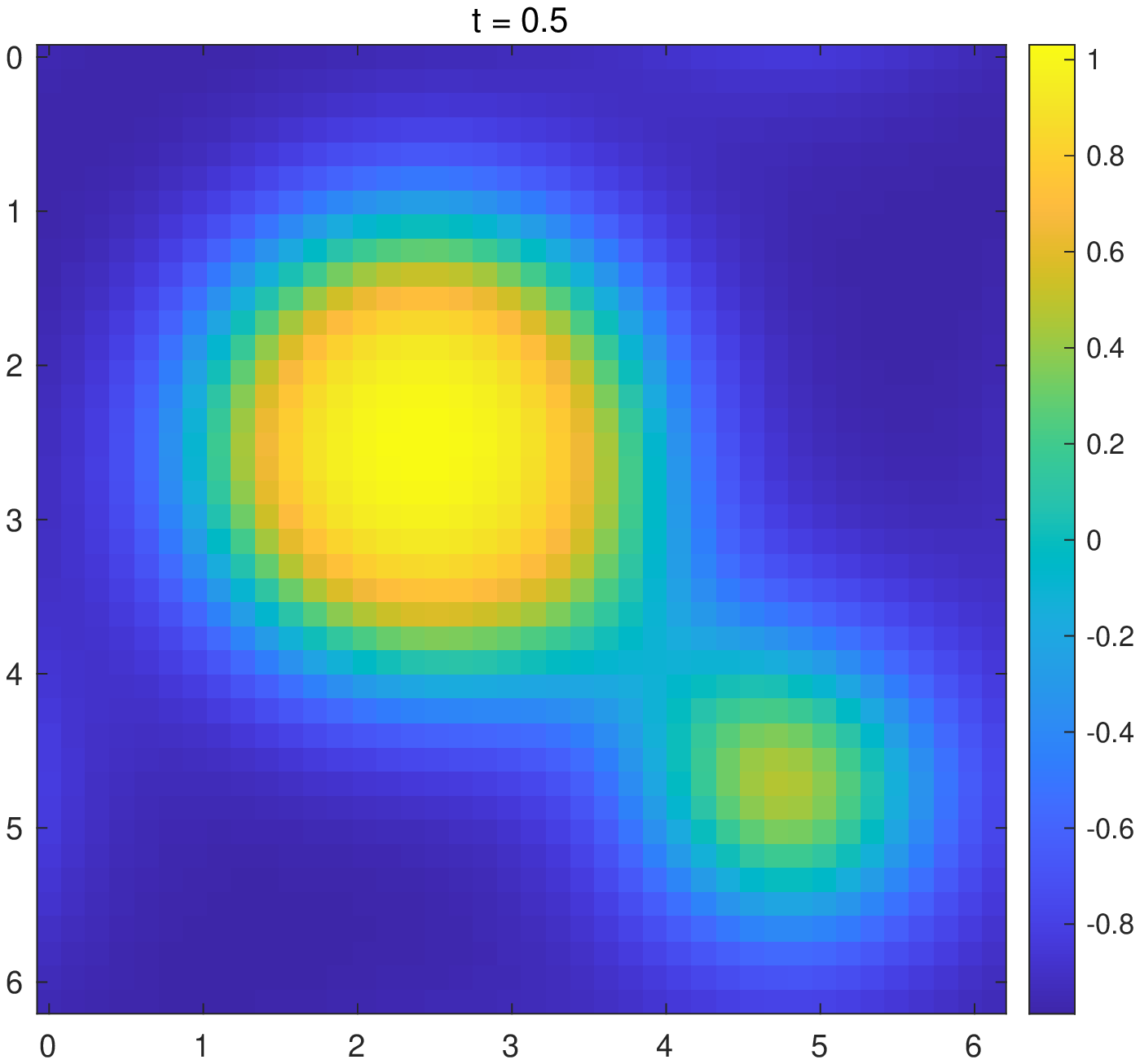}
        \caption*{$t=0.50s$}
    \end{subfigure}
    \begin{subfigure}{0.22\textwidth}
        \centering
        \includegraphics[width=\linewidth]{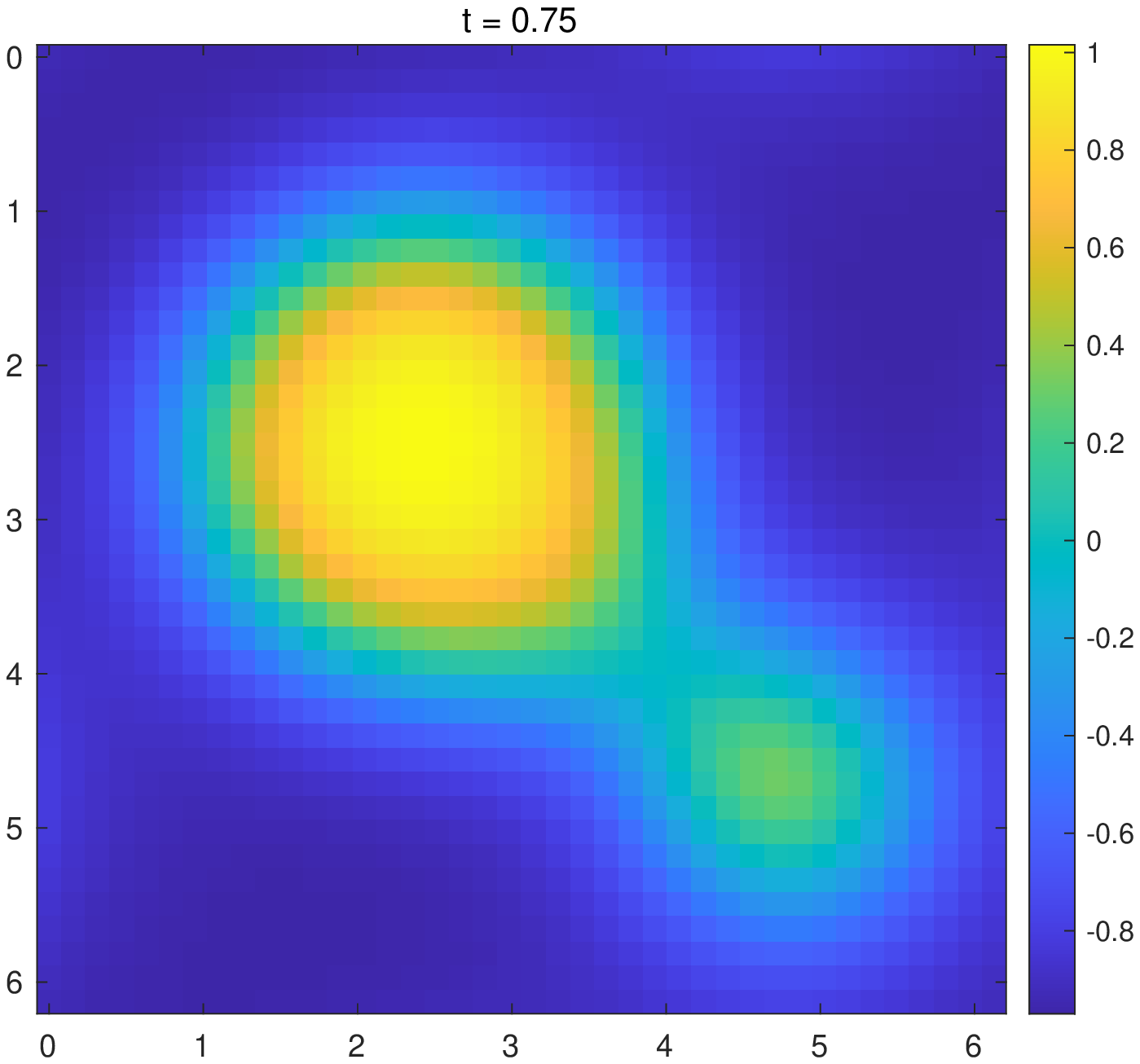}
        \caption*{$t=0.75s$}
    \end{subfigure}
    \begin{subfigure}{0.22\textwidth}
        \centering
        \includegraphics[width=\linewidth]{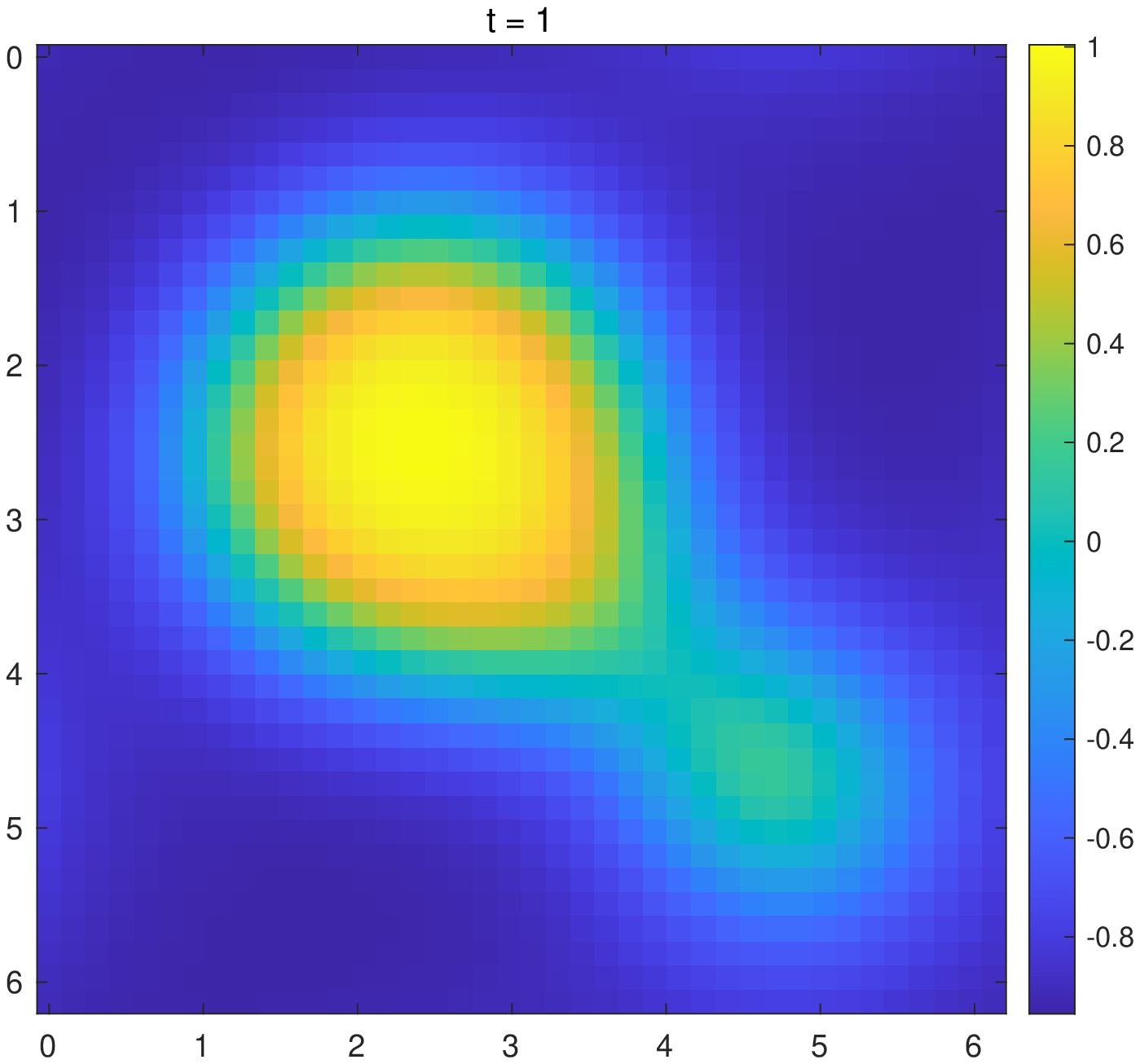}
        \caption*{$t=1.00s$}
    \end{subfigure}
    
    \begin{subfigure}{0.22\textwidth}
        \centering
        \includegraphics[width=\linewidth]{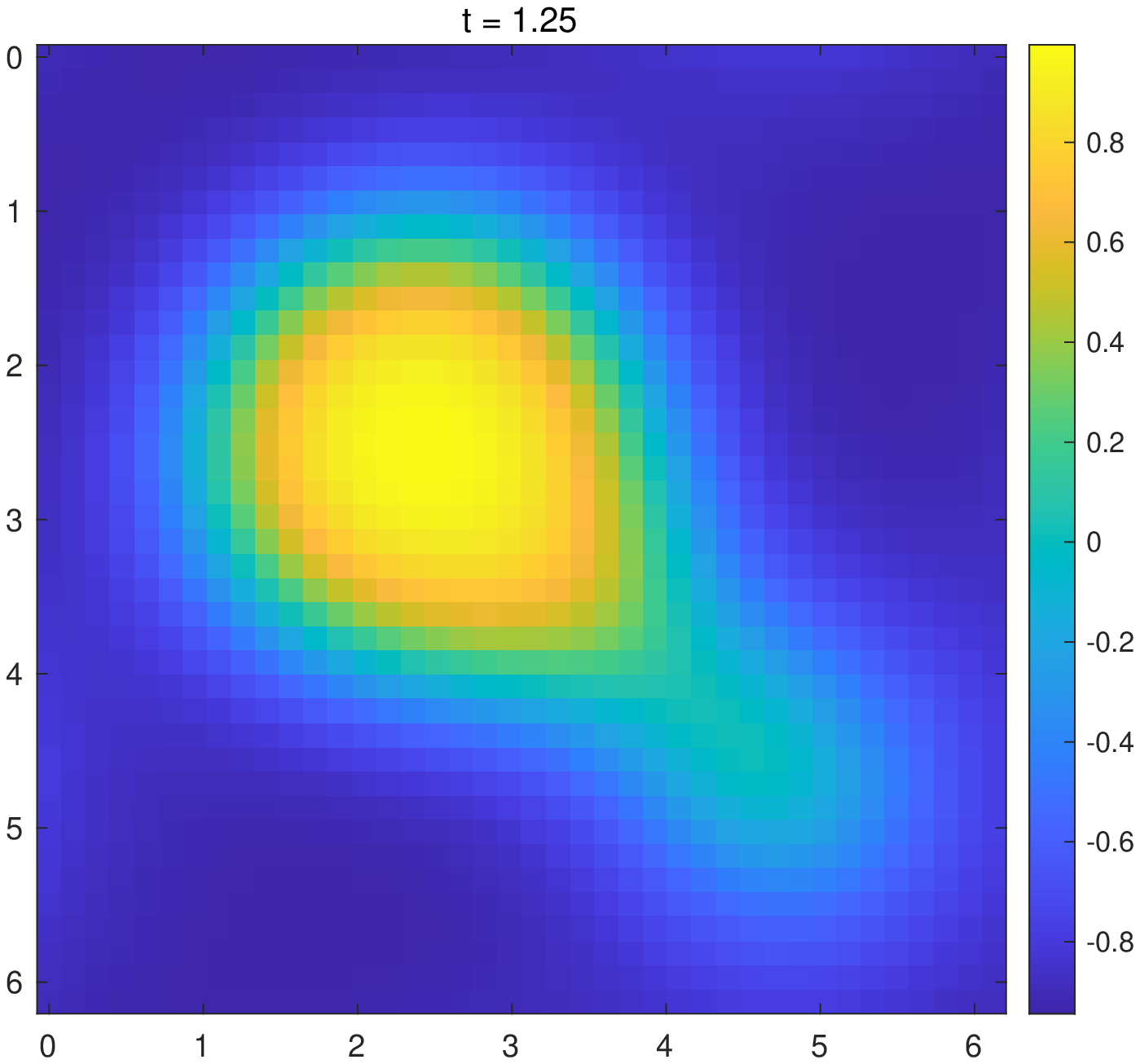}
        \caption*{$t=1.25s$}
    \end{subfigure}%
    \begin{subfigure}{0.22\textwidth}
        \centering
        \includegraphics[width=\linewidth]{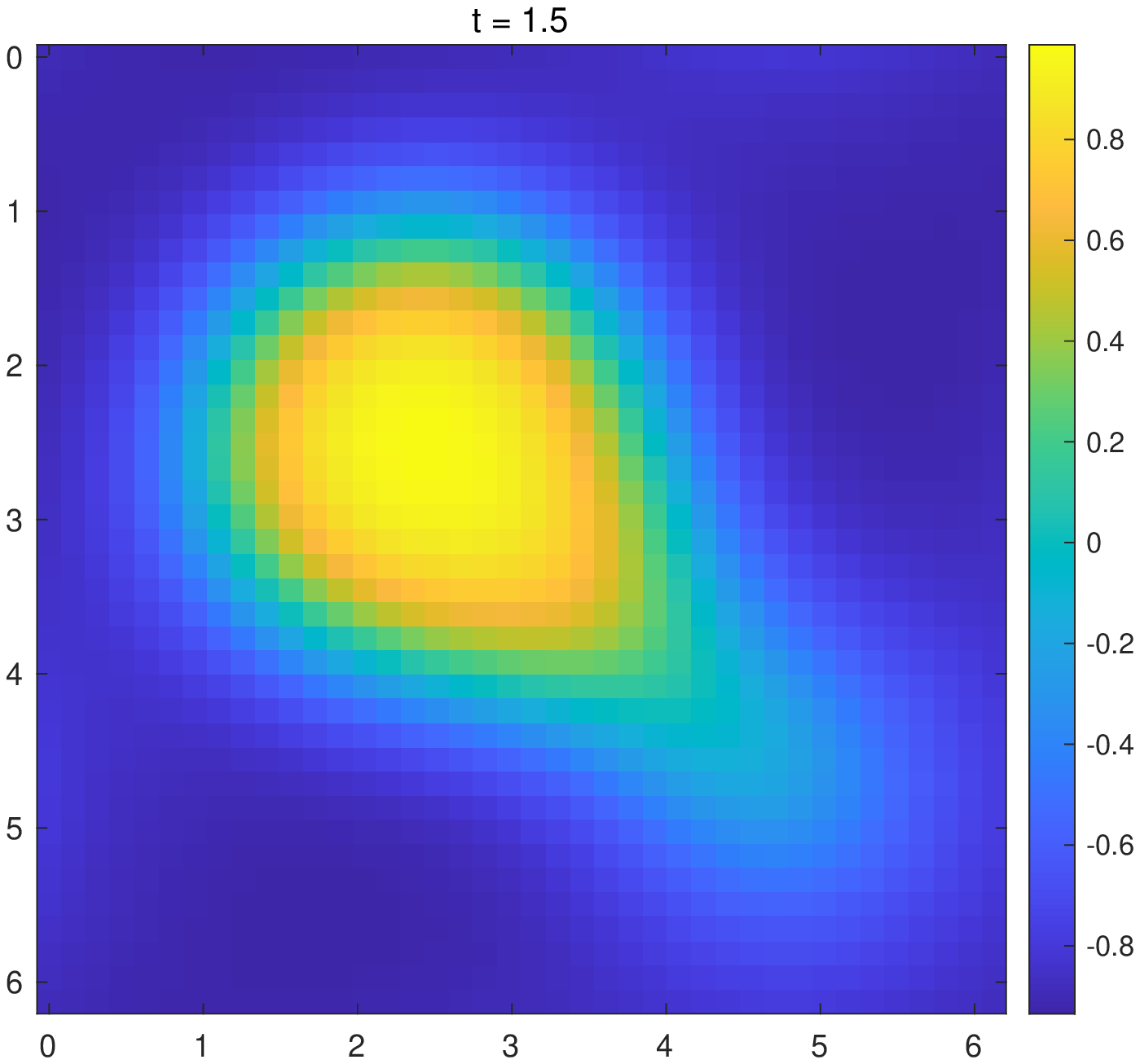}
        \caption*{$t=1.50s$}
    \end{subfigure}
    \begin{subfigure}{0.22\textwidth}
        \centering
        \includegraphics[width=\linewidth]{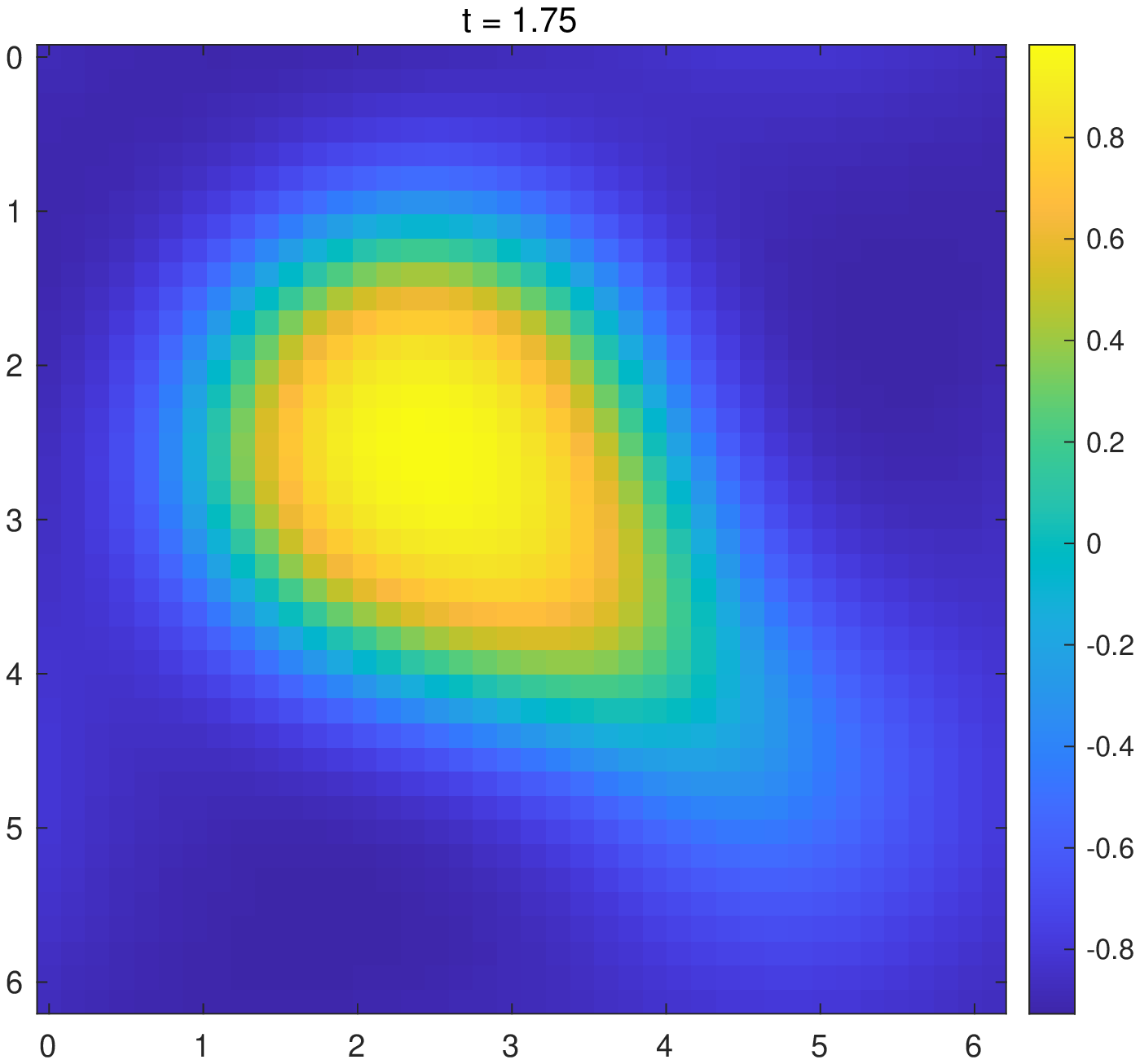}
        \caption*{$t=1.75s$}
    \end{subfigure}
    \begin{subfigure}{0.22\textwidth}
        \centering
        \includegraphics[width=\linewidth]{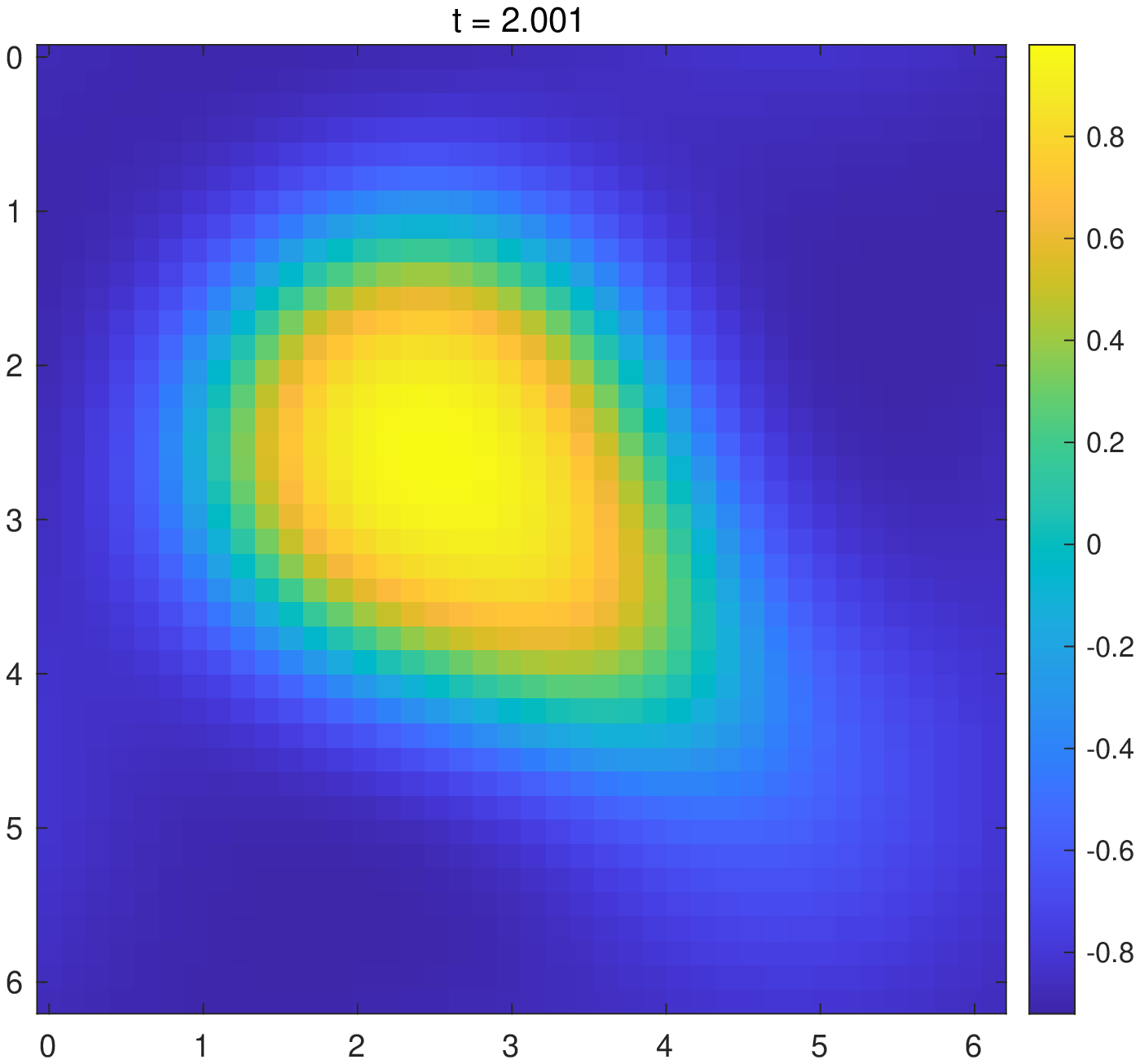}
        \caption*{$t=2.00s$}
    \end{subfigure}

    \begin{subfigure}{0.22\textwidth}
        \centering
        \includegraphics[width=\linewidth]{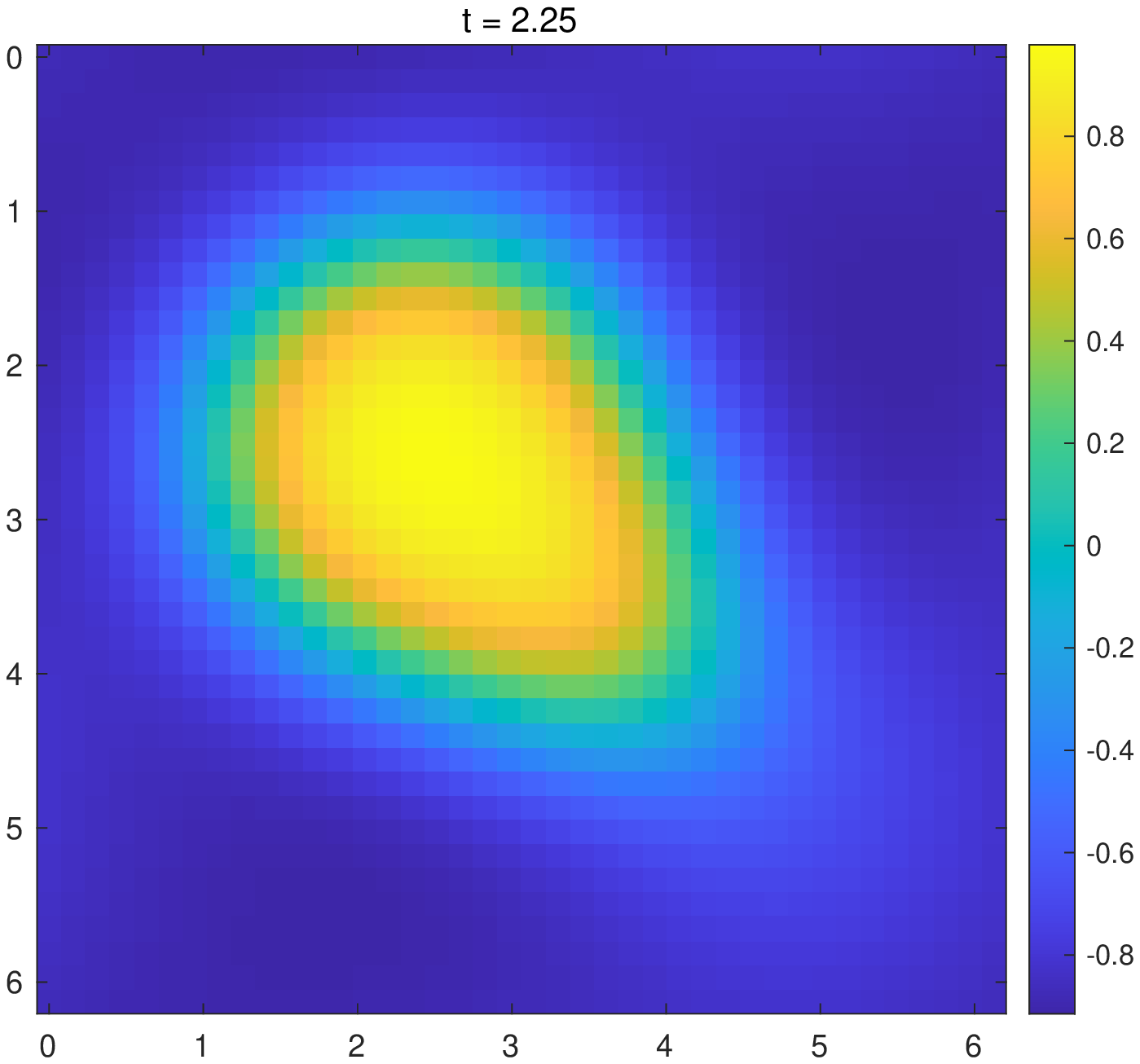}
        \caption*{$t=2.25s$}
    \end{subfigure}%
    \begin{subfigure}{0.22\textwidth}
        \centering
        \includegraphics[width=\linewidth]{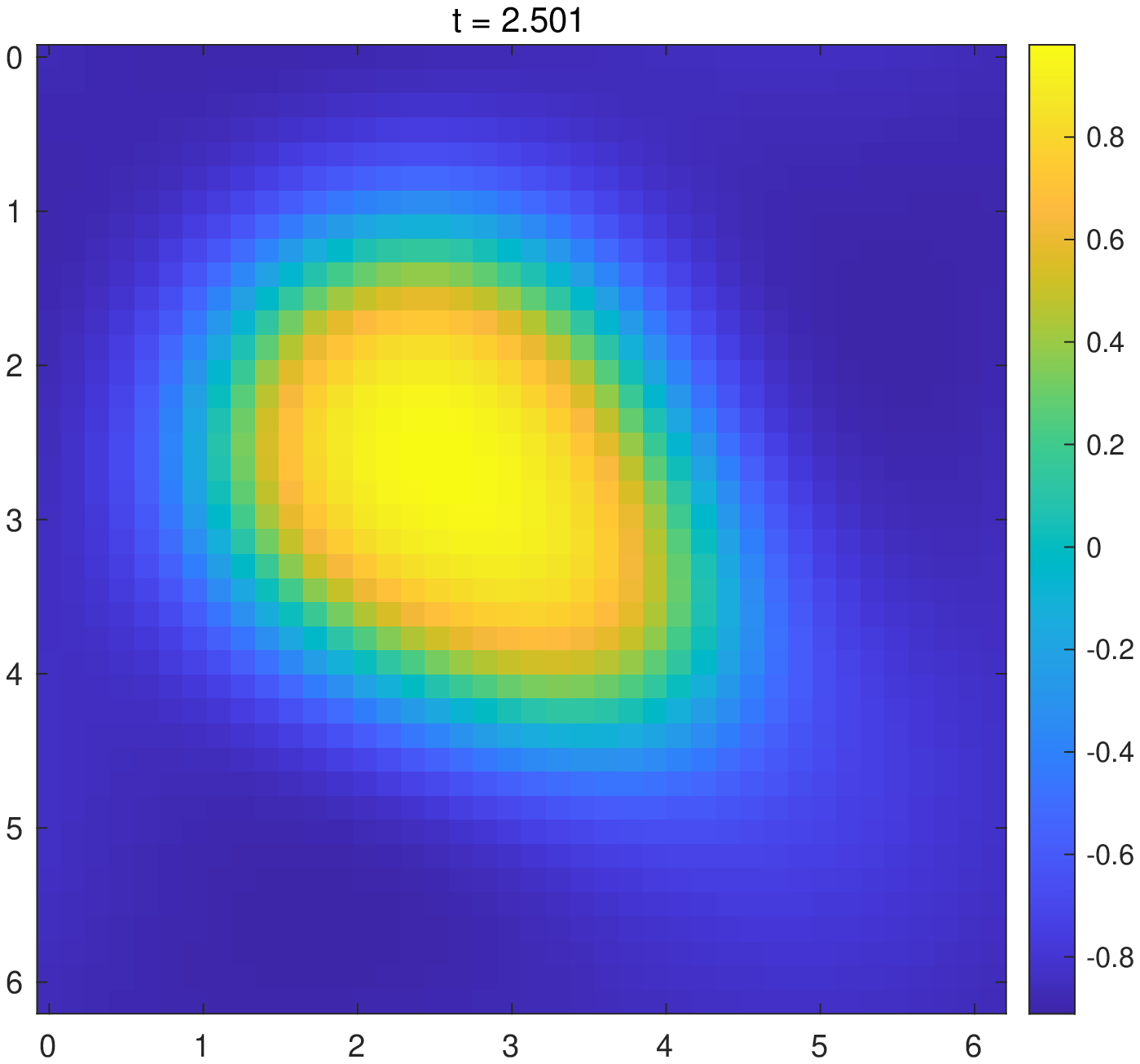}
        \caption*{$t=2.50s$}
    \end{subfigure}
    \begin{subfigure}{0.22\textwidth}
        \centering
        \includegraphics[width=\linewidth]{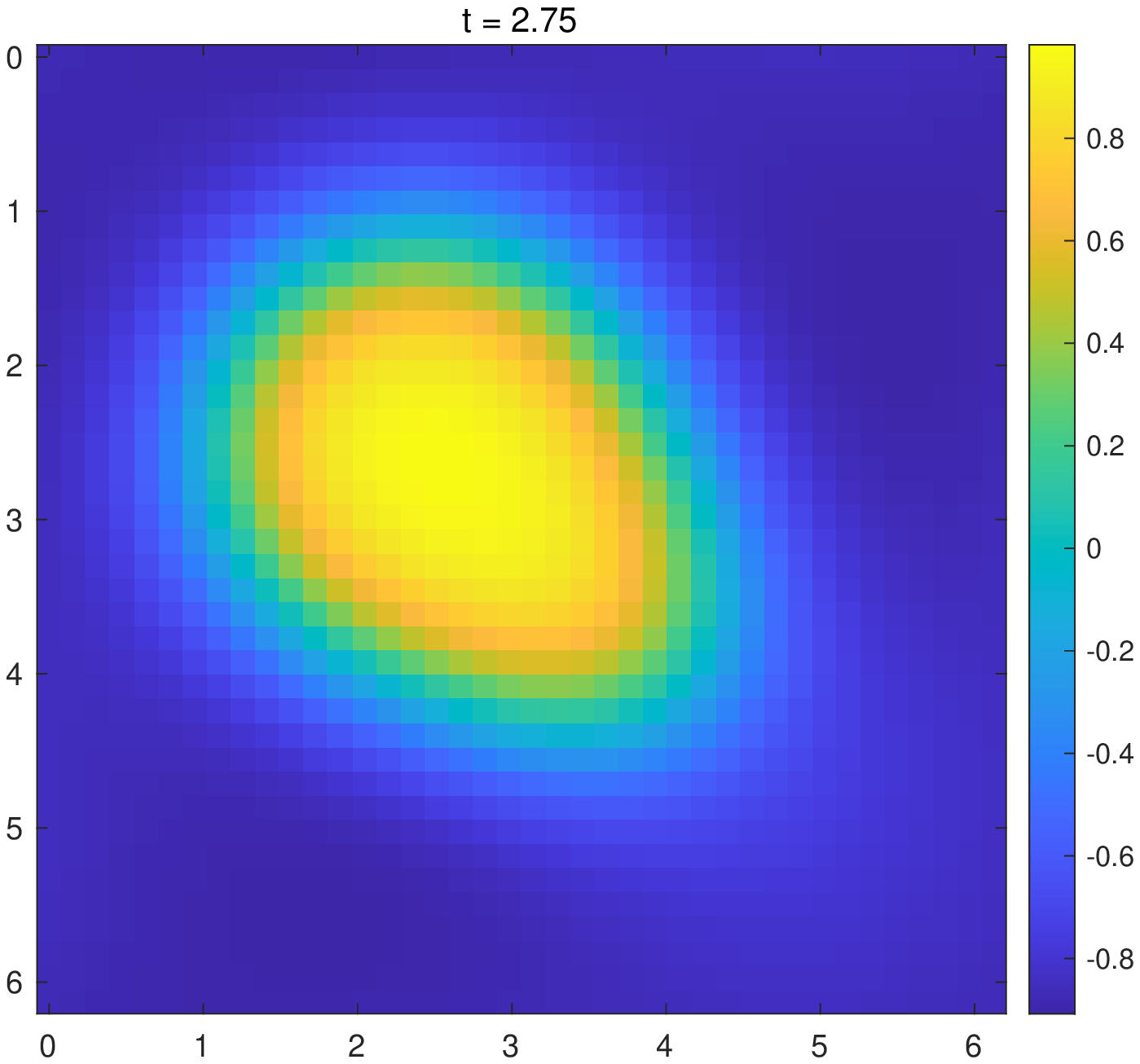}
        \caption*{$t=2.75s$}
    \end{subfigure}
    \begin{subfigure}{0.22\textwidth}
        \centering
        \includegraphics[width=\linewidth]{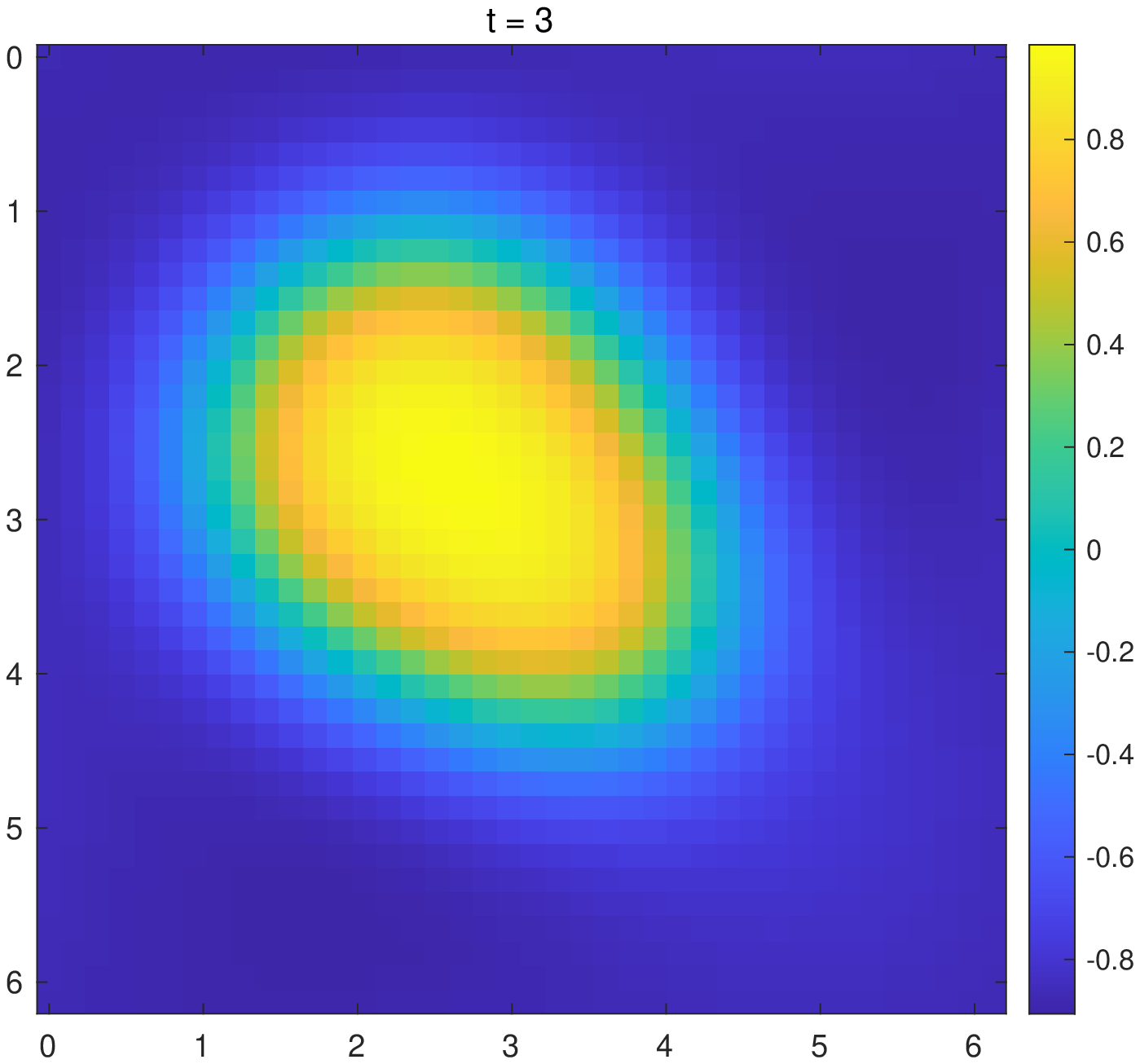}
        \caption*{$t=3.00s$}
    \end{subfigure}

    \caption{Snapshots of $\phi$ taken within the time span $[0,3]$, solving the Cahn-Hilliard equation with initial condition \eqref{eq:coarsening_effect_initial_condition1} by the IEC scheme formulated by the Softplus function with time step size set to be $0.001$.}
    \label{fig:snapshots_CH_IEC}
\end{figure}

Our last exploration entails a new example, which we introduce with the following initial condition:
\begin{equation}
\label{eq:CH_coarsening_effect_rand}
\phi(x,y,0)=0.25+0.4\text{rand}(x,y).
\end{equation}
This numerical experiment reproduces a similar one conducted in \cite{liu2019efficient}. We employ the IEF scheme with the function $g(r)$ set as $r^7$ for our investigations. The time step size adopted is $0.0001$. The dynamic patterns of the mixing of the phase field, starting from a random initial value, can be discerned in Figure \ref{fig:snapshots_CH_IEF}. The observed results, embodying the intriguing interplay of phase transitions, comply nicely with empirical findings.

\begin{figure}
  \centering
  \begin{subfigure}{0.24\textwidth}
    \includegraphics[width=\textwidth]{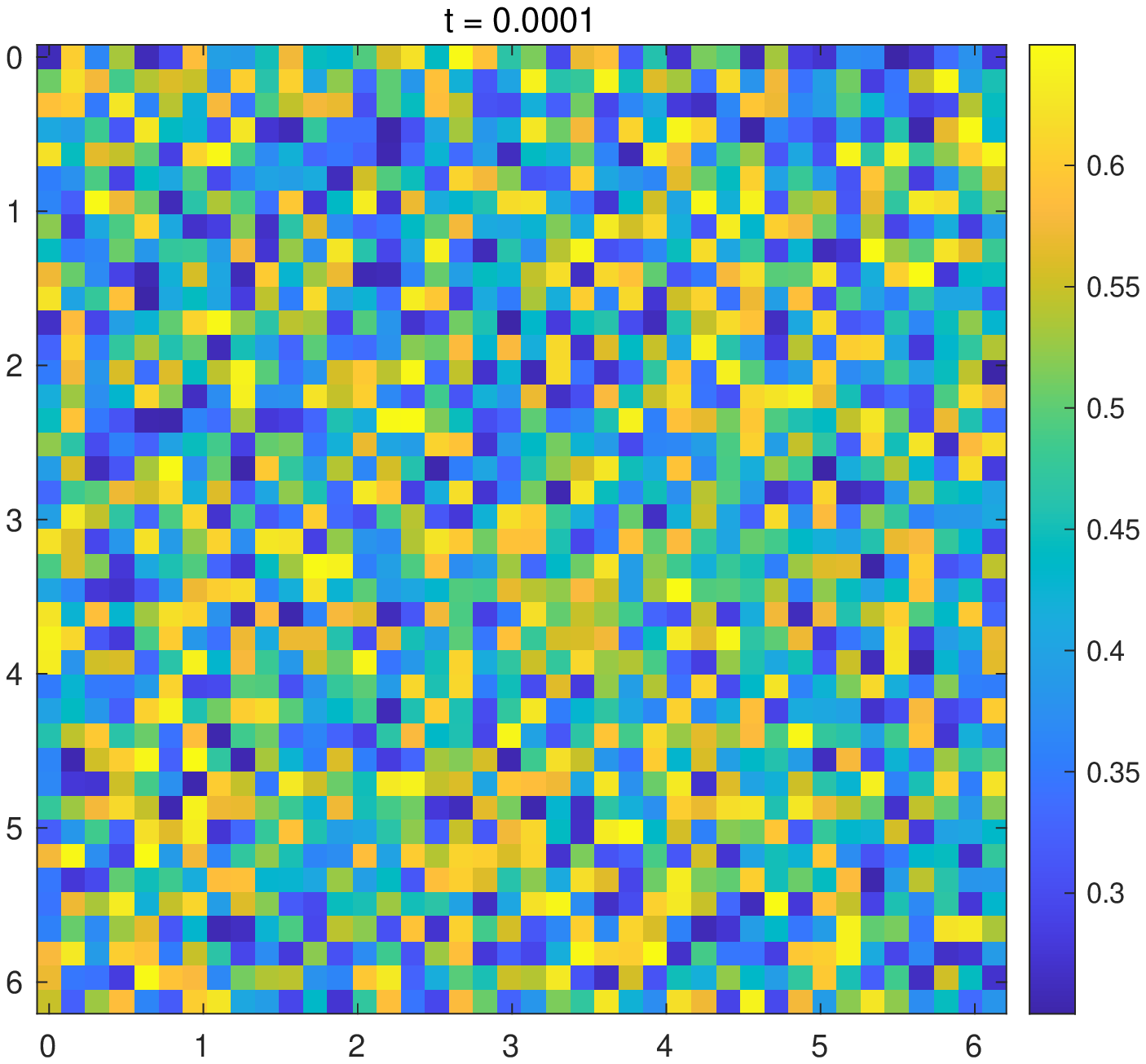}
    \caption{t = 0s}
  \end{subfigure}
  \hfill
  \begin{subfigure}{0.24\textwidth}
    \includegraphics[width=\textwidth]{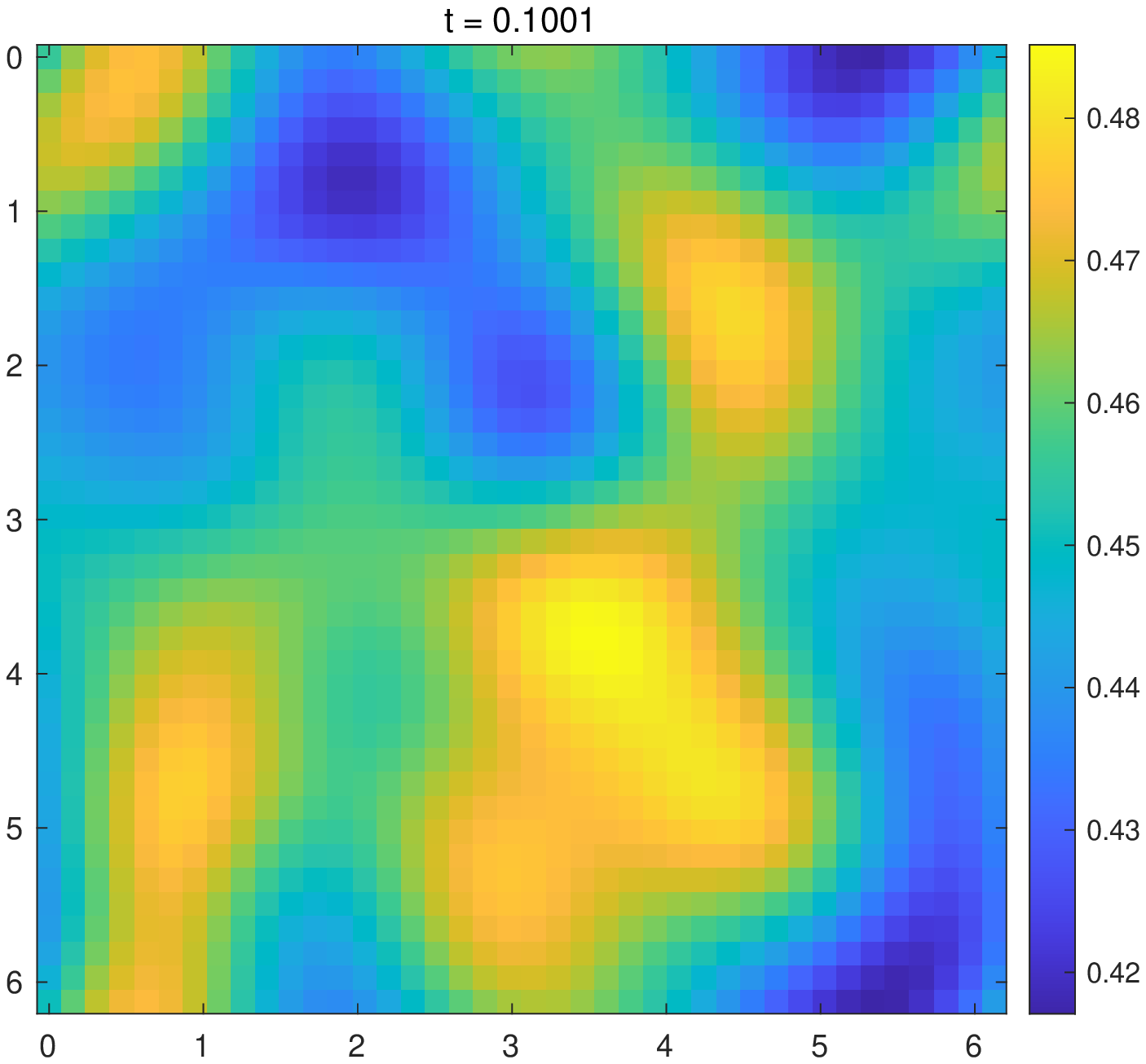}
    \caption{t = 0.1s}
  \end{subfigure}
  \hfill
  \begin{subfigure}{0.24\textwidth}
    \includegraphics[width=\textwidth]{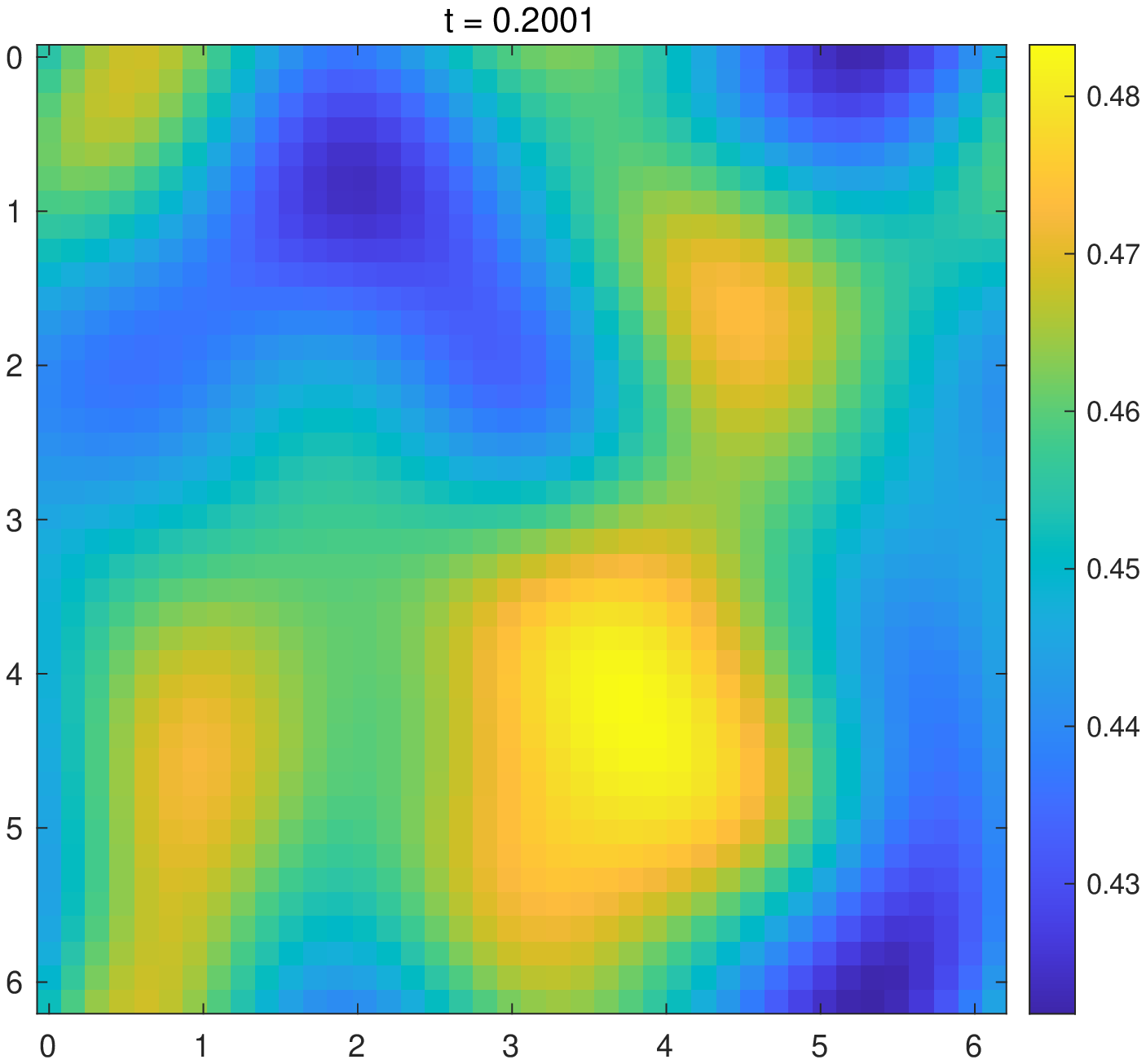}
    \caption{t = 0.2s}
  \end{subfigure}
  \hfill
  \begin{subfigure}{0.24\textwidth}
    \includegraphics[width=\textwidth]{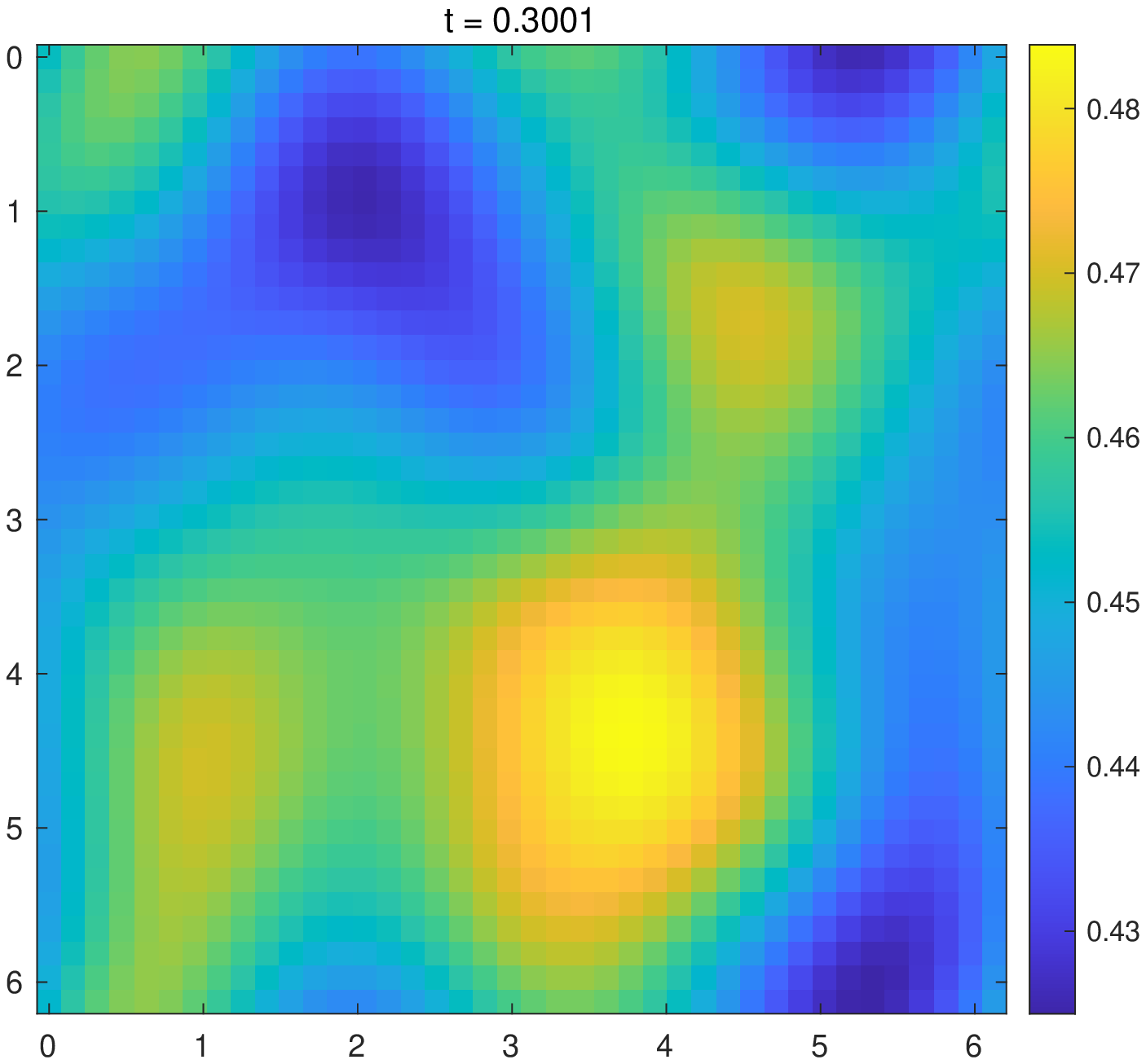}
    \caption{t = 0.3s}
  \end{subfigure}

  \vspace{1em}  

  \begin{subfigure}{0.24\textwidth}
    \includegraphics[width=\textwidth]{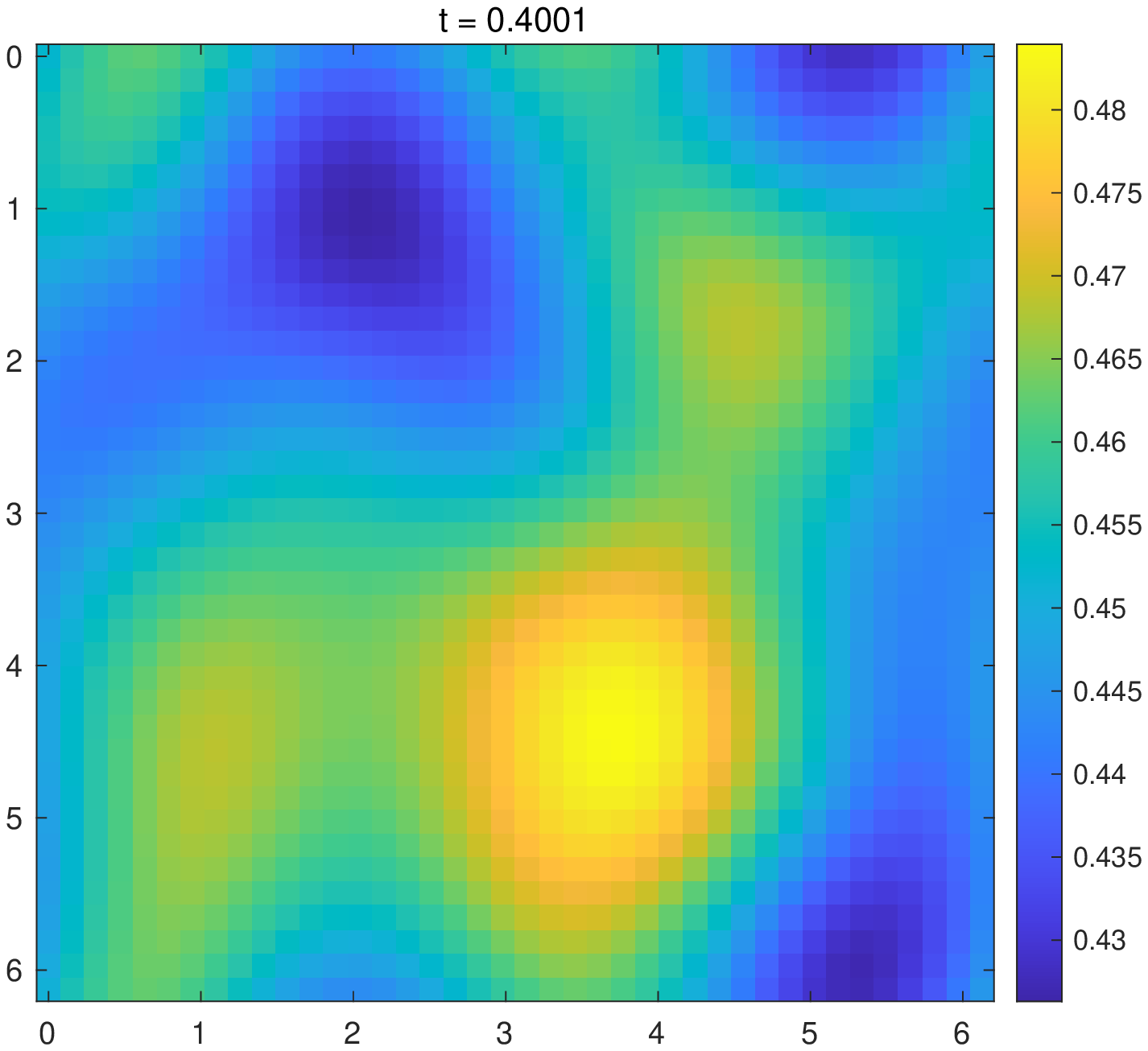}
    \caption{t = 0.4s}
  \end{subfigure}
  \hfill
  \begin{subfigure}{0.24\textwidth}
    \includegraphics[width=\textwidth]{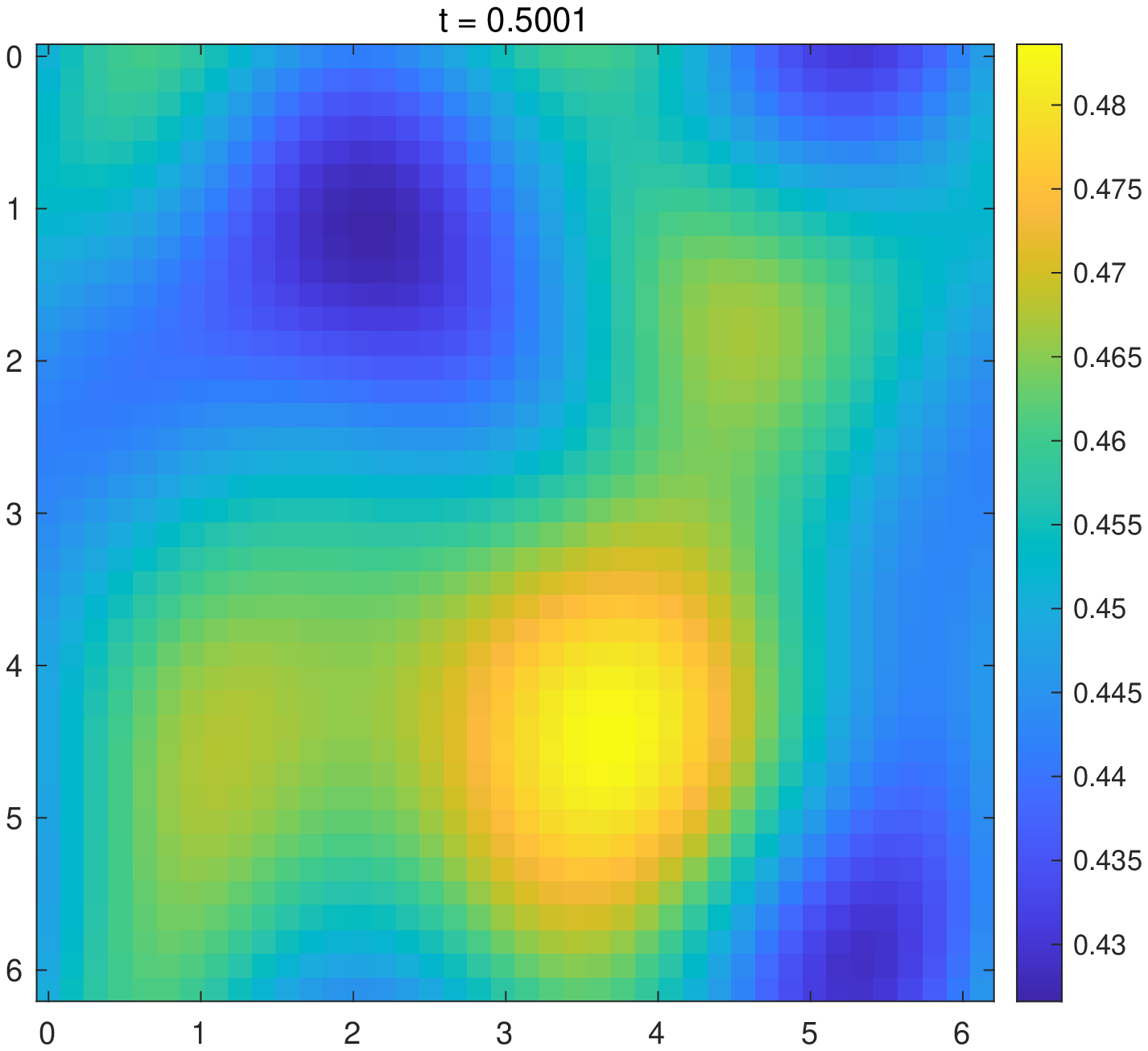}
    \caption{t = 0.5s}
  \end{subfigure}
  \hfill
  \begin{subfigure}{0.24\textwidth}
    \includegraphics[width=\textwidth]{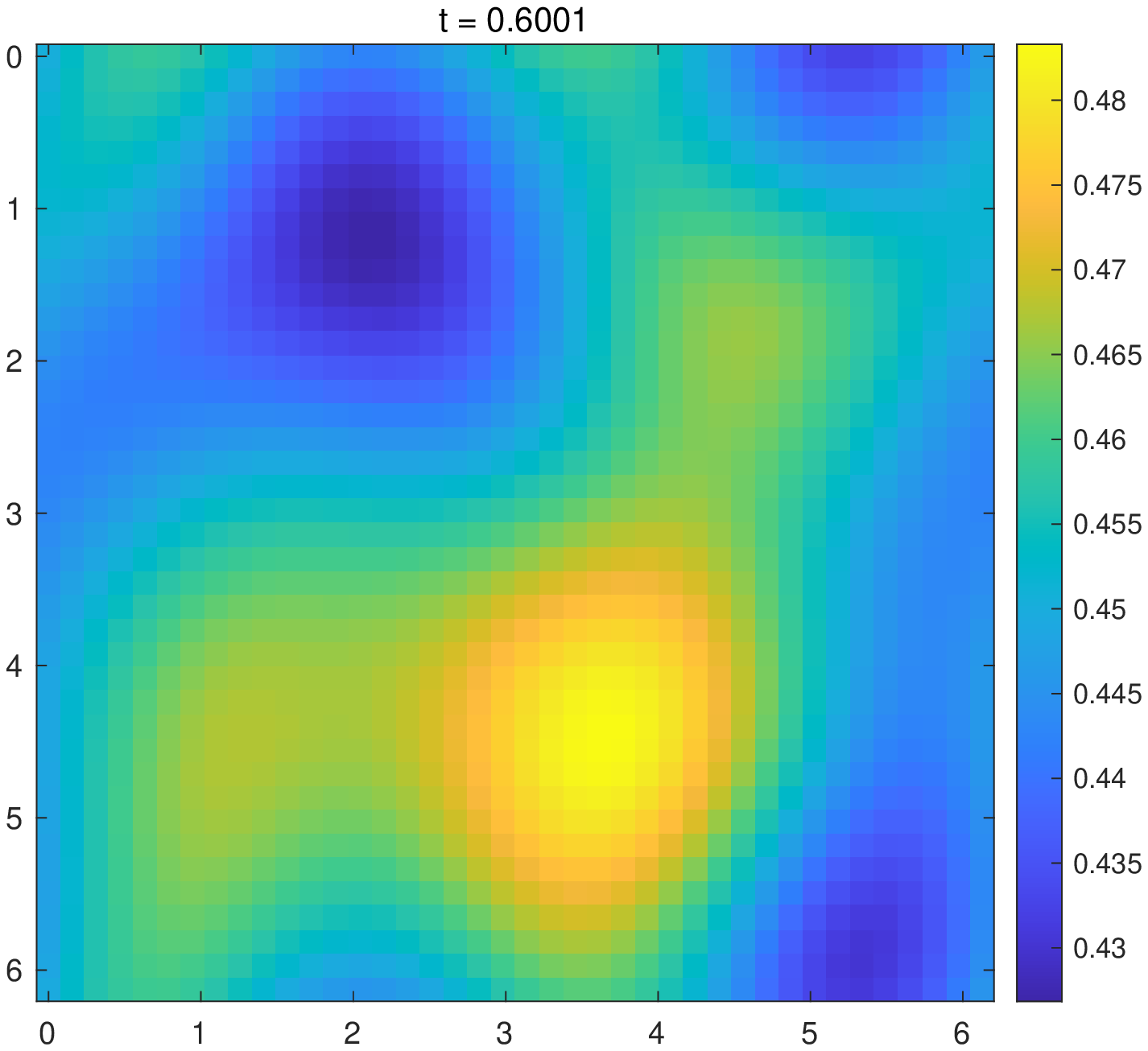}
    \caption{t = 0.6s}
  \end{subfigure}
  \hfill
  \begin{subfigure}{0.24\textwidth}
    \includegraphics[width=\textwidth]{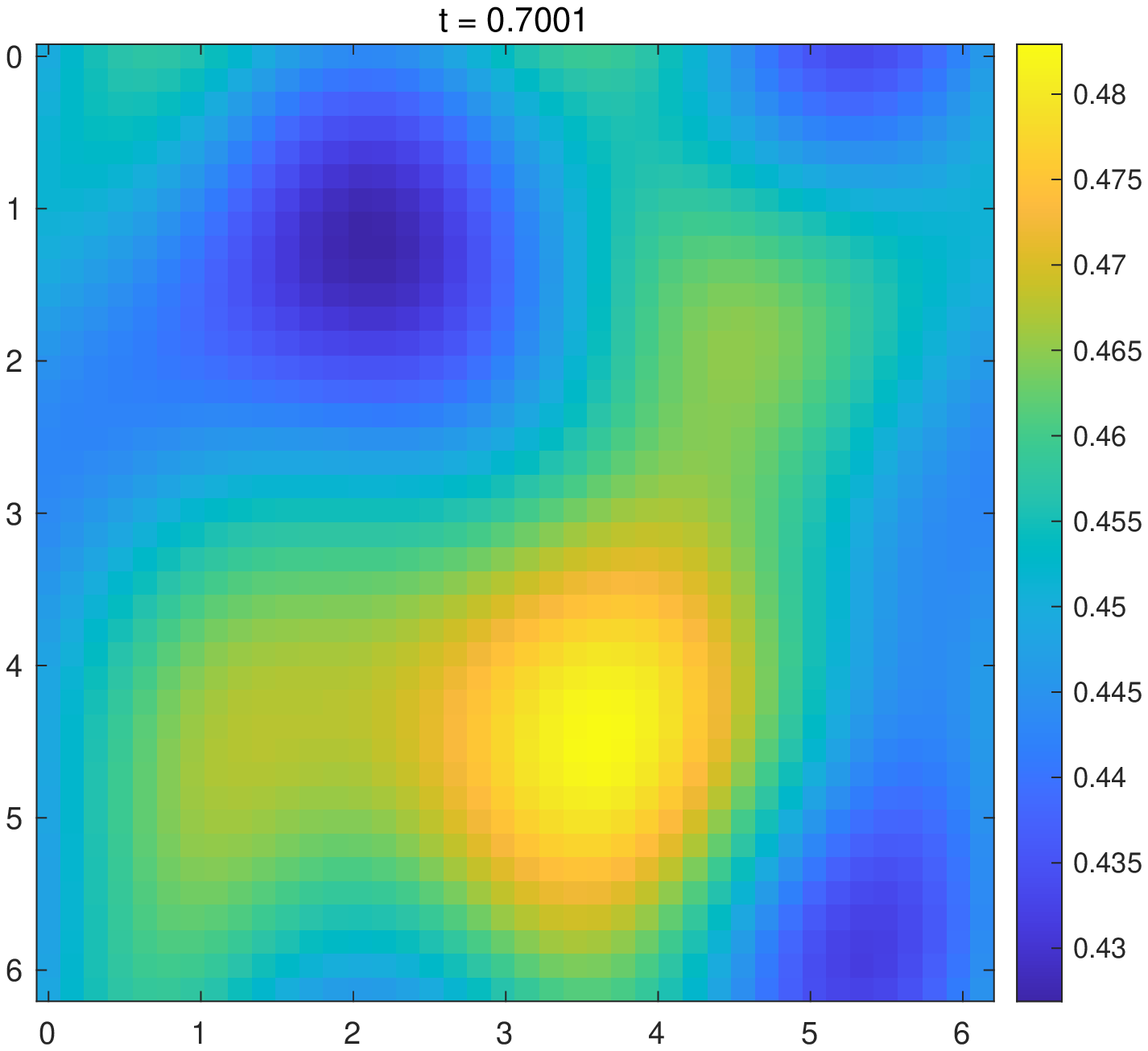}
    \caption{t = 0.7s}
  \end{subfigure}

  \vspace{1em}  

  \begin{subfigure}{0.24\textwidth}
    \includegraphics[width=\textwidth]{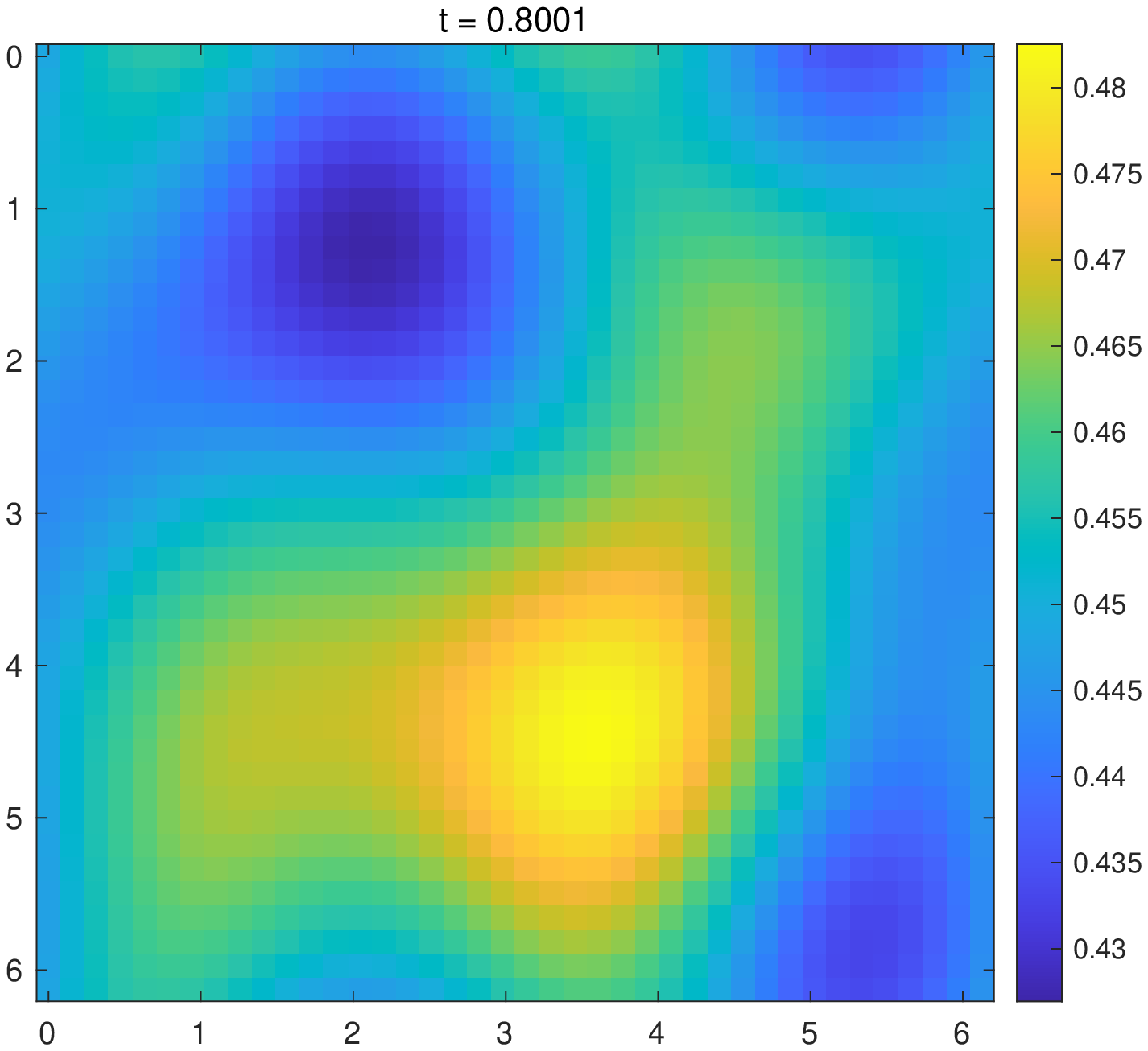}
    \caption{t = 0.8s}
  \end{subfigure}
  \hfill
  \begin{subfigure}{0.24\textwidth}
    \includegraphics[width=\textwidth]{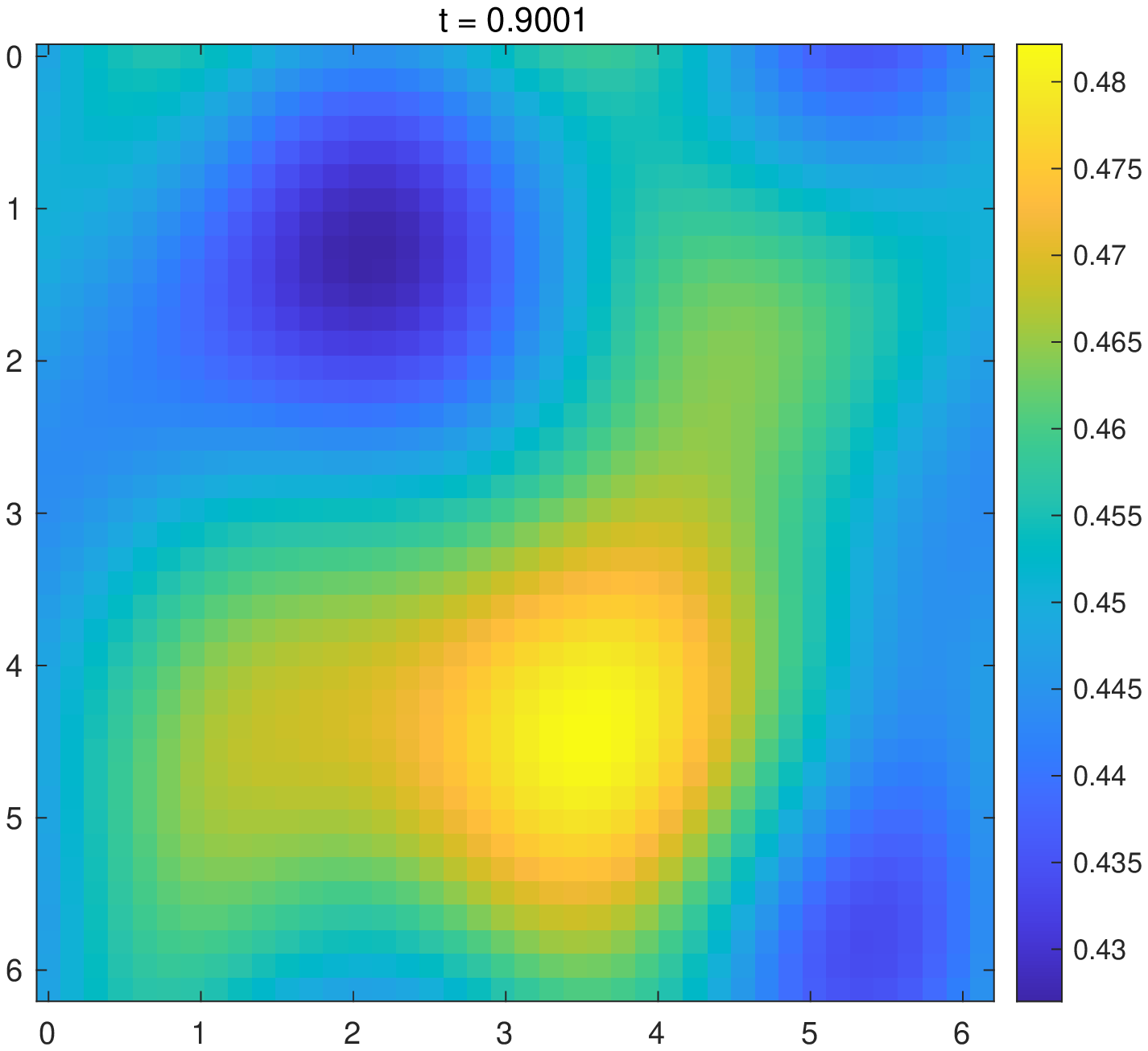}
    \caption{t = 0.9s}
  \end{subfigure}
  \hfill
  \begin{subfigure}{0.24\textwidth}
    \includegraphics[width=\textwidth]{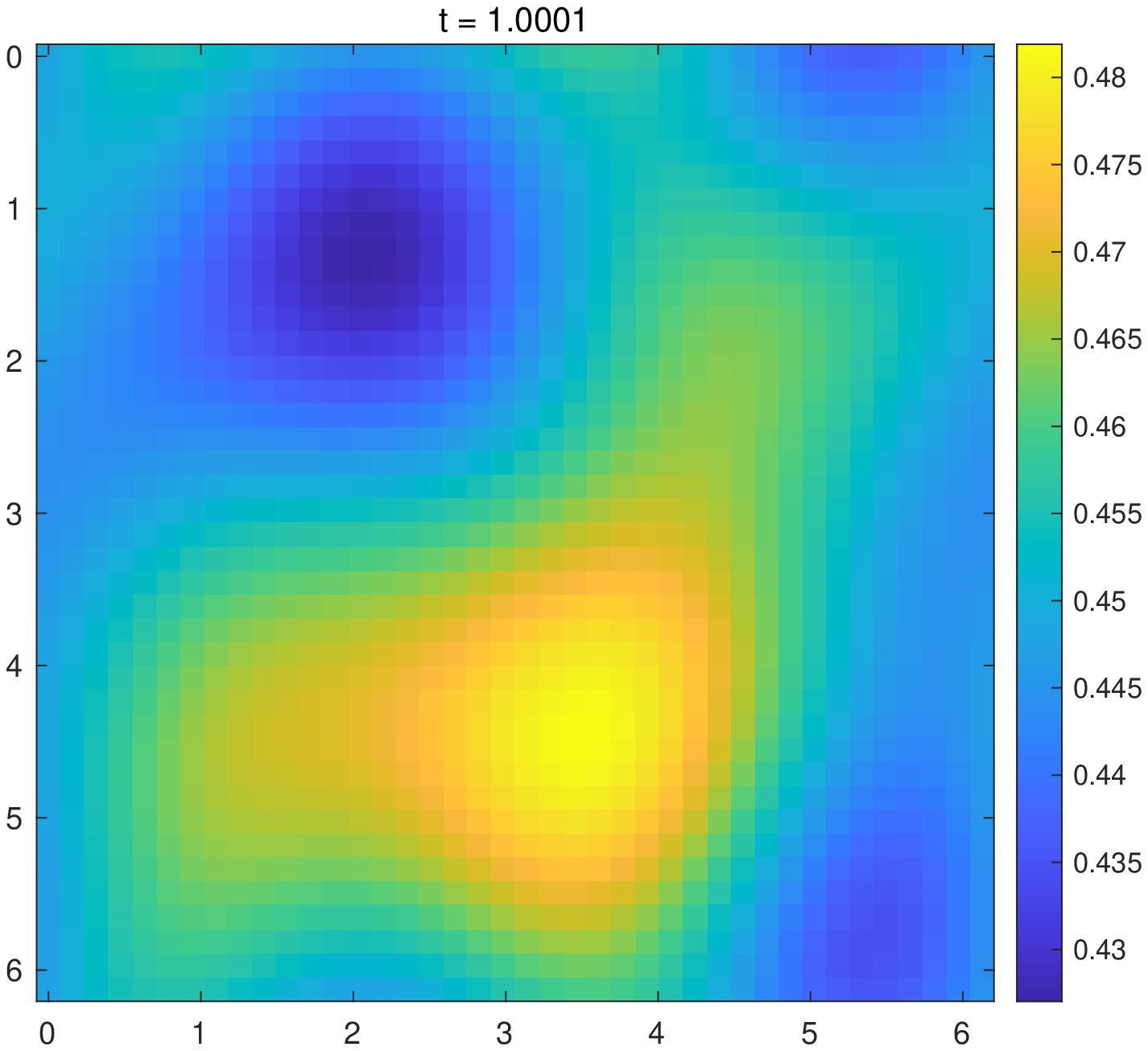}
    \caption{t = 1.0s}
  \end{subfigure}
  \hfill
  \begin{subfigure}{0.24\textwidth}
    \includegraphics[width=\textwidth]{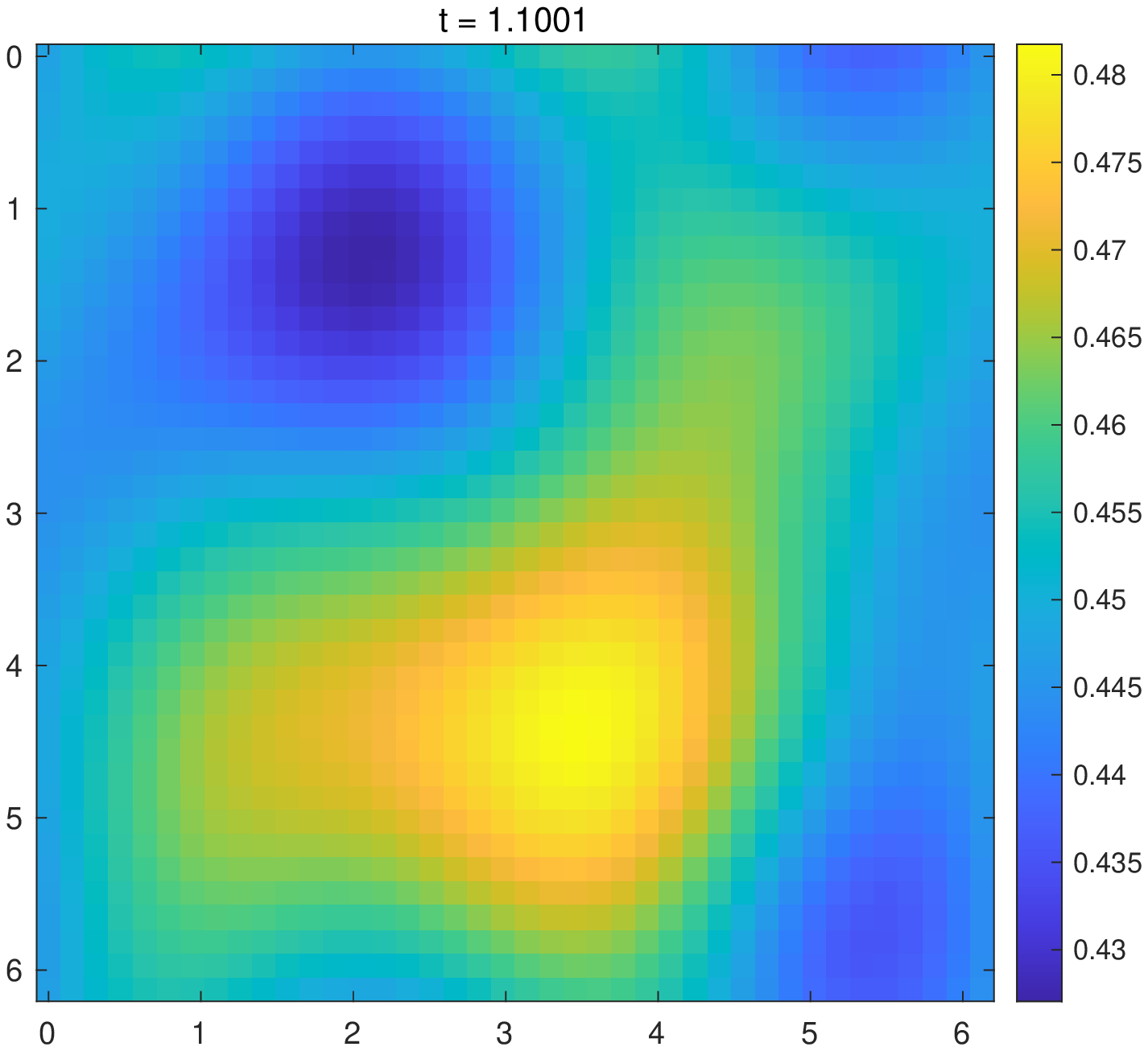}
    \caption{t = 1.1s}
  \end{subfigure}

  \vspace{1em}  

  \caption{Snapshots of $\phi$ taken within the time span $[0,1.1]$, solving the Cahn-Hilliard equation with initial condition \eqref{eq:CH_coarsening_effect_rand} by the IEF scheme formulated by the $g(r) = r^7$ with time step size set to be $0.0001$.}
   \label{fig:snapshots_CH_IEF}
\end{figure}

\section{Extension to the SAV scheme}

In the forthcoming discussion, we present the extensions of the frameworks of the IEC and the IEF scheme onto the Scalar Auxiliary Variable method (SAV). These schemes will share similar properties as our generalization of the IEQ method in the sense that they will also lead to energy-stable linear numerical schemes.

\subsection{A brief review of the SAV method}

The Scalar Auxiliary Variable (SAV) method, initially presented in \cite{shen2018scalar, shen2019new} and subsequently refined \cite{akrivis2019energy,antoine2021scalar, fu2021linear, jiang2022improving, shen2018convergence, shen2020ieq}, is a robust tool analogous to the IEQ method. It is widely employed to construct efficient and accurate time discretization schemes for a broad spectrum of gradient flows. In terms of formulation, it takes a notable resemblance to the IEQ method. Both introduce an auxiliary variable to replace the original energy and then construct an energy-stable linear scheme based on this reformulation. A distinct differential point emerges in the scope of the auxiliary variables utilized by the two methodologies. While the IEQ approach employs a variable over both time and space (represented as $r=\sqrt{F(\phi)+A_1}$ where $F$ is the energy density function outlined in \eqref{eq:gradient_flow_system}), the SAV method utilizes a variable defined solely over time ($r=\sqrt{E_1(\phi)+A_2}$), represented as the square root of $E_1(\phi)$, which is defined as $E_1(\phi)=\int_\Omega F(\phi)\,dx$. $A_2$ here is a constant to ensure the auxiliary variable is well-defined given that $E_1$ is bounded from below, which is necessary for the free energy term to be physically sound. Specifically, as initially outlined in \cite{shen2019new}, a first-order SAV scheme can be carried out as;

\begin{subequations}
\label{eq:SAV_scheme}
\begin{equation}
\label{eq:gf_numer_SAV}
\Dtphi=\mathcal{G} \munew,
\end{equation}
\begin{equation}
\label{eq:munew_SAV}
\munew = -\Delta \phinew+\frac{\rnew}{ \sqrt{E_1(\phiold)}+A_2}F^n,
\end{equation}
\begin{equation}
\label{eq:drdt_SAV}
\rnew-\rold=\frac{1}{c'\left(r(\phi)\right)}\int_\Omega F^n(\phinew-\phiold)\,dx.
\end{equation}
\end{subequations}
We take $F^n = F(\phiold)$ here. 
Various numerical experiments have shown the efficiency and practicability of this method. In addition, as it has been stated in \cite{shen2018scalar, shen2019new}, the SAV methodology inherits all the advantages of the IEQ method and additionally enjoys the benefits of a simpler implementation and the requirement for only $E_1(\phi)$ to be bounded from below (instead of the need for $F(\phi)$ also to be bounded from below). Given the commonality in the core construction principles of these methods, it is feasible to transpose our ideas onto the SAV method, enabling further generalizations. We will present it in the following.

\subsection{The C-SAV method}

 In the following, we will propose a variant of the standard SAV method, the Convex Scalar Auxiliary Variable method, abbreviated as C-SAV. Guided by the parallel adaptations we made for the IEC method, a natural proposition is to supplant $r^2$ with an alternative auxiliary function. In this context, we will adhere to the same assumptions we have put forth for the IEC scheme with respect to the auxiliary function $c$; namely, $c:\mathbb{R}\to\mathbb{R}$ is taken as a smooth, L-smooth, and convex function exhibiting a monotonic increase, with $\mathbb{R}^+$ included within the function's range. Then by introducing a scalar auxiliary variable $r(t)$ such that
 \begin{equation*}
     c\left(r(t)\right)=E_1(\phi)+A_2,
 \end{equation*}
system \eqref{eq:gradient_flow_system} can be reformulated as
\begin{subequations}
    \label{eq:gradient_flow_reformulated_system_C-SAV}
    \begin{equation}
        \label{eq:phi_t_reformulated}
        \frac{\partial \phi}{\partial t} = \mathcal{G}\mu,
    \end{equation}
    \begin{equation}
        \label{eq:mu_reformulate}
        \mu = -\Delta \phi + \frac{c'\left(r(t)\right)}{c'\left(c^{-1}\left(E_1(\phi)+A_2\right)\right)}\frac{\delta F}{\delta \phi},
    \end{equation}
    \begin{equation}
        \label{eq: r_t_reformulate}
        r_t = \frac{1}{c'\left(c^{-1}\left(E_1(\phi)+A_2\right)\right)}\int_\Omega\frac{\delta F}{\delta \phi}\phi_t,
    \end{equation}
\end{subequations}

 From this, we propose the following variant of the IEC scheme, through the application of the same conceptual framework to the SAV method:
\begin{subequations}
\label{eq:C_SAV_scheme}
    \begin{equation}
    \label{eq:gf_numer_CSAV}
        \Dtphi=\mathcal{G} \munew,
    \end{equation}
    \begin{equation}
    \label{eq:munew_CSAV}
        \munew = -\Delta \phinew+\frac{c'(\rold)+\alpha L(\rnew-\rold)}{ c'\left(r(\phiold)\right)}F^n,
    \end{equation}
    \begin{equation}
    \label{eq:drdt_CSAV}
        \rnew-\rold=\frac{1}{c'\left(r\left(\phiold\right)\right)}\int_\Omega F^n(\phinew-\phiold)\,dx.
    \end{equation}
\end{subequations}
By defining modified energy as $E_{CS}^n = \frac{1}{2}|\nabla \phiold|^2+c(\rold)$, we can deduce, through parallel arguments to Theorem \ref{thm:IEC_energy_stability}, that this discrete energy demonstrates a monotonic decrease. We can similarly extend the IEF formulation to the SAV methodology, thereby constructing linear, energy-stable numerical schemes. We omit the details for the sake of brevity, but we encourage interested readers to explore this subject further. 

\section{Concluding remarks}

In this work, we have introduced two novel numerical methodologies: the Invariant Energy Convexification (IEC) and the Invariant Energy Functionalization (IEF) methods. These methods generalize the well-known Invariant Energy Quadratization (IEQ) method. The inherent ability of these proposed techniques to construct linear energy-stable numerical schemes, a crucial feature inherited from the IEQ method, underscores their significance. 

The numerical experiments on the Allen-Cahn and Cahn-Hilliard equations substantiated the theoretical properties deduced for both methods in practical implementations. The results conclusively demonstrated that both methods are robust and provide energy stability. Moreover, if an appropriate function is meticulously selected and carefully chooses the parameters in the method, the IEC method exhibits the potential to outperform the standard IEQ method. 

As we conclude this manuscript, we wish to underscore several notable points and delineate potential lanes for future research related to this work:

\begin{enumerate}
    \item In the extant version of the IEC method, the selection of uniformly L-smooth functions, such as the Softplus function, is imperative for designing the corresponding numerical scheme in theory. However, in our practical implementations, we observed that numerous functions, which are only locally L-smooth, can also facilitate the development of numerical schemes with a commendable performance. It could be attributed to the energy bound of the auxiliary variable, which may consequently result in a bound of the function's second derivative. The necessity for theoretical justification of this aspect needs to be considered.

    \item The accuracy of the numerical solutions is influenced by choice of the auxiliary function $c(x)$ and the parameter $\alpha$, especially when the time discretization is coarse. It is evident from our discussion in Section \ref{sec:Numerics}. Therefore, deriving a theoretical justification to find the optimal formulation for specific problems is challenging and immensely useful.

    \item Our numerical results reveal that both methods lead to convergence to the exact solution with a desirable convergence rate. However, contrary to the IEQ method, the convergence analysis of the IEC and the IEF methods may present formidable challenges due to the non-quadratic structure of their modified energy. It necessitates using specific analytical tools for the auxiliary functions used in their formulation—this facet of our work beckons further investigation.

    \item It is worth noting that the challenges and problems mentioned above are also relevant when considering extending these methods to the Scalar Auxiliary Variable (SAV) method.

\end{enumerate}

\bibliographystyle{abbrv}
\bibliography{reference}

\begin{thebibliography}{10}

\bibitem{akrivis2019energy}
G.~Akrivis, B.~Li, and D.~Li.
\newblock Energy-decaying extrapolated rk--sav methods for the allen--cahn and
  cahn--hilliard equations.
\newblock {\em SIAM Journal on Scientific Computing}, 41(6):A3703--A3727, 2019.

\bibitem{allen1979microscopic}
S.~M. Allen and J.~W. Cahn.
\newblock A microscopic theory for antiphase boundary motion and its
  application to antiphase domain coarsening.
\newblock {\em Acta metallurgica}, 27(6):1085--1095, 1979.

\bibitem{ambrosio2005gradient}
L.~Ambrosio, N.~Gigli, and G.~Savar{\'e}.
\newblock {\em Gradient flows: in metric spaces and in the space of probability
  measures}.
\newblock Springer Science \& Business Media, 2005.

\bibitem{AndersonMcFaddenWheeler}
D.~M. Anderson, G.~B. McFadden, and A.~A. Wheeler.
\newblock Diffuse-interface methods in fluid mechanics.
\newblock {\em Annual Review of Fluid Mechanics}, 30(1):139--165, 1998.

\bibitem{antoine2021scalar}
X.~Antoine, J.~Shen, and Q.~Tang.
\newblock Scalar auxiliary variable/lagrange multiplier based pseudospectral
  schemes for the dynamics of nonlinear schr{\"o}dinger/gross-pitaevskii
  equations.
\newblock {\em Journal of Computational Physics}, 437:110328, 2021.

\bibitem{baskaran2013convergence}
A.~Baskaran, J.~S. Lowengrub, C.~Wang, and S.~M. Wise.
\newblock Convergence analysis of a second order convex splitting scheme for
  the modified phase field crystal equation.
\newblock {\em SIAM Journal on Numerical Analysis}, 51(5):2851--2873, 2013.

\bibitem{bilbao2023explicit}
S.~Bilbao, M.~Ducceschi, and F.~Zama.
\newblock Explicit exactly energy-conserving methods for hamiltonian systems.
\newblock {\em Journal of Computational Physics}, 472:111697, 2023.

\bibitem{brenner2018robust}
S.~C. Brenner, A.~E. Diegel, and L.-Y. Sung.
\newblock A robust solver for a mixed finite element method for the
  cahn--hilliard equation.
\newblock {\em Journal of Scientific Computing}, 77:1234--1249, 2018.

\bibitem{cahn1959free}
J.~W. Cahn and J.~E. Hilliard.
\newblock Free energy of a nonuniform system. iii. nucleation in a
  two-component incompressible fluid.
\newblock {\em The Journal of chemical physics}, 31(3):688--699, 1959.

\bibitem{Cahn-Hilliard}
J.~W. Cahn and J.~E. Hilliard.
\newblock {Free Energy of a Nonuniform System. I. Interfacial Free Energy}.
\newblock {\em The Journal of Chemical Physics}, 28(2):258--267, 08 2004.

\bibitem{chen2019fast}
C.~Chen and X.~Yang.
\newblock Fast, provably unconditionally energy stable, and second-order
  accurate algorithms for the anisotropic cahn--hilliard model.
\newblock {\em Computer Methods in Applied Mechanics and Engineering},
  351:35--59, 2019.

\bibitem{chen2017second}
R.~Chen, X.~Yang, and H.~Zhang.
\newblock Second order, linear, and unconditionally energy stable schemes for a
  hydrodynamic model of smectic-a liquid crystals.
\newblock {\em SIAM Journal on Scientific Computing}, 39(6):A2808--A2833, 2017.

\bibitem{chizat2018global}
L.~Chizat and F.~Bach.
\newblock On the global convergence of gradient descent for over-parameterized
  models using optimal transport.
\newblock {\em Advances in neural information processing systems}, 31, 2018.

\bibitem{christlieb2014high}
A.~Christlieb, J.~Jones, K.~Promislow, B.~Wetton, and M.~Willoughby.
\newblock High accuracy solutions to energy gradient flows from material
  science models.
\newblock {\em Journal of computational physics}, 257:193--215, 2014.

\bibitem{cortes2006finite}
J.~Cort{\'e}s.
\newblock Finite-time convergent gradient flows with applications to network
  consensus.
\newblock {\em Automatica}, 42(11):1993--2000, 2006.

\bibitem{dai2016computational}
S.~Dai and Q.~Du.
\newblock Computational studies of coarsening rates for the cahn--hilliard
  equation with phase-dependent diffusion mobility.
\newblock {\em Journal of Computational Physics}, 310:85--108, 2016.

\bibitem{diegel2016stability}
A.~E. Diegel, C.~Wang, and S.~M. Wise.
\newblock Stability and convergence of a second-order mixed finite element
  method for the cahn--hilliard equation.
\newblock {\em IMA Journal of Numerical Analysis}, 36(4):1867--1897, 2016.

\bibitem{ding2022overparameterization}
Z.~Ding, S.~Chen, Q.~Li, and S.~J. Wright.
\newblock Overparameterization of deep resnet: zero loss and mean-field
  analysis.
\newblock {\em Journal of machine learning research}, 2022.

\bibitem{ericksen1976equilibrium}
J.~Ericksen.
\newblock Equilibrium theory of liquid crystals.
\newblock In {\em Advances in liquid crystals}, volume~2, pages 233--298.
  Elsevier, 1976.

\bibitem{fu2021linear}
G.~Fu and D.~Han.
\newblock A linear second-order in time unconditionally energy stable finite
  element scheme for a cahn--hilliard phase-field model for two-phase
  incompressible flow of variable densities.
\newblock {\em Computer Methods in Applied Mechanics and Engineering},
  387:114186, 2021.

\bibitem{fu2021linearly}
Y.~Fu, W.~Cai, and Y.~Wang.
\newblock A linearly implicit structure-preserving scheme for the fractional
  sine-gordon equation based on the ieq approach.
\newblock {\em Applied Numerical Mathematics}, 160:368--385, 2021.

\bibitem{gao2022unconditionally}
Y.~Gao, D.~Han, X.~He, and U.~R{\"u}de.
\newblock Unconditionally stable numerical methods for
  cahn-hilliard-navier-stokes-darcy system with different densities and
  viscosities.
\newblock {\em Journal of Computational Physics}, 454:110968, 2022.

\bibitem{glasner2016improving}
K.~Glasner and S.~Orizaga.
\newblock Improving the accuracy of convexity splitting methods for gradient
  flow equations.
\newblock {\em Journal of Computational Physics}, 315:52--64, 2016.

\bibitem{gudibanda2022convergence}
V.~M. Gudibanda, F.~Weber, and Y.~Yue.
\newblock Convergence analysis of a fully discrete energy-stable numerical
  scheme for the {Q}-tensor flow of liquid crystals.
\newblock {\em SIAM Journal on Numerical Analysis}, 60(4):2150--2181, 2022.

\bibitem{guillen2019unconditionally}
F.~Guill{\'e}n-Gonz{\'a}lez, M.~Rodr{\'\i}guez-Bellido, and D.~A.
  Rueda-G{\'o}mez.
\newblock Unconditionally energy stable fully discrete schemes for a
  chemo-repulsion model.
\newblock {\em Mathematics of Computation}, 88(319):2069--2099, 2019.

\bibitem{han2020second}
D.~Han and N.~Jiang.
\newblock A second order, linear, unconditionally stable,
  crank--nicolson--leapfrog scheme for phase field models of two-phase
  incompressible flows.
\newblock {\em Applied Mathematics Letters}, 108:106521, 2020.

\bibitem{jiang2019linearly}
C.~Jiang, W.~Cai, and Y.~Wang.
\newblock A linearly implicit and local energy-preserving scheme for the
  sine-gordon equation based on the invariant energy quadratization approach.
\newblock {\em Journal of Scientific Computing}, 80:1629--1655, 2019.

\bibitem{jiang2022improving}
M.~Jiang, Z.~Zhang, and J.~Zhao.
\newblock Improving the accuracy and consistency of the scalar auxiliary
  variable (sav) method with relaxation.
\newblock {\em Journal of Computational Physics}, 456:110954, 2022.

\bibitem{ju2015fast}
L.~Ju, J.~Zhang, and Q.~Du.
\newblock Fast and accurate algorithms for simulating coarsening dynamics of
  cahn--hilliard equations.
\newblock {\em Computational Materials Science}, 108:272--282, 2015.

\bibitem{LiLiu2004}
B.~Li and J.-g. Liu.
\newblock Epitaxial growth without slope selection: Energetics, coarsening, and
  dynamic scaling.
\newblock {\em Journal of Nonlinear Science}, 14:429--451, 01 2004.

\bibitem{liu2000approximation}
C.~Liu and N.~J. Walkington.
\newblock Approximation of liquid crystal flows.
\newblock {\em SIAM Journal on Numerical Analysis}, 37(3):725--741, 2000.

\bibitem{liu2019efficient}
Z.~Liu and X.~Li.
\newblock Efficient modified techniques of invariant energy quadratization
  approach for gradient flows.
\newblock {\em Applied Mathematics Letters}, 98:206--214, 2019.

\bibitem{liu2022step}
Z.~Liu and X.~Li.
\newblock Step-by-step solving schemes based on scalar auxiliary variable and
  invariant energy quadratization approaches for gradient flows.
\newblock {\em Numerical Algorithms}, pages 1--22, 2022.

\bibitem{lu2018relatively}
H.~Lu, R.~M. Freund, and Y.~Nesterov.
\newblock Relatively smooth convex optimization by first-order methods, and
  applications.
\newblock {\em SIAM Journal on Optimization}, 28(1):333--354, 2018.

\bibitem{maas2011gradient}
J.~Maas.
\newblock Gradient flows of the entropy for finite markov chains.
\newblock {\em Journal of Functional Analysis}, 261(8):2250--2292, 2011.

\bibitem{mielke2011weighted}
A.~Mielke and U.~Stefanelli.
\newblock Weighted energy-dissipation functionals for gradient flows.
\newblock {\em ESAIM: Control, Optimisation and Calculus of Variations},
  17(1):52--85, 2011.

\bibitem{santambrogio2017euclidean}
F.~Santambrogio.
\newblock {Euclidean, metric, and Wasserstein} gradient flows: an overview.
\newblock {\em Bulletin of Mathematical Sciences}, 7:87--154, 2017.

\bibitem{shen2012second}
J.~Shen, C.~Wang, X.~Wang, and S.~M. Wise.
\newblock Second-order convex splitting schemes for gradient flows with
  {Ehrlich-Schwoebel} type energy: application to thin film epitaxy.
\newblock {\em SIAM Journal on Numerical Analysis}, 50(1):105--125, 2012.

\bibitem{shen2018convergence}
J.~Shen and J.~Xu.
\newblock Convergence and error analysis for the scalar auxiliary variable
  (sav) schemes to gradient flows.
\newblock {\em SIAM Journal on Numerical Analysis}, 56(5):2895--2912, 2018.

\bibitem{shen2018scalar}
J.~Shen, J.~Xu, and J.~Yang.
\newblock The scalar auxiliary variable ({SAV}) approach for gradient flows.
\newblock {\em Journal of Computational Physics}, 353:407--416, 2018.

\bibitem{shen2019new}
J.~Shen, J.~Xu, and J.~Yang.
\newblock A new class of efficient and robust energy stable schemes for
  gradient flows.
\newblock {\em SIAM Review}, 61(3):474--506, 2019.

\bibitem{shen2020ieq}
J.~Shen and X.~Yang.
\newblock The {IEQ} and {SAV} approaches and their extensions for a class of
  highly nonlinear gradient flow systems.
\newblock {\em Contemp. Math.}, 754:217--245, 2020.

\bibitem{shin2017unconditionally}
J.~Shin, H.~G. Lee, and J.-Y. Lee.
\newblock Unconditionally stable methods for gradient flow using convex
  splitting {Runge-Kutta} scheme.
\newblock {\em Journal of Computational Physics}, 347:367--381, 2017.

\bibitem{szandala2021review}
T.~Szanda{\l}a.
\newblock Review and comparison of commonly used activation functions for deep
  neural networks.
\newblock {\em Bio-inspired neurocomputing}, pages 203--224, 2021.

\bibitem{vardoulakis1991gradient}
I.~Vardoulakis and E.~C. Aifantis.
\newblock A gradient flow theory of plasticity for granular materials.
\newblock {\em Acta mechanica}, 87(3-4):197--217, 1991.

\bibitem{yang2017efficient}
X.~Yang and L.~Ju.
\newblock Efficient linear schemes with unconditional energy stability for the
  phase field elastic bending energy model.
\newblock {\em Computer Methods in Applied Mechanics and Engineering},
  315:691--712, 2017.

\bibitem{yang2020convergence}
X.~Yang and G.-D. Zhang.
\newblock Convergence analysis for the invariant energy quadratization ({IEQ})
  schemes for solving the {Cahn-Hilliard} and {Allen-Cahn} equations with
  general nonlinear potential.
\newblock {\em Journal of scientific computing}, 82:1--28, 2020.

\bibitem{yang2017numerical}
X.~Yang, J.~Zhao, and Q.~Wang.
\newblock Numerical approximations for the molecular beam epitaxial growth
  model based on the invariant energy quadratization method.
\newblock {\em Journal of Computational Physics}, 333:104--127, 2017.

\bibitem{yang2017numerical2}
X.~Yang, J.~Zhao, Q.~Wang, and J.~Shen.
\newblock Numerical approximations for a three-component {Cahn-Hilliard}
  phase-field model based on the invariant energy quadratization method.
\newblock {\em Mathematical Models and Methods in Applied Sciences},
  27(11):1993--2030, 2017.

\bibitem{yue2023convergence}
Y.~Yue.
\newblock On the convergence of an unconditionally stable numerical scheme for
  the q-tensor flow based on the invariant quadratization method.
\newblock {\em Applied Mathematics Letters}, 138:108522, 2023.

\bibitem{zhao2021second}
J.~Zhao and D.~Han.
\newblock Second-order decoupled energy-stable schemes for
  cahn-hilliard-navier-stokes equations.
\newblock {\em Journal of Computational Physics}, 443:110536, 2021.

\bibitem{zhao2017numerical}
J.~Zhao, Q.~Wang, and X.~Yang.
\newblock Numerical approximations for a phase field dendritic crystal growth
  model based on the invariant energy quadratization approach.
\newblock {\em International Journal for Numerical Methods in Engineering},
  110(3):279--300, 2017.

\bibitem{zhao2017novel}
J.~Zhao, X.~Yang, Y.~Gong, and Q.~Wang.
\newblock A novel linear second order unconditionally energy stable scheme for
  a hydrodynamic q-tensor model of liquid crystals.
\newblock {\em Computer Methods in Applied Mechanics and Engineering},
  318:803--825, 2017.

\end{thebibliography}

\end{document}